\documentclass{amsart}

\usepackage[leqno]{amsmath}
\usepackage{amssymb}
\usepackage{amsthm}
\usepackage{stmaryrd}
\usepackage{enumerate}
\usepackage{url}
\usepackage{wasysym}
\usepackage{multicol}
\usepackage{color}
\usepackage[normalem]{ulem}
\usepackage{xfrac}
\usepackage{mathtools}
\usepackage[titletoc]{appendix}

\newtheorem{theorem}{Theorem}[section]
\newtheorem{fact}[theorem]{Fact}
\newtheorem{lemma}[theorem]{Lemma}
\newtheorem{corollary}[theorem]{Corollary}
\newtheorem{proposition}[theorem]{Proposition}

\theoremstyle{definition}
\newtheorem{definition}[theorem]{Definition}
\newtheorem{example}[theorem]{Example}

\newtheorem{remark}[theorem]{Remark}
\newtheorem{question}[theorem]{Question}

\newcommand{\abar}{\bar{a}}
\newcommand{\bbar}{\bar{b}}
\newcommand{\cbar}{\bar{c}}
\newcommand{\dbar}{\bar{d}}
\newcommand{\wbar}{\bar{w}}

\newcommand{\xbar}{\bar{x}}
\newcommand{\ybar}{\bar{y}}
\newcommand{\zbar}{\bar{z}}
\newcommand{\smd}{\raisebox{.75pt}{\textrm{\scriptsize{~\!$\triangle$\!~}}}}

\def\seq{\subseteq}
\def\feq{\textit{feq}}

\def\Def{\operatorname{Def}}

\def\tp{\operatorname{tp}}

\def\dfs{\operatorname{dfs}}

\def\fim{\operatorname{fim}}
\def\Av{\operatorname{Av}}
\def\Aut{\operatorname{Aut}}

\def\Th{\operatorname{Th}}
\def\N{\mathbb{N}}

\def\R{\mathbb{R}}

\def\cL{\mathcal{L}}
\def\cU{\mathcal{U}}
\def\cX{\mathcal{X}}
\def\cY{\mathcal{Y}}

\def\cB{\mathcal{B}}
\def\cP{\mathcal{P}}

\def\cF{\mathcal{F}}
\def\cH{\mathcal{H}}
\def\cI{\mathcal{I}}
\def\cS{\mathcal{S}}
\def\cQ{\mathcal{Q}}
\def\cL{\mathcal{L}}
\def\kM{\mathfrak{M}}
\def\seq{\subseteq}
\def\tp{\operatorname{tp}}
\def\Av{\operatorname{Av}}
\def\abar{\bar{a}}
\def\bbar{\bar{b}}
\def\xbar{\bar{x}}
\def\ybar{\bar{y}}
\newcommand{\miff}{\makebox[.4in]{$\Leftrightarrow$}}

\def\Def{\operatorname{Def}}
\def\feq{\textit{feq}}
\def\e{\epsilon}
\def\Lc{\mathcal{L}} 
\def\Uc{\mathcal{U}}
\def\cop{^{\textnormal{\textsf{c}}}}

\def\lp{(}
\def\rp{)}
\def\lb{[}
\def\rb{]}
\def\dfs{\textnormal{\textit{dfs}}}
\def\fam{\textnormal{\textit{fam}}}
\def\fim{\textnormal{\textit{fim}}}
\def\sqseq{\sqsubseteq}
\def\im{\text{im}}

\newcommand{\ABA}{\mathsf{ABA}}

\renewcommand{\phi}{\varphi}
\renewcommand{\emptyset}{\varnothing}
\renewcommand{\epsilon}{\varepsilon}
\newcommand{\Fraisse}{Fra\"{i}ss\'{e}}
\newcommand{\claim}{\hfill$\dashv_{\text{\scriptsize{claim}}}$}
\newcommand{\inv}{^{\text{-}1}}
\newcommand{\Me}{M_{\textnormal{\sfrac{1}{2}}}}
\newcommand{\Mi}{M^\infty_{\textnormal{\sfrac{1}{2}}}}
\newcommand{\Te}{T_{\textnormal{\sfrac{1}{2}}}}
\newcommand{\Ti}{T^\infty_{\textnormal{\sfrac{1}{2}}}}
\newcommand{\Avo}[4]{\Av^{#1}_{\!<{#2}}(#3,#4)}
\newcommand{\Avc}[4]{\Av^{#1}_{\!\leq{#2}}(#3,#4)}
\newcommand{\Avd}[4]{\Av^{#1}_{#2}(#3,#4)}

\def\Ind{\setbox0=\hbox{$x$}\kern\wd0\hbox to 0pt{\hss$\mid$\hss}
\lower.9\ht0\hbox to 0pt{\hss$\smile$\hss}\kern\wd0}

\def\Notind{\setbox0=\hbox{$x$}\kern\wd0\hbox to 0pt{\mathchardef
\nn=12854\hss$\nn$\kern1.4\wd0\hss}\hbox to
0pt{\hss$\mid$\hss}\lower.9\ht0 \hbox to 0pt{\hss$\smile$\hss}\kern\wd0}

\usepackage{graphicx}
\makeatletter
\providecommand{\bigsqcap}{%
  \mathop{%
    \mathpalette\@updown\bigsqcup
  }%
}
\newcommand*{\@updown}[2]{%
  \rotatebox[origin=c]{180}{$\m@th#1#2$}%
}
\makeatother
\newcommand{\sqin}{%
  \mathrel{\vphantom{\sqsubset}\text{%
    \mathsurround=0pt
    \ooalign{$\sqsubset$\cr$-$\cr}%
  }}%
}

\makeatletter
\DeclareRobustCommand{\cset}{\@ifstar\star@cset\normal@cset}
\newcommand{\star@cset}[1]{\left\llbracket#1\right\rrbracket}
\newcommand{\normal@cset}[2][]{\mathopen{#1\llbracket}#2\mathclose{#1\rrbracket}}
\makeatother

\makeatletter
\newcommand{\leqnomode}{\tagsleft@true\let\veqno\@@leqno}
\newcommand{\reqnomode}{\tagsleft@false\let\veqno\@@eqno}
\makeatother

\title{Keisler Measures in the Wild}

\date{March 13, 2023}
\author[G. Conant]{Gabriel Conant}

\address{Department of Pure Mathematics and Mathematical Statistics\\
University of Cambridge\\
Cambridge CB3 0WB\\
 UK}
\email{gconant@maths.cam.ac.uk}

\author[K. Gannon]{Kyle Gannon}

\address{Department of Mathematics\\
University of California\\
Los Angeles, CA, 90095, USA}

\email{gannon@math.ucla.edu}

\author[J. Hanson]{James Hanson}

\address{Department of Mathematics\\
University of Wisconsin \\
Madison, WI, 53706, USA}

\email{jehanson2@wisc.edu}

\thanks{GC was partially supported by NSF grant DMS-1855503. KG was partially supported by the NSF CAREER grant DMS-1651321.}

\begin{document}

\begin{abstract}
We investigate Keisler measures in arbitrary theories. Our initial focus is on Borel definability. We show that when working over countable parameter sets in countable theories, Borel definable measures are closed under Morley products and satisfy associativity. However, we also demonstrate failures of both properties over uncountable parameter sets. In particular, we show that the Morley product of Borel definable types need not be Borel definable (correcting an erroneous result from the literature). We then study various notions of generic stability for Keisler measures and generalize several results from the NIP setting to arbitrary theories. 
We also prove some positive results for the class of frequency interpretation measures in arbitrary theories, namely, that such measures are closed under convex combinations and commute with all Borel definable measures.
Finally, we construct the first example of a complete type which is definable and finitely satisfiable in a small model, but not finitely approximated over any small model. 
\end{abstract}

\maketitle

\setcounter{tocdepth}{1}
\tableofcontents

\section*{Introduction}
Finitely additive probability measures on definable sets were originally introduced by Keisler as a tool to study forking in NIP theories \cite{Keis}. Since then, Keisler measures have found extensive connections to various contexts in both pure and applied model theory. They played a pivotal role in resolving the Pillay conjectures on definably compact groups definable in o-minimal theories \cite{HPS} and, more generally, are a crucial tool in the study of definably amenable groups definable in NIP theories \cite{ChSi, OnPi}. Keisler measures are a vital component in the interplay between model theory and combinatorics (especially in connection to the regularity theorems) \cite{ChStNIP, CPTNIP}. Additionally, Keisler measures also arise naturally in continuous logic as types over models of the randomization \cite{BYK}.

  Despite this plethora of research, there is an obvious gap in the existing results: A clear structural understanding of Keisler measures only exists in NIP theories, and, with a few specialized exceptions, the deepest results concerning Keisler measures exist in that context. This paper lays the foundational groundwork for the study of Keisler measures outside the boundary of NIP. Our results demonstrate that the \textit{general theory of Keisler measures} is fundamentally more complicated than previously thought. Broadly speaking, whereas Keisler measures in NIP theories can be sufficiently approximated by types, and so are tame, measures in arbitrary theories are far more sensitive to analytic and descriptive set theoretic issues, and it is no longer possible to directly generalize proofs from types to measures. Indeed, we will develop several examples demonstrating novel and exotic behavior of arbitrary Keisler measures outside of NIP. However, we also prove positive results concerning the theory of `generically stable' Keisler measures, which further demonstrate that structural understanding is possible.

In arbitrary theories, invariant types can be `freely' amalgamated using the Morley product operation, sometimes also called the nonforking product. In NIP theories, this operation extends automatically to invariant Keisler measures, thanks to the result of Hrushovski and Pillay \cite{HP} that any such measure in an NIP theory is \emph{Borel definable}. While this need not hold outside of NIP, one can still define the Morley product of measures that are Borel definable. 
Thus the first main goal of this paper is to establish basic properties of Borel definable Keisler measures in arbitrary theories. We consider the two questions of whether Borel definability is preserved by Morley products, and whether the Morley product is associative for Borel definable measures. Despite the fundamental nature of these questions,  the previous literature has been somewhat vague regarding the answers.  A positive answer to the first question is stated without proof in \cite[Lemma 1.6]{HPS}. Moreover, while associativity of the Morley product seems to be tacitly assumed in various places, it is only directly addressed in \cite{Sibook} under the assumption of NIP (and, even in this case, a complete proof of associativity was given only recently;
see Remark \ref{rem:NIPnice}). The goal of Sections \ref{sec:BD} and \ref{sec:badBD} is to clarify this situation. In Section \ref{sec:BD}, we show that if $T$ is any countable theory, and $A\subset \cU$ is countable, then the set of measures that are Borel definable over $A$ is closed under Morley products, and associativity holds for such measures. On the other hand, we will see in Section \ref{sec:badBD} that both properties can fail without the extra countability assumptions (even in simple theories). In particular, this refutes the unproven claim in \cite{HPS}. 

The rest of the paper is devoted to developing various notions of `generic stability' for Keisler measures in arbitrary theories. We focus on three classes: measures that are \emph{definable and finitely satisfiable} in a small model (or \dfs), measures that are \emph{finitely approximated} in a small model (or \fam), and measures that are \emph{frequency interpretation measures} with respect to a small model (or \fim). Section \ref{sec:fimfam} provides definitions and a review of basic facts about \fim, \fam, and \dfs\ measures. 

In NIP theories, the three classes of measures described above coincide, and a Keisler measure with these properties is called `generically stable'. Outside of NIP theories, these properties are no longer equivalent, and thus one obtains three competing notions of generic stability for measures. An assessment of this competition was undertaken in \cite{CoGa}, mostly focusing on types. The present article continues this work with a greater emphasis on measures. A recurring question is the extent to which fundamental results on generically stable Keisler measures in NIP theories can be generalized to arbitrary theories, and we will prove several results to this effect. These results help to clarify when NIP is playing a crucial role in a given result about measures, versus when a similar result can be obtained in general, perhaps after some appropriate modification of the working assumptions. In particular, a fundamental fact about NIP theories is that any Keisler measure can be locally approximated by types. This result is often used to replace measures by types in various arguments, and thus avoid the necessity of pure measure theory and integration techniques. On the other hand, our work will show that generalizations of certain results on NIP theories can indeed be obtained using more measure-theoretic proofs which, although possibly more complicated methodologically, are also shorter and in some cases more concise.

In Section \ref{sec:com}, we focus on the question of commutativity for the Morley product of Borel definable Keisler measures. This is motivated by the result of Hrushovski, Pillay, and Simon \cite{HPS} that, in NIP theories, definable measures commute with finitely satisfiable measures and, moreover,  \dfs\  measures commute with arbitrary invariant measures. The goal of Section \ref{sec:com} is to obtain suitable generalizations of these results for arbitrary theories. We first show that in any theory, if $\mu$ is a definable measure, and $\nu$ is Borel definable and finitely satisfiable, then $\mu$ and $\nu$ commute provided that for any small model $M$, $\mu|_M$ has some definable global extension that commutes with $\nu$ (see Theorem \ref{thm:gencom}). This recovers the corresponding fact from \cite{HPS} since, in NIP theories, any measure over a small model has a \emph{smooth} global extension, and it is easy to show that smooth measures (in any theory) are definable and commute with all Borel definable measures. We also show later in the paper that, in Theorem \ref{thm:gencom}, the extra assumption on restrictions of $\mu$ to small models (which is automatic in NIP theories) is necessary. In particular, we construct a theory with a \dfs\ type and a definable measure that do not commute (see Proposition \ref{prop:nocom}). Finally in Section \ref{sec:com}, we show that in any theory, \fim\ measures commute with Borel definable measures (see Theorem \ref{thm:fimcom}). In other words, the corresponding result for NIP theories from \cite{HPS} generalizes to arbitrary theories, provided one replaces \dfs\ with \fim. 

In Section \ref{sec:fim}, we focus on further properties of \fim\ measures. Evidence suggests that \fim\ is the `right' notion of generic stability for measures in arbitrary theories. In particular, the notion of a generically stable \emph{type} is well established in the literature, and the first two authors showed in \cite{CoGa} that this notion coincides with \fim\ when viewing types as $\{0,1\}$-valued measures.
In Theorem \ref{thm:fimCC}, we show that \fim\ measures are closed under convex combinations. The analogous result for \dfs\ and \fam\ measures is quite easy to prove (see Proposition \ref{prop:CCbasic}) and so, in light of \cite{HPS}, Theorem \ref{thm:fimCC} again generalizes known facts from the study of NIP theories. However, we will see that working directly with \fim\ measures in general theories leads to significantly more complicated proofs. We finish Section \ref{sec:fim} with a discussion of the still open question of whether \fim\ measures are preserved by Morley products (an earlier draft of this article contained an erroneous proof of a positive answer).

In Section \ref{sec:dfsfam}, we answer one of the main questions left open in \cite{CoGa}, which is on the existence of a complete global measure that is \dfs\ and not \fam\ (an example involving a local type was given in \cite{CoGa}). We will first give a new local example of this phenomenon, which is built using subsets of the interval $[0,1]$ of Lebesgue measure $\frac{1}{2}$. Then we develop this example into a more complicated theory with a complete \dfs\ type that is not \fam. 

Section \ref{sec:famfim} focuses on examples of measures that are \fam\ and not \fim. We first show that a purported example from \cite{ACPgs} of this phenomenon does not work. Then we revisit a different example from \cite{CoGa} in the theory of the generic $K_s$-free graph (for fixed $s\geq 3$). We develop further properties of this example, and correct an erroneous proof from \cite{CoGa}. Finally, we give a new example of a complete type that is \fam\ and not \fim, which is obtained by taking a certain reduct of the \dfs\ and non-\fam\ type from Section \ref{sec:dfsfam}.

\subsection*{Corrigendum} 

For the sake of clarifying the literature, we summarize the incorrect results and proofs from previous work that are addressed in this article.

\begin{enumerate}[$(1)$]
\item We recall that the product of two Borel definable Keisler measures is Borel definable in the NIP setting. In \cite[Lemma 1.6]{HPS}, it is claimed, but not proved, that the Morley product of two Borel definable Keisler measures is Borel definable. We show here that this is not always true, even for Borel definable \emph{types} (see Proposition \ref{prop:ternary1}).

\item Example 1.7 of \cite{ACPgs} describes a complete theory that is claimed to admit a global generically stable type $p$ such that $p\otimes p$ is not generically stable. This claim is repeated in \cite[Fact 5.4]{CoGa}. It turns out that $p$ is not well-defined, and we show here that this particular theory has no global non-algebraic generically stable types (see Theorem \ref{thm:Tfeq2triv}). 
The question of whether Morley products preserve generic stability remains open. See the end of Section \ref{sec:fim} for further discussion.
\item Remark 4.2 of \cite{CoGa} makes an unjustified claim that \dfs, \fam, and \fim\ measures are closed under localization at arbitrary Borel sets, which seems likely to be false. In Section \ref{sec:famfim}, we supply correct proofs of the results in \cite{CoGa} that used this remark (see Proposition \ref{prop:dfstriv1}, Remark \ref{rem:CGcor}, and Theorem \ref{thm:Hfimtriv}).
\end{enumerate}

\section{Basic definitions and notation}
We start with some general notation that will be used throughout the paper. Let $X$ be a set. Given a point $a\in X$, we let $\delta_a$ denote the Dirac measure on $X$ concentrating at $a$. For $\abar\in X^n$, we let $\Av(\abar)$ denote the `average' measure $\frac{1}{n}\sum_{i=1}^n\delta_{a_i}$. Given a (bounded) real-valued function $f$ on $X$, define $\|f\|_\infty\coloneqq \sup_{x\in X}|f(x)|$. 

Given $r,s\in \R$ and $\epsilon>0$, we write $r\approx_\epsilon s$ to denote $|r-s|<\epsilon$. Given an integer $n\geq 1$, let $[n]=\{1,\ldots,n\}$.

Now let $T$ be a complete $\cL$-theory with monster model $\cU$. We work with formulas in the language $\cL$ with parameters from $\cU$. A formula $\phi(x)$ is \textbf{over $A\seq\cU$} if all parameters in $\phi(x)$ come from $A$. In this case we say $\phi(x)$ is an \textbf{$\cL_A$-formula}.  An \textbf{$\cL$-formula} is a formula without parameters. We will use $x,y,z$, etc., to denote tuples of variables, although at times we may also employ vector notation $\xbar,\ybar,\zbar$, etc., for clarity. As usual, we often partition the free variables in a formula $\phi(x,y)$ into object variables $x$ and parameter variables $y$.  

Given $A\seq \cU$, let $\Def_x(A)$ denote the Boolean algebra of $\cL_A$-formulas with free variables $x$, up to equivalence modulo $T$ expanded by constants for $A$. The corresponding Stone space of types is denoted $S_x(A)$. Given an $\cL_A$-formula $\phi(x)$, we let $[\phi(x)]$ denote the clopen set of types in $S_x(A)$ containing $\phi(x)$.

We let $\kM_x(A)$ denote the space of Keisler measures (i.e., finitely additive probability measures) on $\Def_x(A)$.  
Recall that any $\mu\in\kM_x(A)$ determines a  unique regular Borel probability measure $\widetilde{\mu}$ on $S_x(A)$ such that if $\phi(x)$ is an $\cL_A$-formula then $\mu(\phi(x))=\widetilde{\mu}([\phi(x)])$ (and, furthermore, any regular Borel probability measure on $S_x(A)$ is of this form). See \cite[Section 7.1]{Sibook} for an explicit construction of $\widetilde{\mu}$. By identifying $\mu$ and $\widetilde{\mu}$, we can view $\mu$  as a regular Borel probability measure on $S_x(A)$. For further details on Borel measures and regularity, see Section \ref{app:Borel} of the appendix. We will use the following special case of Fact \ref{fact:BStone}.
 
 \begin{fact}\label{fact:basicBorel}
 Fix $A\seq \cU$ and $\mu\in\kM_x(A)$.
 \begin{enumerate}[$(a)$]
 \item If $U\seq S_x(A)$ is open then 
 \[
 \mu(U)=\sup\{\mu(\phi(x)):\text{$\phi(x)$ is an $\cL_A$-formula and $[\phi(x)]\seq U$}\}.
 \]
 \item If $\nu$ is a regular Borel probability measure on $S_x(A)$, and $\nu(\phi(x))=\mu(\phi(x))$ for any $\cL_A$-formula $\phi(x)$, then $\mu=\nu$.
\end{enumerate}
\end{fact}

Given $A\seq B\seq \cU$ and a tuple $x$ of variables, let $\rho^x_{B,A}\colon S_x(B)\to S_x(A)$ denote the restriction map. Note that $\rho^x_{B,A}$  is a continuous surjective map between compact Hausdorff spaces, and thus is a quotient map. 
Let $\rho^x_A$ denote $\rho^x_{\cU,A}$. Given $\mu\in\kM_x(\cU)$ and $A\seq \cU$, we let $\mu|_A$ denote the restriction of $\mu$ to $\Def_x(A)$.

\begin{remark}\label{rem:push}
If $\mu\in\kM_x(\cU)$ and $A\seq\cU$ then,
as a regular Borel measure on $S_x(A)$, $\mu|_A$ is the pushforward of $\mu$ to $S_x(A)$ along $\rho^x_A$. In other words, if $X\seq S_x(A)$ is Borel then $\mu|_A(X)=\mu((\rho^x_A)\inv (X))$. Indeed, by definition of $\mu|_A$, this holds when $X=[\varphi(x)]$  for some $\cL_A$-formula $\phi(x)$. Thus it holds for all Borel $X$ by Fact \ref{fact:basicBorel}$(b)$, and since pushforwards preserve regularity in this context (see Fact \ref{fact:push}). 
\end{remark}

We write $A\subset\cU$ to denote that $A$ is a subset of $\cU$ which is \textbf{small}, i.e., $\cU$ is $|A|^+$-saturated and strongly $|A|^+$-homogeneous. A measure $\mu\in\kM_x(\cU)$ is \textbf{invariant} if there is some $A\subset \cU$ such that for any $\cL$-formula $\phi(x,y)$, if $b,b'\in\cU^y$ have the same type over $A$, then $\mu(\phi(x,b))=\mu(\phi(x,b'))$. In this case, we also say that $\mu$ is \textbf{invariant over $A$} or \textbf{$A$-invariant}.

Suppose $\mu\in\kM_x(\cU)$ is invariant over $A\subset\cU$. Given an $\cL_A$-formula $\phi(x,y)$, define $F^\phi_{\mu,A}\colon S_y(A)\to [0,1]$ such that $F^\phi_{\mu,A}(q)=\mu(\phi(x,b))$ for some/any $b\models q$.  Note that if $B\supseteq A$ then $\mu$ is invariant over $B$ and, if $\phi(x,y)$ is an $\cL_A$-formula, then $F^\phi_{\mu,B}=F^\phi_{\mu,A}\circ \rho^y_{B,A}$. 

A Keisler measure $\mu\in\kM_x(\cU)$ is \textbf{Borel definable} if there is some $A\subset \cU$ such that $\mu$ is $A$-invariant and $F^{\phi}_{\mu,A}$ is a Borel map for any $\cL$-formula $\phi(x,y)$. In this case, we also say that $\mu$ is \textbf{Borel definable over $A$}. Note that if $\mu$ is Borel definable over $A$, then $F^\phi_{\mu,A}$ is Borel for any $\cL_A$-formula $\phi(x,y)$ and, moreover, $\mu$ is Borel definable over any $B\supseteq A$ (see also \cite[Proposition 2.22]{Ganthesis}).

Finally, we define the Morley product of Keisler measures. Given a Borel definable measure $\mu\in\kM_x(\cU)$ and a measure $\nu\in\kM_y(\cU)$, we define a measure $\mu\otimes\nu$ in $\kM_{xy}(\cU)$ such that, given an $\cL_{\cU}$-formula $\phi(x,y)$,
\[
(\mu\otimes\nu)(\phi(x,y))=\int_{S_y(A)}F^\phi_{\mu,A}\,d\nu|_A,
\]
where $A\subset \cU$ is any small set such that $\phi(x,y)$ is over $A$ and $\mu$ is Borel definable over $A$.  One can show that this does not depend on the choice of $A$.   The measure $\mu\otimes\nu$ is called the \textbf{Morley product} of $\mu$ and $\nu$.

\begin{remark}
To help ease notation, we will write integrals $\int_{S_y(A)}f\, d\nu|_A$ simply as $\int_{S_y(A)}f\, d\nu$. In other words, the fact that we integrate with respect to $\nu|_A$ is implied by the domain of integration $S_y(A)$. When $f$ is (or involves) a function of the form $F^\phi_{\mu,A}$, we write $\int_{S_y(A)}F^\phi_\mu\, d\nu$ instead of $\int_{S_y(A)}F^\phi_{\mu,A}\, d\nu$. 
\end{remark}

Recall that $S_x(\cU)$ can be identified with a closed subset of $\kM_x(\cU)$ by viewing types as $\{0,1\}$-valued measures. 
If $q\in S_y(\cU)$ is a type, then we have a well-defined Morley product $\mu\otimes q$ for any \emph{invariant} $\mu\in\kM_x(\cU)$ since any function is integrable with respect to $q$ as a Dirac measure. More explicitly,  $(\mu\otimes q)(\phi(x,y))=\mu(\phi(x,b))$, where $\mu$ is $A$-invariant, $\phi(x,y)$ is over $A$, and $b\models q|_A$. If $\mu$ is a type $p\in S_x(\cU)$, then $p\otimes q$ is a type in $S_{xy}(\cU)$, and $\phi(x,y)\in p\otimes q$ if and only if $\phi(x,b)\in p$ (where $b$ is as before).  We recall the following easy exercise.

\begin{fact}\label{fact:basicinv}
Suppose $\mu\in\kM_x(\cU)$ and $\nu\in\kM_y(\cU)$ are invariant. If $\mu$ is Borel definable, or if $\nu$ is a type, then $\mu\otimes\nu$ is invariant.
\end{fact}

\section{Borel definability over countable sets}\label{sec:BD}

As explained in the introduction, one main goal of this paper is to settle the question of whether the Morley product of Keisler measures preserves Borel definability, and also to address associativity. In this section, we show that both properties hold when working over countable parameter sets in countable theories. We will approach this result from a general perspective that will lead to further facts about Borel definable measures, and also explain precisely how the situation turns complicated (and counterintuitive) over uncountable sets. This perspective will also lead to some useful conclusions for definable measures (see Section \ref{sec:definable}).

\subsection{Fiber functions over Borel sets}
Recall that if $\mu\in\kM_x(\cU)$ is Borel definable over $A\subset\cU$, then it is Borel definable over any $B\supseteq A$. We now observe that Borel definability can also be dropped to smaller parameter sets, provided one still has invariance. The proof uses a result from \cite{HolSpur}, which can be viewed as a  Borel variation on the universal property of quotient maps.

\begin{theorem}[Holick\'{y} \& Spurn\'{y} \cite{HolSpur}]\label{thm:HolSpur}
Suppose $\rho\colon X\to Y$ is a surjective continuous map between compact Hausdorff spaces. Then for any $E\seq Y$, if $\rho\inv(E)$ is Borel then $E$ is Borel. Therefore, if $f\colon Y\to Z$ is a map to a topological space $Z$, and $f\circ\rho$ is Borel, then $f$ is Borel. 
\end{theorem}
\begin{proof}
The first claim is a special case of \cite[Theorem 10]{HolSpur}. The second claim follows from the first. Indeed, if $U\seq Z$ is open and $f\circ\rho$ is Borel, then $\rho\inv(f\inv (U))$ is a Borel set, and thus so is $f\inv(U)$.
\end{proof}

\begin{corollary}\label{cor:BDdrop}
Suppose $\mu\in\kM_x(\cU)$ is Borel definable, and invariant over $A\subset\cU$. Then $\mu$ is Borel definable over $A$. 
\end{corollary}
\begin{proof}
This follows from Theorem \ref{thm:HolSpur} since if $\mu$ is Borel definable over $B\supseteq A$ then, for any $\cL$-formula $\phi(x,y)$, $F^\phi_{\mu,B}=F^\phi_{\mu,A}\circ\rho^y_{B,A}$. 
\end{proof}

\begin{remark}\label{rem:BDdrop}
Despite the simplicity of the proof, Corollary \ref{cor:BDdrop} does not appear in previous literature, possibly due to the use of \cite{HolSpur}. On the other hand,  the analogue of this corollary for \emph{definable} measures (which are discussed in Section \ref{sec:definable}) is well known and follows from a similar proof. Indeed, if $\mu$ is definable over $B$ and invariant over $A\seq B$, then $F^\phi_{\mu,B}$ is continuous and thus $F^\phi_{\mu,A}$ is continuous by the universal property of quotient maps. It follows that $\mu$ is definable over $A$. 
\end{remark}

Our next goal is to redefine $F^\phi_{\mu,A}$ with an arbitrary Borel set $W(x,y)\seq S_{xy}(A)$ in place of $\phi(x,y)$.  The underlying idea is quite natural. We will `plug in' a parameter $b$ for the $y$ variables, and apply the measure $\mu$. This perspective of treating Borel sets like formulas crops up in the literature, though often informally. We will see that while some techniques pass from formulas to Borel sets without any issues, there are certain places where things can go wrong. These subtleties will eventually lead to examples where Borel definable measures fail to be closed under Morley products, and where associativity of the Morley product fails. For this reason, we will proceed carefully with the next few definitions and basic observations, so as to ensure a solid foundation for the passage from formulas to Borel sets. 

\begin{definition}
Given $A\subset\cU$, we say that a set $W\seq  S_x(\cU)$ is \textbf{$\rho^x_A$-invariant} if membership in $W$ depends only on $\rho^x_A$, i.e., $W=(\rho^x_A)\inv(\rho^x_A(W))$. 
\end{definition}

\begin{remark}\label{rem:pushBorel}
If $W\seq S_x(\cU)$ is Borel and $\rho^x_A$-invariant for some $A\subset\cU$, then $\rho^x_A(W)$ is a Borel set in $S_x(A)$ by Theorem \ref{thm:HolSpur}.
\end{remark}

\begin{definition}
Suppose $A\subset\cU$ and $W\seq S_{xy}(A)$. Given $b\in \cU^y$, we define
\[
W(x,b)=\{p\in S_x(\cU):\tp(a,b/A)\in W\text{ for some/any }a\models p|_{Ab}\}.
\]
Note that $W(x,b)$ is $\rho^x_{Ab}$-invariant.
\end{definition}

\begin{lemma}\label{lem:Borel-fiber}
Suppose $A\subset\cU$ and $W\seq S_{xy}(A)$ is Borel.
\begin{enumerate}[$(a)$]
\item If  $b\in \cU^y$ then $W(x,b)$ is a Borel subset of $S_x(\cU)$.
\item If $\mu\in\kM_x(\cU)$ is $A$-invariant, and $b,b'\in\cU^y$ with $b\equiv_A b'$, then $\mu(W(x,b))=\mu(W(x,b'))$.
\end{enumerate}
\end{lemma}
\begin{proof}
Both parts can be proved directly by induction on the $\boldsymbol{\Sigma}$-complexity of $W$, using only elementary steps. We will sketch alternative `high-level' arguments. For part $(a)$, fix $b\in\cU^y$ and let $X=\{q\in S_{xy}(A):q(x,b)\text{ is consistent}\}$, which is closed in $S_{xy}(A)$. Set $\tau\colon X\to S_x(Ab)$ such that $\tau(q)=q(x,b)$. Then $\tau$ is surjective and continuous, and $\tau\inv(\tau(W\cap X))=W\cap X$. So $\tau(W\cap X)$ is Borel by Theorem \ref{thm:HolSpur}. Thus $W(x,b)=(\rho^x_{Ab})\inv(\tau(W\cap X))$ is Borel.

For part $(b)$, suppose we have $b,b'\in\cU^y$ and $\sigma\in\Aut(\cU/A)$ such that $\sigma(b)=b'$. Then $\sigma$ induces a homeomorphism of $S_x(\cU)$, which yields a regular Borel measure $\nu\coloneqq \mu\sigma$ on $S_x(\cU)$. Since $\mu$ is $A$-invariant, it agrees with $\nu$ on clopen sets, and thus $\mu=\nu$ by Fact \ref{fact:basicBorel}$(b)$. So $\mu(W(x,b'))=\nu(W(x,b))=\mu(W(x,b))$. 
\end{proof}

\begin{definition}
Suppose $\mu\in\kM_x(\cU)$ is invariant over $A\subset\cU$, and $W\seq S_{xy}(A)$ is Borel. Define $F^W_{\mu,A}\colon S_y(A)\to [0,1]$ such that $F^W_{\mu,A}(q)=\mu(W(x,b))$ where $b\models q$.
\end{definition}

Note that $F^W_{\mu,A}$ is well-defined by Lemma \ref{lem:Borel-fiber}. We also note that if $W$ is the \emph{clopen} set determined by some $\cL_A$-formula $\phi(x,y)$, then $F^W_{\mu,A}$ coincides with $F^\phi_{\mu,A}$. 

\subsection{Products and associativity}

In this subsection, we formulate some ad hoc conditions on Borel definable measures that allow one to prove preservation under Morley products and associativity. In the next subsection, we will see that these conditions hold over countable sets. We start with some motivation.

Consider a measure $\mu\in\kM_x(\cU)$ that is Borel definable over some $A\subset \cU$. Then maps of the form $F^W_{\mu,A}$ are Borel for any \emph{clopen} set $W\seq S_{xy}(A)$. But we will eventually see that this is not enough to ensure $F^W_{\mu,A}$ is Borel for general Borel sets $W$. To obtain this, one needs to further assume that $F^U_{\mu,A}$ is Borel for any \emph{open} $U\seq S_{xy}(A)$ (see Lemma \ref{lem:SBDall}). We will show that this assumption suffices to address preservation of Borel definability in Morley products. The issue of associativity, however, requires consideration of further subtleties. In particular, with $\mu$ as above, suppose we have some fixed open set $U\seq S_{xy}(A)$ such that $F^U_{\mu,A}$ is Borel. Then given some $\nu\in\kM_y(\cU)$, we have a well-defined integral $\int_{S_y(A)}F^U_\mu\,d\nu$. On other hand, Fact \ref{fact:basicBorel}$(a)$ gives an explicit expression for $(\mu\otimes\nu)|_A(U)$ which, as we will see in later examples, need not be the same as the previous integral (note that if $U$ is clopen then we do have such an equality by definition of the Morley product). Altogether, this discussion motivates the following definition.

\begin{definition}
Suppose $\mu\in\kM_x(\cU)$ is invariant over $A\subset\cU$.  We say $\mu$ is \textbf{BD$^{\text{+}}$} over $A$ if, for any $y$ and any open $U\seq S_{xy}(A)$,  the map $F^U_{\mu,A}$ is Borel. Moreover, we say $\mu$ is \textbf{BD$^{\text{++}}$} over $A$ if it is BD$^{+}$ over $A$ and, for any $y$, any open $U\seq S_{xy}(A)$, and any $\nu\in\kM_y(\cU)$, we have  $(\mu\otimes\nu)|_A(U)=\int_{S_y(A)}F^U_{\mu}\,d\nu$. 
\end{definition}

Next we show that the defining properties of BD$^{+}$ and BD$^{++}$ extend automatically from open sets to arbitrary Borel sets. For BD$^{+}$, this boils down to the fact that pointwise limits of Borel functions are Borel. For BD$^{++}$ we will apply the Dominated Convergence Theorem \cite[Theorem 2.4.5]{Cohn}.

\begin{lemma}\label{lem:SBDall}
Suppose $\mu\in\kM_x(\cU)$ is BD$^{+}$ over $A\subset\cU$, and $W\seq S_{xy}(A)$ is Borel. Then $F^W_{\mu,A}$ is Borel. Moreover, if $\mu$ is BD$^{++}$ over $A$ then, for any $\nu\in\kM_y(\cU)$, we have $(\mu\otimes\nu)|_A(W)=\int_{S_y(A)}F^W_{\mu}\, d\nu$.
\end{lemma}
\begin{proof}
We proceed by induction on the $\boldsymbol{\Sigma}$-complexity of $W$. The base case when $W$ is open holds by assumption. So fix $1<\alpha<\omega_1$ and assume the result for $\boldsymbol{\Sigma}^0_\beta$ subsets of $S_{xy}(A)$ for all $\beta<\alpha$. Suppose $W$ is a $\boldsymbol{\Sigma}^0_{\alpha}$ set. Then, for $i<\omega$, we have $W_i\in\boldsymbol{\Sigma}^0_{\alpha_i}$ for some $\alpha_i<\alpha$, such that $W=\bigcup_{i<\omega}\neg W_i$. Since each $\boldsymbol{\Pi}^0_\beta$ class is closed under finite unions (see \cite[Theorem 2.1]{MillerDST}), we may assume without loss of generality that $\neg W_i\seq\neg W_{i+1}$ for all $i<\omega$. Given $q\in S_y(A)$ and $b\models q$, we have
\begin{multline*}
F^B_{\mu,A}(q)=\mu(W(x,b))=\lim_{i\to\infty}\mu((\neg W_i)(x,b))\\
=\lim_{i\to\infty}(1-\mu(W_i(x,b)))=\lim_{i\to\infty}(1-F^{W_i}_{\mu,A}(q)).
\end{multline*}
By induction, $F^W_{\mu,A}$ is a pointwise limit  of Borel functions, and thus is Borel. Moreover, if $\mu$ is BD$^{++}$ over $A$, then, by induction and the Dominated Convergence Theorem, we have
\begin{multline*}
(\mu\otimes\nu)|_A(W)=\lim_{i\to\infty}(\mu\otimes\nu)|_A(\neg W_i)=1-\lim_{i\to\infty}\int_{S_y(A)}F_{\mu}^{W_i}\, d\nu\\
=\int_{S_y(A)}\lim_{i\to\infty}(1-F_{\mu}^{W_i})\, d\nu=\int_{S_y(A)}F^W_{\mu}\, d\nu.\qedhere
\end{multline*}
\end{proof}

We can now prove the main result concerning BD$^{+}$ and BD$^{++}$.

\begin{theorem}\label{thm:SBDmain}
Suppose $\mu\in\kM_x(\cU)$ is Borel definable over $A\subset\cU$. 
\begin{enumerate}[$(a)$]
\item  If $\nu\in\kM_y(\cU)$ is BD$^{+}$ over $A$, then $\mu\otimes\nu$ is Borel definable over $A$. 
\item If $\nu\in\kM_y(\cU)$ is BD$^{++}$ over $A$, then $((\mu\otimes\nu)\otimes\lambda)|_A=(\mu\otimes(\nu\otimes\lambda))|_A$ for any $\lambda\in\kM_z(\cU)$.
\end{enumerate}
\end{theorem}
\begin{proof}
Before proving the two statements, we will develop some preliminaries.
Fix an $\cL_A$-formula $\phi(x,y,z)$. Then $F^\phi_{\mu,A}\colon S_{yz}(A)\to[0,1]$ is Borel, and so there is a sequence $(f_n)_{n=0}^\infty$ of simple Borel functions on $S_{yz}(A)$ converging pointwise to $F^\phi_{\mu,A}$. (See Fact \ref{fact:Bapprox}; in fact, this convergence can be made uniform, but we will work with pointwise convergence in preparation for Remark \ref{rem:Baire}.) For $n\geq 0$, write $f_n=\sum_{i=1}^{m_n}\alpha_{n,i}\boldsymbol{1}_{W_{n,i}}$ where $W_{n,i}\seq S_{yz}(A)$ is Borel and $\alpha_{n,i}\in [0,1]$. 

Given $c\in\cU^z$, we set the following notation.  Let $W^c_{n,i}\coloneqq \rho^y_{Ac}(W_{n,i}(y,c))$, which is a Borel subset of $S_y(Ac)$ by Remark \ref{rem:pushBorel}. Define the map $f^c_n=\sum_{i=1}^{m_n}\alpha_{n,i}\boldsymbol{1}_{W^c_{n,i}}$ on $S_y(Ac)$. Finally, let $\phi_c(x,y)$ denote $\phi(x,y,c)$. \medskip

\noindent\textit{Claim 1:} Fix $c\in\cU^z$. Then $(f^c_n)_{n=0}^\infty$ converges pointwise to  $F^{\phi_c}_{\mu,Ac}$. 

\noindent\textit{Proof:} Fix $q\in S_y(Ac)$, and let  $s=\tp(b,c/A)$ where $b\models q$. Then for any $i\le m_n$, we have $s\in W_{n,i}$ if and only if $q\in W^c_{n,i}$. It follows that $f_n(s)=f^c_n(q)$ for any $n\geq 0$. Therefore
\reqnomode
\begin{equation*}
F^{\phi_c}_{\mu,Ac}(q)=\mu(\phi(x,b,c))=F^\phi_{\mu,A}(s)=\lim_{n\to\infty} f_n(s)=\lim_{n\to\infty} f^c_n(q).\tag*{\claim}
\end{equation*}

Now fix some $A$-invariant measure $\nu\in\kM_y(\cU)$. Given $n\geq 0$, define the function $h_n=\sum_{i=1}^{m_n}\alpha_{n,i}F^{W_{n,i}}_{\nu,A}$ on $S_z(A)$. 
\medskip

\noindent\textit{Claim 2:} $(h_n)_{n=0}^\infty$ converges pointwise to $F^{\phi}_{\mu\otimes\nu,A}$.  

\noindent\textit{Proof:}  Fix $r\in S_z(A)$, and let $c\models r$. Then $\nu|_{Ac}(W^c_{n,i})=F^{W_{n,i}}_{\nu,A}(r)$ for any $n\geq 0$ and $i\leq m_n$. So for any $n\geq 0$, we have
\[
h_n(r)=\sum_{i=1}^{m_n}\alpha_{n,i} F^{W_{n,i}}_{\nu,A}(r)=\sum_{i=1}^{m_n}\alpha_{n,i}\nu|_{Ac}(W^c_{n,i})= \int_{S_y(Ac)}f^c_n\,d\nu.
\]
Therefore, by Claim 1 and the Dominated Convergence Theorem, we have
\begin{equation*}
F^{\phi}_{\mu\otimes\nu,A}(r)=\int_{S_y(Ac)}F^{\phi_c}_\mu\,d\nu=\lim_{n\to\infty} \int_{S_y(Ac)}f^c_n\,d\nu=\lim_{n\to\infty}h_n(r).\tag*{\claim}
\end{equation*}

We can now prove the theorem. For part $(a)$,  suppose $\nu\in\kM_y(\cU)$ is BD$^{+}$ over $A$. Then each function $h_n$ above is Borel by Lemma \ref{lem:SBDall}. So $F^{\phi}_{\mu\otimes\nu,A}$ is a pointwise limit of Borel functions by Claim 2, and thus is Borel. Since $\phi(x,y,z)$ is an arbitrary $\cL_A$-formula, we have that $\mu\otimes\nu$ is Borel definable over $A$.

Finally, for part $(b)$, suppose $\nu\in\kM_y(\cU)$ is BD$^{++}$ over $A$. Fix some $\lambda\in\kM_z(\cU)$. Then for any $n\geq 0$ and $i\leq m_n$, we have $(\nu\otimes\lambda)|_A(W_{n,i})=\int_{S_z(A)}F^{W_{n,i}}_{\nu}\,d\lambda$ by Lemma \ref{lem:SBDall}. Therefore
\begin{multline*}
(\mu\otimes(\nu\otimes\lambda))(\phi(x,y,z))=\int_{S_{yz}(A)}F^\phi_{\mu}\,  d(\nu\otimes\lambda) =\lim_{n\to\infty}\int_{S_{yz}(A)}f_n\, d(\nu\otimes\lambda)\\
=\lim_{n\to\infty}\sum_{i=1}^{m_n}\alpha_{n,i}(\nu\otimes\lambda)(W_{n,i})=\lim_{n\to\infty}\sum_{i=1}^{m_n}\alpha_{n,i}\int_{S_z(A)}F^{W_{n,i}}_{\nu}\,d\lambda\\
= \lim_{n\to\infty}\int_{S_z(A)}h_n\,d\lambda = \int_{S_z(A)}F^\phi_{\mu\otimes\nu}\,d\lambda=((\mu\otimes\nu)\otimes\lambda)(\phi(x,y,z)).
\end{multline*}
Note that the second and sixth equalities again use dominated convergence. 
Since $\phi(x,y,z)$ is an arbitrary $\cL_A$-formula, we have  $((\mu\otimes\nu)\otimes\lambda)|_A=(\mu\otimes(\nu\otimes\lambda))|_A$.
\end{proof}

\subsection{Countable sets}
Next we show that in a countable theory, Borel definability coincides with BD$^{++}$ over countable parameter sets. This is another straightforward application of dominated convergence (similar to Lemma \ref{lem:SBDall}).

\begin{lemma}\label{lem:BDcountable}
Assume $T$ is countable, and suppose $\mu\in\kM_x(\cU)$ is Borel definable over a countable set $A\subset\cU$. Then $\mu$ is BD$^{++}$ over $A$.
\end{lemma}
\begin{proof}
Fix an open set $U\seq S_{xy}(A)$. Since $T$ and $A$ are countable, we can write $U=\bigcup_{n<\omega}[\phi_n(x,y)]$, where each $\phi_n(x,y)$ is an $\cL_A$-formula and $\varphi_{n}(\cU^{xy}) \seq \varphi_{n + 1}(\cU^{xy})$ for all $n<\omega$. Given $q\in S_y(A)$ and $b\models q$, we have 
\[
F^U_{\mu,A}(q)=\mu(U(x,b))=\lim_{n \to \infty}\mu(\phi_n(x,b))=\lim_{n \to \infty }F^{\phi_n}_{\mu,A}(q).
\]
So $F^U_{\mu,A}$ is the pointwise limit of a countable sequence of Borel functions, and hence is Borel. Now fix another measure $\nu\in\kM_y(\cU)$. Then
\begin{multline*}
(\mu\otimes\nu)(U)=\lim_{n \to \infty }(\mu\otimes\nu)(\phi_n(x,y))=\lim_{n \to \infty }\int_{S_y(A)} F^{\phi_n}_{\mu}\, d\nu\\
= \int_{S_y(A)}\lim_{n \to \infty }F^{\phi_n}_{\mu}\, d\nu=\int_{S_y(A)} F^U_{\mu}\, d\nu,
\end{multline*}
where the third equality uses the Dominated Convergence Theorem.
\end{proof}

\begin{theorem}\label{thm:ctble}
Assume $T$ is countable, and suppose $\mu\in\kM_x(\cU)$ and $\nu\in\kM_y(\cU)$ are Borel definable over a countable set $A\subset\cU$. Then $\mu\otimes\nu$ is Borel definable over $A$ and, for any $\lambda\in\kM_z(\cU)$, we have $(\mu\otimes\nu)\otimes\lambda=\mu\otimes(\nu\otimes\lambda)$.
\end{theorem}
\begin{proof}
By Theorem \ref{thm:SBDmain}$(a)$ and Lemma \ref{lem:BDcountable}, $\mu\otimes\nu$ is Borel definable over $A$. Note that $\mu$ and $\nu$ are Borel definable over any $B\supseteq A$. So $\nu$ is BD$^{++}$ over any countable $B\supseteq A$ by Lemma \ref{lem:BDcountable}. Therefore, for any $\lambda\in\kM_z(\cU)$, we have $((\mu\otimes\nu)\otimes\lambda)|_B=(\mu\otimes(\nu\otimes\lambda))|_B$ for any countable $B\supseteq A$ by Theorem \ref{thm:SBDmain}$(b)$. It follows that $(\mu\otimes\nu)\otimes\lambda=\mu\otimes(\nu\otimes\lambda)$.
\end{proof}

It is well-known that the Dominated Convergence Theorem does not hold for nets, and so the proof of Lemma \ref{lem:BDcountable} cannot be generalized to Borel definable measures over uncountable models. Indeed, in Section \ref{sec:badBD} we will give an example showing that Theorem \ref{thm:ctble} can fail without the countability assumptions.

\begin{remark}\label{rem:NIPnice}  
Theorem \ref{thm:ctble} holds for NIP theories without the countability assumptions. In particular, Hrushovski and Pillay \cite{HP} proved that if  $T$ is NIP then any invariant Keisler measure is Borel definable (see also \cite[Proposition 7.19]{Sibook}). Combined with Fact \ref{fact:basicinv}, it follows that Borel definability is preserved by Morley products in NIP theories. As for associativity of the Morley product in NIP theories, a proof sketch  is given after Exercise 7.20 in \cite{Sibook}. However, as pointed out recently by Krupi\'{n}ski,  the argument tacitly uses assumptions along the lines of BD$^{++}$ without justification. This motivated the first two authors \cite{CoGaNA}  to write a different proof of associativity in NIP theories, which uses the existence of `smooth extensions'. We will note a similar proof in Corollary \ref{cor:NIPassoc}. 
\end{remark}

\subsection{Definable measures}\label{sec:definable}
In this section we use the material developed above to prove some useful facts about \emph{definable} measures. The definition of this notion is based on  the following standard exercise (see also \cite[Proposition 2.17]{Ganthesis}).

\begin{fact}\label{fact:defdefs}
Suppose $\mu\in\kM_x(\cU)$ is invariant over $A\subset\cU$. Given  an $\cL$-formula $\phi(x,y)$, the following are equivalent.
\begin{enumerate}[$(i)$]
\item $F^\phi_{\mu,A}$ is continuous.
\item For any any $\epsilon>0$, there are $\cL_A$-formulas $\psi_1(y),\ldots,\psi_n(y)$ and real numbers $r_1,\ldots,r_n\in[0,1]$ such that $\|F^\phi_{\mu,A}-\sum_{i=1}^n r_i\boldsymbol{1}_{\psi_i}\|_\infty<\epsilon$.
\item For any $\epsilon>0$, the set $\{b\in\cU^y:\mu(\phi(x,b))\leq\epsilon\}$ is type-definable over $A$.
\end{enumerate}
\end{fact}

\begin{definition}\label{def:definable}
A measure $\mu\in\kM_x(\cU)$ is \textbf{definable} if there is some $A\subset \cU$ such that $\mu$ is $A$-invariant and, for any $\cL$-formula $\phi(x,y)$, the equivalent conditions of Fact \ref{fact:defdefs} hold. In this case, we also say that $\mu$ is \textbf{definable over $A$}.
\end{definition}

Note that condition $(iii)$ of Fact \ref{fact:defdefs} makes sense without assuming $\mu$ is invariant. Moreover, if $(iii)$ holds for all $\cL$-formulas $\phi(x,y)$, then it follows that $\mu$ is $A$-invariant. Therefore, a measure $\mu\in\kM_x(\cU)$ is definable over $A\subset\cU$ if and only if, for any $\cL$-formula $\phi(x,y)$, condition $(iii)$ holds.

In \cite{CoGa}, it is shown that definable measures are closed under Morley products and satisfy associativity.  Here we prove a more general associativity result when only one definable measure is involved. Note first that any definable measure is clearly Borel definable. The next lemma strengthens this fact.

\begin{lemma}\label{lem:defBD2}
If $\mu\in\kM_x(\cU)$ is definable over $A\subset\cU$, then it is BD$^{++}$ over $A$. 
\end{lemma}
\begin{proof}
  Let $U\seq S_{xy}(A)$ be open. Similar to the proof of Lemma \ref{lem:BDcountable}, we can write $U=\bigcup_{i\in I}[\phi_i(x,y)]$ for some collection $\{\phi_i(x,y):i\in I\}$ of $\cL_A$-formulas where $I$ is a directed partial order and for any $i, j \in I$, if $i \leq j$ then $\phi_{i}(\cU^{xy}) \seq \phi_{j}(\cU^{xy})$.  
  Then $F^U_{\mu,A}$ is the pointwise limit of the \emph{increasing} net $(F^{\phi_i}_{\mu,A})_{i\in I}$. Moreover, given $r\in [0,1]$, we claim that $(F_{\mu,A}^U)\inv(\lp r,1\rb)= \bigcup_{i \in I} (F_{\mu,A}^{\phi_i})\inv(\lp r,1\rb)$. Indeed, $p \in (F_{\mu,A}^U)\inv(\lp r,1\rb)$ implies that $\mu(U(x,b)) \in \lp r,1\rb$. By regularity, there exists some $\varphi_{i}(x,b)$ such that $\mu(\varphi_i(x,b)) \in \lp r,1\rb$. The other direction is similar. 
Since each $F^{\varphi_i}_{\mu,A}$ is continuous, we now have that for any $r\in [0,1]$, $(F^U_{\mu,A})\inv(\lp r,1\rb)$ is open.  Since sets of the form $\lp r,1\rb$ generate the Borel $\sigma$-algebra on $[0,1]$, it follows that $F_{\mu,A}^{U}$ is Borel (in fact, upper semi-continuous).
Now, for any $\nu\in\kM_y(\cU)$, we have $\int_{S_y(A)}F^U_{\mu}\,d\nu=\lim_i \int_{S_y(A)}F^{\phi_i}_{\mu}\, d\nu$ by the monotone convergence theorem for uniformly bounded increasing nets of continuous functions on compact Hausdorff spaces (see \cite[Theorem IV.15]{ReedSim}).  
It follows that $\mu$ is BD$^{++}$ over $A$, as in the proof of Lemma \ref{lem:BDcountable}. 
\end{proof}

We can now prove the main associativity result for definable measures.

\begin{theorem}\label{thm:defassoc}
Suppose $\mu\in\kM_x(\cU)$ and $\nu\in\kM_y(\cU)$ are Borel definable over some $A\subset\cU$, and at least one of $\mu$ or $\nu$ is definable over $A$. Then $\mu\otimes\nu$ is Borel definable over $A$ and, for any $\lambda\in\kM_z(\cU)$, we have $(\mu\otimes\nu)\otimes\lambda=\mu\otimes(\nu\otimes\lambda)$.
\end{theorem}
\begin{proof}
First assume $\nu$ is definable over $A$. Then, by Lemma \ref{lem:defBD2} and Theorem \ref{thm:SBDmain}, $\mu\otimes\nu$ is Borel definable over $A$ and $((\mu\otimes\nu)\otimes\lambda)|_A=(\mu\otimes(\nu\otimes\lambda))|_A$ for any $\lambda\in\kM_z(\cU)$. Moreover, note that if $B\supseteq A$ then $\mu$ is definable over $B$ and $\nu$ is Borel definable over $B$. So this argument works over any $B\supseteq A$.
 
Now assume $\mu$ is definable over $A$.  In this case, the proof is almost the same as that of Theorem \ref{thm:SBDmain} (applied to $\mu$ and $\nu$), and so we just explain the necessary adjustments. In particular, since $\mu$ is definable, we can use Fact \ref{fact:defdefs} to assume that the Borel sets $W_{n,i}$ in the proof of Theorem \ref{thm:SBDmain} are actually \emph{clopen}. Therefore, one only needs Borel definability of $\nu$ to conclude that the maps $F^{W_{n,i}}_{\nu,A}$ and $h_n$ are each Borel. This is all that is needed to conclude $\mu\otimes\nu$ is Borel definable over $A$. Finally, in the associativity argument, we do not need to assume $\nu$ is BD$^{++}$ to know that $(\nu\otimes\lambda)|_A(W_{n,i})=\int_{S_z(A)}F^{W_{n,i}}_{\nu}\,d\nu$. Indeed, since $W_{n,i}$ is clopen, this follows from the definition of the Morley product. So we have $((\mu\otimes\nu)\otimes\lambda)|_A=(\mu\otimes(\nu\otimes\lambda))|_A$ by the same steps. Once again, this argument works over any $B\supseteq A$.
\end{proof}

\begin{remark}\label{rem:Baire}
Call an $A$-invariant measure $\mu\in\kM_x(\cU)$ \emph{Baire-1 definable over $A$} if 
for any $\cL$-formula $\phi(x,y)$, $F^\phi_{\mu,A}$ is a function of Baire  class 1, i.e., the pointwise limit of  a sequence of continuous functions. Note that, as a property of measures, Baire-1 definability is stronger than Borel definability, but weaker than definability.
 We claim that if $\mu\in\kM_x(\cU)$ is Baire-1 definable over $A$, and $\nu\in\kM_y(\cU)$ is Borel definable over $A$,  then $\mu\otimes\nu$ is Borel definable over $A$ and, for any $\lambda\in\kM_z(\cU)$, we have $(\mu\otimes\nu)\otimes\lambda=\mu\otimes(\nu\otimes\lambda)$. Indeed, the proof of Theorem \ref{thm:SBDmain} only required pointwise limits, and so one can argue using the same adjustments as in Theorem \ref{thm:defassoc} (together with the exercise that a Baire-1 function on a Stone space is a pointwise limit of finite linear combinations of indicator functions of clopen sets).
\end{remark}

In light of Theorem \ref{thm:defassoc}, it is natural to ask if one gains any traction in proving associativity by assuming that the measure in the third position is definable. In Corollary \ref{cor:TRassoc}, we will give an example of Borel definable types $p$ and $q$, and a definable measure $\lambda$, such that $p\otimes q$ is Borel definable,  but $(p\otimes q)\otimes\lambda\neq p\otimes(q\otimes\lambda)$. On the other hand, we do have the following result.  

\begin{corollary}\label{cor:DCassoc}
Fix $\mu\in\kM_x(\cU)$ and $\nu\in\kM_y(\cU)$ such that $\mu$, $\nu$, and $\mu\otimes\nu$ are each Borel definable over some $A\subset\cU$. Suppose $\lambda\in\kM_z(\cU)$ is such that  $\lambda|_A$ has a  definable global extension $\hat{\lambda}\in\kM_z(\cU)$ that commutes with $\mu$, $\nu$, and $\mu\otimes\nu$. Then $((\mu\otimes\nu)\otimes\lambda)|_A=(\mu\otimes(\nu\otimes\lambda))|_A$.
\end{corollary}
\begin{proof}
Fix an $\cL_A$-formula $\phi(x,y,z)$. Then we have the following calculations (individual steps are justified afterward):
\begin{multline*}
((\mu\otimes\nu)\otimes\lambda)(\phi(x,y,z))=((\mu\otimes\nu)\otimes\hat{\lambda})(\phi(x,y,z))\\
=(\hat{\lambda}\otimes(\mu\otimes\nu))(\phi(x,y,z))=((\hat{\lambda}\otimes\mu)\otimes\nu)(\phi(x,y,z))\\
=((\mu\otimes\hat{\lambda})\otimes\nu)(\phi(x,y,z))=(\mu\otimes(\hat{\lambda}\otimes\nu))(\phi(x,y,z))\\
=(\mu\otimes(\nu\otimes\hat{\lambda}))(\phi(x,y,z))=(\mu\otimes(\nu\otimes\lambda))(\phi(x,y,z)).
\end{multline*}
In the above calculations, the first and last equalities  use $\lambda|_A=\hat{\lambda}|_A$, the third and fifth equalities use Theorem \ref{thm:defassoc} and definability of $\hat{\lambda}$, and the remaining equalities use the commutativity assumptions on $\hat{\lambda}$. 
\end{proof}

\begin{remark}\label{rem:thirdtype}
Note that in the previous result we do not need to assume that $\hat{\lambda}$ is definable \emph{over $A$}. For example, given  $p\in S_x(\cU)$ and $A\subset \cU$, if $a\models p|_A$ and $\hat{p}=\tp(a/\cU)$, then $p|_A=\hat{p}|_A$, $\hat{p}$ is definable over $\{a\}$, and $\hat{p}$ commutes with any invariant measure. So  if the measure $\lambda$ in Corollary \ref{cor:DCassoc} is a type, then such a $\hat{\lambda}$ exists for any $A\subset\cU$ (See Fact \ref{fact:thirdtype} below for a full account of associativity  when the measure in the third position is a type.) 
\end{remark}

For a more interesting example of when Corollary \ref{cor:DCassoc} is applicable, one can turn to the class of NIP theories. Recall that if $T$ is NIP then any invariant global Keisler measure is Borel definable by \cite{HP}. Moreover, by the original work of Keisler \cite{Keis}, any  measure  over a small model $M$ of an NIP theory has a global extension that is \emph{smooth}, i.e., it is the unique global extension of its restriction to some small model $N\succeq M$  (see also \cite[Proposition 7.9]{Sibook}). For example, if $p\in S_x(M)$ and $a\in\cU^x$ realizes $p$, then $\tp(a/\cU)$ is a smooth global extension of $p$. It is not hard to show that smooth measures are definable and commute with all Borel definable measures (see \cite[Section 2]{HPS}). So if $T$ is NIP then any global measure $\lambda$ satisfies the assumptions of Corollary \ref{cor:DCassoc} for any $A\subset\cU$. Altogether, this reaffirms associativity of the Morley product for invariant measures in NIP theories (recall Remark \ref{rem:NIPnice}). 

\begin{corollary}\label{cor:NIPassoc}
If $T$ is NIP then the Morley product of invariant measures is associative.
\end{corollary}

\section{Counterexamples in Borel definability}\label{sec:badBD}

The goal of this section is to show that, over uncountable sets, the Morley product of two Borel definable measures need not be Borel definable and, moreover, that the Morley product can fail to be associative (even when all products involved are well-defined and Borel definable). In fact, we will demonstrate this behavior in a relatively straightforward simple unstable (countable) theory.  Before getting into this example, we first discuss some preliminaries.

\subsection{Strongly continuous measures}
The purpose of this subsection is mainly to provide context for how  Morley products can fail associativity. Let $T$ be a complete $\cL$-theory with monster model $\cU$. It is well-known (and easy to show) that the Morley product is associative with respect to invariant types (see \cite[Fact 2.20]{Sibook}), and so a  failure of associativity must involve at least one `true' measure. Toward making this remark more precise, we state the following fact, which is left as an exercise (the mechanics of the proof are similar to that of Corollary \ref{cor:DCassoc}).

\begin{fact}\label{fact:thirdtype}
Suppose $\mu\in\kM_x(\cU)$ is Borel definable and $\nu\in\kM_y(\cU)$ is invariant. Then $(\mu\otimes\nu)\otimes r=\mu\otimes(\nu \otimes r)$ for any $r\in S_z(\cU)$. 
\end{fact}

Next, we recall some terminology (taken from the theory of \emph{charges} \cite{RaoRao}) that will be used to make the idea of a `true' measure more rigorous. 

\begin{definition}
Given $A\seq\cU$, a measure $\mu\in\kM_x(A)$ is \textbf{strongly continuous} if, for any $\epsilon>0$, there is a partition of $S_x(A)$ into finitely  many clopen sets of measure less than $\epsilon$. 
\end{definition}

Despite the use of the word `continuous' in the previous definition, we caution the reader that there is no general connection between strongly continuous measures and definable measures.

\begin{remark}\label{rem:SCp}
By compactness, a measure $\mu\in\kM_x(A)$ is strongly continuous if and only if $\mu(\{p\})=0$ for all $p\in S_x(A)$. Note also that if $\mu\in\kM_x(\cU)$ is strongly continuous, then there is some countable $A\subset\cU$ such that $\mu|_A$ is strongly continuous.
\end{remark}

Let $\mu\in\kM_x(\cU)$ be a Keisler measure. By the Sobczyk-Hammer decomposition theorem for finitely additive bounded charges (see \cite[Theorem 5.2.7]{RaoRao}), one can write $\mu=\alpha_0\mu_0+\sum_{n=1}^\infty \alpha_n p_n$ where each $\alpha_n$ is in $[0,1]$ with $\sum_{n=0}^\infty \alpha_n=1$, each $p_n$ is a type in $ S_x(\cU)$, and either $\mu_0$ is a strongly continuous  measure in $\kM_x(\cU)$, or $\alpha_0=0$ and $\mu_0$ is  the identically zero measure (in this case, we call $\mu$ \emph{atomic}). The next fact, which we leave as an exercise, follows from Fact \ref{fact:thirdtype} together with standard measure-theoretic computations.

\begin{fact} Fix $\mu \in \kM_{x}(\cU)$, $\nu \in \kM_{y}(\cU)$, and $\lambda \in \kM_{z}(\cU)$. Let $\lambda=\alpha_0\lambda_0+\sum_{n=1}^\infty \alpha_n p_n$ be the Sobczyk-Hammer decomposition of $\lambda$ described above. 
\begin{enumerate}[$(a)$]
\item Assume $\lambda$ is atomic. If $\mu$ is Borel definable and $\nu$ is invariant, then $(\mu\otimes\nu)\otimes\lambda$ is well-defined and equal to $\mu\otimes(\nu\otimes\lambda)$.
\item Assume $\lambda$ is not atomic. If  $\mu$, $\nu$, and $\mu\otimes\nu$ are each Borel definable,  then $(\mu\otimes\nu)\otimes\lambda=\mu\otimes(\nu\otimes\lambda)$ if and only if $(\mu\otimes\nu)\otimes\lambda_0=\mu\otimes(\nu\otimes\lambda_0)$
\end{enumerate}
\end{fact}

In other words, this fact says that in any situation where the Morley product of Keisler measures fails associativity, the measure in the third coordinate cannot be atomic, and so the failure of associativity can be traced back to an underlying strongly continuous measure.

Finally, we recall a known result from the folklore characterizing the existence of strongly continuous Keisler measures. 

\begin{fact}\label{fact:SC}
Given a complete theory $T$, the following are equivalent.
\begin{enumerate}[$(i)$]
\item $T$ is totally transcendental (i.e., every formula has ordinal Morley rank).
\item There is no strongly continuous measure in $\kM_x(\cU)$ for any $x$.
\item There is no strongly continuous measure in $\kM_x(\cU)$ for any tuple of variables $x$ of length one. 
\end{enumerate}
\end{fact}
\begin{proof}
This follows from standard results on type spaces in totally transcendental theories, combined with various facts about strongly continuous measures (see \cite[Theorem 5.3.2, Lemma 5.3.8, Theorem 5.3.9]{RaoRao}). See also \cite[Fact 3.3]{ChGan}, which discusses some details, including the relevance of \cite[Lemma 1.7]{Keis}. 
\end{proof}

\subsection{Relative measurability}\label{sec:bern}
Later in this section, we will construct Borel definable global types $p$ and $q$ (in a specific theory) such that $p\otimes q$ is not Borel definable. In this case, one might wonder if $p\otimes q$ still admits Morley products with restricted classes of measures with nice behavior (e.g., if $r$ is a type then $(p\otimes q)\otimes r$ is well-defined since $p\otimes q$ is still invariant). However, our construction will show that $p\otimes q$ can be arbitrarily bad. This will be made precise using the following notions.

\begin{definition}
Let $X$ be a topological space. Given a Borel measure $\mu$ on $X$, a subset of $X$ is called \textbf{$\mu$-measurable} if it is measurable with respect to the completion of $\mu$. A subset $Z$ of  $X$  is called a \textbf{Bernstein set} if both $Z$ and $X\backslash Z$ nontrivially intersect every uncountable closed subset of $X$.
\end{definition}

Recall that, in the above context, a subset of $X$ is $\mu$-measurable if and only if it is of the form $B\cup E$, where $B$ is Borel and $E$ is contained in a $\mu$-null Borel set. It is also a standard fact that any Polish space contains a Bernstein set (see, e.g., Theorem 4 in \cite[Chapter 11]{JustWeese}).

\begin{lemma}\label{lem:bern}
Suppose $\rho\colon X\to Y$ is a surjective continuous map between compact Hausdorff spaces, and $\mu$ is a regular Borel  measure on $X$. Let $\nu$ be the pushforward of $\mu$ along $\rho$, and assume that any singleton in $Y$ is $\nu$-null. Then, for any Bernstein set $Z\seq Y$, the set $\rho\inv(Z)$ is not $\mu$-measurable.
\end{lemma}
\begin{proof}
Suppose $\rho\inv(Z)$ is $\mu$-measurable. Since $Y\backslash Z$ is also a Bernstein set, we may assume without loss of generality that $\mu(\rho\inv(Z))>0$. By regularity, there is a closed set $C\seq \rho\inv(Z)$ such that $\mu(C)>0$. Since $\rho(C)$ is closed, and contained in $Z$, it must be countable. So $\nu(\rho(C))=0$ by assumption on $\nu$ and countable additivity. But then $\mu(C)\leq \mu(\rho\inv(\rho(C)))=\nu(\rho(C))=0$, which is a contradiction.
\end{proof}

\begin{corollary}\label{cor:bern}
Let $T$ be a complete $\cL$-theory with monster model $\cU$. Fix $A\seq\cU$ and suppose $\mu\in\kM_x(A)$ is strongly continuous. Then there is a subset of $S_x(A)$ that is not $\mu$-measurable.
\end{corollary}
\begin{proof}
Choose a countable set $A_0\seq A$, and a countable sub-language $\cL_0\seq\cL$, such that if $\nu$ is the restriction of $\mu|_{A_0}$ to $(\cL_0)_{A_0}$-formulas, then $\nu$ is still strongly continuous. Let $X=S^{\cL}_x(A)$ and $Y=S^{\cL_0}_x(A_0)$, and define $\rho\colon X\to Y$ to be the composition of $\rho^x_{A,A_0}$ with restriction to $\cL_0$. Then $\nu$ is the pushforward of $\mu$ along $\rho$ (as in Remark \ref{rem:push}). Not that any singleton in $Y$ is $\nu$-null by strong continuity of $\mu$. Since $Y$ is Polish, there is a Bernstein set $Z\seq Y$. Altogether, $\rho\inv(Z)\seq S_x(A)$  is not $\mu$-measurable by Lemma \ref{lem:bern}. 
\end{proof}

\subsection{The random ternary relation}\label{sec:TR}
Given any finite relational language $\cL$, the class of finite $\cL$-structures is a \Fraisse\ class, and the complete theory of the corresponding \Fraisse\ limit is $\aleph_0$-categorical and has quantifier elimination (this follows from \cite[Theorem 7.4.1]{Hobook}). It is well known, and not hard to prove, that any theory obtained this way is supersimple of SU-rank $1$, and is also unstable if and only if $\cL$ contains a relation of arity at least $2$. 

In this subsection, we work with the theory $T_R$ obtained in the above fashion, where $\cL$ consists of a single \emph{ternary} relation $R(x,y,z)$. Throughout this section, $\cU$ is a monster model of $T_R$. 
We will first show that in $T_R$, Borel definability of measures is not preserved by Morley products. So this refutes the unproven claim in \cite[Lemma 1.6]{HPS}. In fact, we show that the product of Borel definable \emph{types} is not necessarily Borel definable.

\begin{proposition}\label{prop:ternary1}
There are Borel definable $p,q\in S_1(\cU)$ such that $p\otimes q$ is not Borel definable.  
\end{proposition}
\begin{proof}
Fix an infinite set $B\subset\cU$ and an arbitrary set $Z\seq S_1(B)$ such that $\kappa\coloneqq |Z|\geq|B|$.
We will construct types $p,q\in S_1(\cU)$ satisfying the following properties.
\begin{enumerate}[$(i)$]
\item $p$ and $q$ are Borel definable over some $A\supseteq B$ of cardinality $\kappa$.
\item For any $c\in\cU$, $R(x,y,c)\in p\otimes q$ if and only if $\tp(c/B)\in Z$.
\end{enumerate}
In particular, setting $Z^*= (\rho^z_{A,B})\inv(Z)$, we have  $F^{R(x,y;z)}_{p\otimes q,A}=\boldsymbol{1}_{Z^*}$ by $(ii)$. So if $Z$ is not Borel then $p\otimes q$ is not Borel definable by Theorem \ref{thm:HolSpur} and Corollary \ref{cor:BDdrop}.
 Note also that $S_1(B)$ has a topological basis of size $|B|$, and thus has at most $2^{|B|}$ Borel subsets. On the other hand, $S_1(B)$ has $2^{2^{|B|}}$ subsets (since it has size $2^{|B|}$). So there are non-Borel choices for the set $Z$ above. 

Fix $A\supseteq B$ of cardinality $\kappa$. Given an $A$-invariant type $p\in S_{\xbar}(\cU)$ and a formula $\phi(\xbar;\ybar)$, we define  $dp(\phi)=\{s\in S_{\ybar}(A):\phi(\xbar;\bbar)\in p\text{ for }\bbar\models s\}$ (in particular, we have $F_{p,A}^{\varphi} = \mathbf{1}_{dp(\varphi)}$). Let $x,y,z$ be tuples of variables of length one, and let $R_1(x;y,z)$ and $R_2(y;x,z)$ be partitions of $R(x,y,z)$. Enumerate $A=\{a_i:i<\kappa\}$ and $Z=\{r_i:i<\kappa\}$. We construct the desired types $p$ and $q$. 

First, define $p\in S_x(\cU)$ so that the positive instances of $R$ in $p$ (which actually involve $x$) are precisely those of the form $R(x,b,c)$ where $b,c\in\cU$ are such that 
$R(a_i,b,c)$ holds for some $i<\kappa$.
Note that $p$ is $A$-invariant. To prove Borel definability of $p$, it suffices by quantifier elimination to focus on atomic formulas; and by definition of $p$, we only need to consider $R_1(x;y,z)$. 
By construction, $dp(R_1)=\bigcup_{i<\kappa}[R(m_i,y,z)]$, which is open. 

Now define $q\in S_y(\cU)$ so that the positive instances of $R$ in $q$ (which actually involve $y$) are precisely those of the form $R(a,y,c)$ where $a,c\in\cU$ are such that $a=a_i$ and $c\models r_i$ for some $i<\kappa$.
Then $q$ is $A$-invariant by construction. We claim that $q$ is Borel-definable. By quantifier elimination, and the definition of $q$, it suffices to consider $R_2(y;x,z)$. Given $i<\kappa$, set $K_i=\{s\in S_{xz}(A):s_{z}|_B=r_i\}$, and note that $K_i$ is closed. From the definition of $q$, we have that $dq(R_2)=\bigcup_{i<\kappa} K_i\cap [x=a_i]$, which is an infinite union of closed sets. However, if we set $K=\bigcap_{i<\kappa}K_i\cup [x\neq a_i]$ and $U=\bigcup_{i<\kappa}[x=a_i]$, then $K$ is closed, $U$ is open, and $dq(R_2)=K\cap U$.

We have now built $p$ and $q$ satisfying $(i)$. It remains to show that $p\otimes q$ satisfies $(ii)$. It is easy to check that any positive instance of $R$  in $p\otimes q$ involving the variables $x$ and $y$ must have the form $R(x,y,c)$ for some $c\in\cU$. 
So fix $c\in\cU$.  Using the definition of $p$, we have $R(x,y,c)\in p\otimes q$ if and only if there are $i<\kappa$ and $b\models q|_{Ac}$ such that $R(a_i,b,c)$. Using the definition of $q$, we conclude that $R(x,y,c)\in p\otimes q$ if and only if $c\models r_i$ for some $i<\kappa$.  Altogether, we have property $(ii)$.
\end{proof}

We now use the previous construction  to produce Borel definable types $p,q\in S_1(\cU)$, and a measure $\lambda\in\kM_1(\cU)$, such that the Morley product of $p\otimes q$ with $\lambda$ is not well-defined.  First, using Fact \ref{fact:SC}, we may fix a strongly continuous measure $\lambda\in\kM_1(\cU)$ (one can even choose $\lambda$ to be \emph{definable} via Lemma \ref{lem:LebTR} below). Let $B\subset \cU$ be an infinite set such that $\lambda|_B$ is strongly continuous. By Corollary \ref{cor:bern}, there is a set $Z\seq S_1(B)$ that is not $\lambda|_B$-measurable. Now choose $p,q\in S_1(\cU)$ as in the proof of Proposition \ref{prop:ternary1}; in particular, $R(x,y,c)\in p\otimes q$ if and only if $\tp(c/B)\in Z$. Then $(p\otimes q)\otimes \lambda$ is not well-defined since $F^R_{p\otimes q,B}=\boldsymbol{1}_Z$. Note that $p\otimes (q\otimes\lambda)$ is well-defined however, and so this also produces a rather cheap failure of associativity. 

Next we will demonstrate a more substantial failure of associativity  in which all Morley products involved are well-defined. Intuitively speaking, the construction uses something like the `first-year probability theory paradox' that the measure of $[0,1]$ is $1$, yet the measure of $\{x\}$ for any $x \in [0,1]$ is $0$.

\begin{proposition}\label{prop:ternary2}
Suppose $\lambda\in\kM_1(\cU)$ is a strongly continuous measure. Then there are types $p,q\in S_1(\cU)$ such that $p$, $q$, and $p\otimes q$ are Borel definable, the measures $q\otimes\lambda$, $(p\otimes q)\otimes\lambda$, and $p\otimes(q\otimes\lambda)$ are well-defined, but $((p\otimes q)\otimes\lambda)(R(x,y,z))=1$ and $(p\otimes(q\otimes\lambda))(R(x,y,z))=0$.
\end{proposition}
\begin{proof}
We use similar notation as in the proof of Proposition \ref{prop:ternary1}. Using Remark \ref{rem:SCp}, we may choose infinite $B\seq A\subset\cU$ such that $\kappa\coloneqq |A|=2^{|B|}$ and $\lambda|_A(\{r\})=0$ for all $r\in S_1(A)$. Enumerate $A=\{a_i:i<\kappa\}$ and $S_1(B)=\{r_i:i<\kappa\}$.

Define $p\in S_x(\cU)$ so that the positive instances of $R$ in $p$ are precisely those of the form $R(x,b,c)$ where $b,c\in\cU$ are such that $R(a_i,b,c)$ holds for some $i<\kappa$. Define $q\in S_y(\cU)$ so that the positive instances of $R$ in $q$ are precisely those of the form $R(a,y,c)$ where $a,c\in \cU$ are such that $a=a_i$ and $c\models r_i$ for some $i<\kappa$. In other words, $p$ and $q$ are exactly as in the proof of Proposition \ref{prop:ternary1}, if one chooses $Z=S_z(B)$ in the definition of $q$. So $p$, $q$, and $p\otimes q$ are Borel definable over $A$. (Note that, in the general construction from the proof of Proposition \ref{prop:ternary1}, $p\otimes q$ is Borel definable if and only if $Z$ is Borel.)

 Note that $q\otimes\lambda$, $(p\otimes q)\otimes\lambda$, and $p\otimes(q\otimes\lambda)$ are well-defined since the left most term in each product is Borel definable.  Let $\eta_1=(p\otimes q)\otimes \lambda$ and $\eta_2=p\otimes(q\otimes\lambda)$.  Then  $\eta_1(R(x,y,z))=1$ since $F^{R(x,y;z)}_{p\otimes q,A}$ takes the constant value $1$ on $S_z(A)$. It remains to  show that $\eta_2(R(x,y,z))=0$. 
 
By definition, we have $\eta_2(R(x,y,z))=(q\otimes\lambda)|_A(U)$ where $U\coloneqq dp(R(x;y,z))=\bigcup_{i<\kappa}[R(a_i,y,z)]$. So $U$  is an open set in $S_{yz}(A)$.  Moreover,  if $i<\kappa$ then 
\[
(q\otimes\lambda)(R(a_i,y,z))=\lambda|_A(dq(R(a_i,y;z)))=\lambda|_A(\{r_i\})=0.
\] 
So by compactness, Fact \ref{fact:basicBorel}$(a)$, and finite additivity of $(q\otimes\lambda)|_A$, we have
\[
\eta_2(R(x,y,z))=(q\otimes \lambda)|_A(U) \leq \sup_{I\in [\kappa]^{<\omega}}\sum_{i\in I}(q \otimes \lambda)|_A (R(a_i,y,z))=0.\qedhere
\]
\end{proof}

In fact, we can strengthen the previous result using the existence of a \emph{definable} strongly continuous measure in $T_R$, namely, the  `coin-flipping' measure, which independently assigns $R$-neighborhoods measure $\frac{1}{2}$. Similar measures on the random graph are studied by Albert in \cite{AlbRG}.

\begin{lemma}\label{lem:LebTR}
There is an $\emptyset$-definable strongly continuous measure in $\kM_1(\cU)$.
\end{lemma}
\begin{proof}
Let $\lambda$ be the unique measure in $\kM_1(\cU)$ satisfying the property that if $\theta_1(x),\ldots,\theta_n(x)$ are pairwise distinct (positive) instances of $R$ in one free variable, and $\psi_i(x)$ is either $\theta_i(x)$ or $\neg\theta_i(x)$, then  
\[
\lambda(\psi_1(x)\wedge\ldots\wedge\psi_n(x))=\frac{1}{2^n}.
\]
The justification that such a measure exists is given in Section \ref{app:TR} of the appendix. 

To see that $\lambda$ is strongly continuous, fix $n>0$ and distinct $a_1,\ldots,a_n\in \cU$. Suppose $\theta(x)=\bigwedge_{i=1}^n\theta_i(x)$, where $\theta_i(x)$ is either $R(x,a_i,a_i)$ or $\neg R(x,a_i,a_i)$. Then $\lambda(\theta(x))=\frac{1}{2^n}$. Since the collection of all such $\theta(x)$ forms a finite partition of $\cU^x$, we conclude that $\lambda$ is strongly continuous.

Finally, to see that $\lambda$ is $\emptyset$-definable, fix a formula $\phi(x;y_1,\ldots,y_n)$. By quantifier-elimination, we may assume $\phi$ is a conjunction of atomic and negated atomic formulas. We may also assume without loss of generality that $\phi$ contains $y_i\neq y_j$ for all $i\leq j$. Note that $\lambda(X)=0$ for any finite $X\seq\cU$. Altogether, every consistent instance of $\phi$ has the same measure. Therefore for any formula $\theta(x;y_1,...,y_n) := \varphi(x;y_1,...,y_n) \bigwedge_{1 \leq i < j \leq n}^{n} y_i \neq y_j$, the map $F_{\lambda,\emptyset}^{\theta}:S_{\ybar}(\emptyset) \to [0,1]$ is well-defined and constant, and in particular, continuous. Hence $\lambda$ is $\emptyset$-definable.
\end{proof}

The two previous results together yield a  strong failure of associativity for the Morley product in $T_R$, which also provides a counterpoint to Theorem \ref{thm:defassoc}.

\begin{corollary}\label{cor:TRassoc}
There are $p,q\in S_1(\cU)$ and $\lambda\in S_1(\cU)$ such that $p$, $q$, and $p\otimes q$ are Borel definable, and $\lambda$ is $\emptyset$-definable, but $((p\otimes q)\otimes\lambda)(R(x,y,z))=1$ and $(p\otimes(q\otimes\lambda))(R(x,y,z))=0$.
\end{corollary}

Note that in the previous result, $q\otimes\lambda$ and $(p\otimes q)\otimes\lambda$ are also Borel definable by Theorem \ref{thm:defassoc}. One can further show that if $\lambda$ is the specific measure from Lemma \ref{lem:LebTR}, then $p\otimes (q\otimes\lambda)$ is Borel definable. But this involves a somewhat technical case analysis  so we omit the details.

\section{Fim fam flim flam}\label{sec:fimfam}
We now change our overall focus from Borel definability to stronger notions motivated by the study of model-theoretic tameness.
In this section, we review several properties which, in the setting of NIP theories, characterize a canonical notion of `generic stability' for invariant Keisler measures. These properties are referred to using the descriptors \fim, \fam, and \dfs, which stand for \emph{frequency interpretation measure}, \emph{finitely approximated measure}, and \emph{definable and finitely satisfiable}, respectively. See Definition \ref{def:fimfam} below for full details. Much of the motivation for studying these notions comes from the fundamental result, due to Hrushovski, Pillay, and Simon \cite{HPS}, that if $T$ is NIP then \fim, \fam, and \dfs\ are equivalent. More precisely, we have the following  implications:
\[
\fim\makebox[.3in]{$\Rightarrow$}\fam\makebox[.3in]{$\Rightarrow$}\dfs\makebox[.4in]{$\stackrel{\raisebox{1pt}{\text{\tiny NIP}}}{\Rightarrow}$}\fim.
\]
The first implication is clear from the definitions (given below), the second is a standard exercise (e.g., \cite[Proposition 2.30]{Ganthesis}; see also Proposition \ref{prop:dfschar} below), and the third is  \cite[Theorem 3.2]{HPS}. The purpose of this section is to rapidly review the parade of definitions and basic facts about \fim, \fam, and \dfs\ that we will need for later results.

Let $T$ be a complete theory with monster model $\cU$.  Given measures $\mu,\nu\in\kM_x(\cU)$, and some $\cL_{\cU}$-formula $\phi(x,y)$, we write $\mu\approx^\phi_\epsilon\nu$ to denote that $\mu(\phi(x,b))\approx_\epsilon\nu(\phi(x,b))$ for all $b\in\cU^y$. Note that if $\mu$ and $\nu$ are invariant over $A\subset\cU$ and $\phi(x,y)$ is an $\cL_A$-formula, then $\mu\approx^\phi_\epsilon\nu$ if and only if $\|F^\phi_{\mu,A}-F^\phi_{\nu,A}\|_\infty<\epsilon$.

Given a Borel definable measure $\mu\in\kM_x(\cU)$ and some $n\geq 1$, we  define $\mu^{(n)}\in \kM_{x_1\ldots x_n}(\cU)$ by setting $\mu^{(1)}=\mu_{x_1}$ and $\mu^{(n+1)}=\mu_{x_{n+1}}\otimes\mu^{(n)}_{x_1\ldots x_n}$. Note that even if $\mu^{(n)}$ is not Borel definable for some $n$, the product involved in the definition of $\mu^{(n+1)}$ is still well-defined.   Also, if $\mu$ is a type $p\in S_x(\cU)$, then one only needs invariance in order to define $p^{(n)}$.

We now recall the definitions of the properties mentioned above. For various reasons, these notions are more effective when formulated over small models, rather than arbitrary parameter sets. Thus we will now shift our focus to small models.

\begin{definition}\label{def:fimfam}
Fix $\mu\in\kM_x(\cU)$.
\begin{enumerate}[$(1)$]
\item $\mu$ is \textbf{finitely satisfiable in $M\prec\cU$} if for any $\cL_\cU$-formula $\phi(x)$, if $\mu(\phi(x))>0$ then $\phi(x)$ is realized in $M$.
\item $\mu$ is \textbf{\textit{dfs}} if there is some $M\prec\cU$ such that $\mu$ is definable over $M$ and finitely satisfiable in $M$. In this case, we also say that $\mu$ is \textbf{\textit{dfs} over $M$}.
\item $\mu$ is  \textbf{\textit{fam}} (`finitely approximated measure') if there is some $M\prec\cU$ such that, for any $\cL$-formula $\phi(x,y)$ and any $\epsilon>0$, there is $\abar\in (M^x)^n$ such that $\mu \approx^\phi_\epsilon \Av(\abar)$. In this case, we also say that $M$ is \textbf{\textit{fam} over $M$}. 
\item $\mu$ is \textbf{\textit{fim}} (`frequency interpretation measure') if there is some $M\prec\cU$ such that, for any $\cL$-formula $\phi(x,y)$, there is  a sequence $(\theta_n(x_1,\ldots,x_n))_{n=1}^\infty$ of consistent $\cL_M$-formulas satisfying the following properties:
\begin{enumerate}[$(i)$]
\item For any $\epsilon>0$ there is some $n(\epsilon)\geq 1$ such that, if $n\geq n(\epsilon)$ and  $\abar\models\theta_n$, then $\mu\approx^\phi_\epsilon\Av(\abar)$.
\item $\lim_{n\to\infty}\mu^{(n)}(\theta_n)=1$.
\end{enumerate}
In this case, we also say that $M$ is \textbf{\textit{fim} over $M$}. 
\end{enumerate}
\end{definition}

The definition of \fim\ implicitly assumes that the Morley products $\mu^{(n)}$ are well-defined. This is justified by the fact that condition $(i)$ ensures $\mu$ is \fam\ (since we work over small models), and thus definable. We also note the following easy facts.

\begin{remark}
Fix $\mu\in\kM_x(\cU)$ and $M\prec\cU$. 
\begin{enumerate}[$(1)$]
\item If $\mu$ is finitely satisfiable in $M$ then it is invariant over $M$.
\item If $\abar\in (M^x)^n$ then $\Av(\abar)\in\kM_x(\cU)$ is \fim\ over $M$.
\end{enumerate}
\end{remark}

Next we take the opportunity  to provide a novel characterization of \dfs, which is formulated using \fam-like behavior. In particular, we show that \dfs\ is equivalent to being ``piecewise" \fam. This result also provides further evidence that  \dfs\ is a natural notion in its own right, rather than just an arbitrary combination of two separate notions. 

\begin{proposition}\label{prop:dfschar}
Fix $\mu\in\kM_x(\cU)$ and $M\prec\cU$. Then $\mu$ is \dfs\ over $M$ if and only if for any $\cL$-formula $\varphi(x,y)$ and any $\epsilon>0$, there are  tuples $\abar_1\in(M^x)^{n_1},\ldots,\abar_k\in(M^x)^{n_k}$ and $\cL_M$-formulas $\psi_1(y),\ldots,\psi_k(y)$ partitioning $M^y$ such that for any $b\in \cU^y$, if $\cU\models\psi_i(b)$ then $\mu(\varphi(x,b))\approx_\epsilon \Av(\abar_i)(\varphi(x,b))$. 
\end{proposition}
\begin{proof}
Suppose first that $\mu$ satisfies the latter condition. Then $\mu$ is clearly invariant over $M$ and, for any $\cL$-formula $\varphi(x,y)$, $F_{\mu,M}^\varphi$ is a uniform limit of continuous functions on $S_y(M)$ (specifically piecewise constant functions with clopen pieces). Therefore $F_{\mu,M}^\varphi$ is continuous for any $\varphi(x,y)$, whence $\mu$ is definable over $M$. Finite satisfiability of $\mu$ in $M$ is straightforward and left to the reader (the argument is nearly identical to showing that \fam\ over $M$ implies finitely satisfiable in $M$).

Conversely, suppose $\mu$ is \dfs\ over $M$. Fix an $\cL$-formula $\varphi(x,y)$. Then $F^\varphi_{\mu,M}$ is continuous. Let $\cF$ be the set of functions $F^{\varphi}_{\Av(\abar),M}$ for $\abar\in (M^x)^{<\omega}$. In particular, $\cF$ is a set of continuous functions from $S_y(M)$ to $[0,1]$.  To ease notation, let $g=F^\varphi_{\mu,M}$ and $X=S_y(M)$.

\medskip

\noindent\emph{Claim:} For any $q\in X$ and any $\epsilon>0$, there is some $f\in\cF$ such that $|f(q)-g(q)|<\epsilon$.

\noindent\emph{Proof:} Fix $q\in X$ and $b\models q$. There are three cases: either $\mu(\varphi(x,b)) = 0$, $\mu(\varphi(x,b)) = 1$, or $\mu(\varphi(x,b)) \in (0,1)$. 
    
    In the first case, we can find $a \in M^x$ such that $\neg \varphi(a,b)$ holds. In the second case, we can find $a \in M^x$ such that $\varphi(a,b)$ holds. So in either case, if $f=F^{\varphi}_{\Av(a),M}=\boldsymbol{1}_{\varphi(a,y)}$, then $f\in\cF$ and $f(q)=g(q)$.

    In the third case, we can find $a_+,a_-\in M^x$ such that $\varphi(a_+,b)$ and $\neg \varphi(a_-,b)$ both hold. Fix $\epsilon>0$, and choose a rational $r=m/n\in [0,1]$ such that $|r-g(q)|<\epsilon$. Let $\abar=(a_1,\ldots,a_n)$ with $a_i=a_+$ for $1\leq i\leq m$ and $a_i=a_-$ for $m<i\leq n$. If $f=F^{\varphi}_{\Av(\abar),M}$ then $f\in \cF$ and $f(q)=r$, hence $|f(q)-g(q)|<\epsilon$.\claim
 
 \medskip 
 
Fix $\e>0$. We will now find a clopen partition $A_1,\ldots,A_k$ of $X$, along with functions $f_1,\ldots,f_k\in \cF$, such that for any $q\in X$ and $1\leq i\leq k$, if $q\in A_i$ then $|f_i(q)-g(q)|<\epsilon$. By choice of $\cF$ and $g$, this will finish the proof.

For each $q \in X$, we apply the claim to find $f_q \in \cF$ such that $|f_q(x)-g(q)| < \frac{1}{3}\e$. Since $g$ and $f_q$ are continuous, we can also find a clopen neighborhood $B_q$ of $q$ such that for any $p \in B_q$, $|f_q(p)-f_q(q)|< \frac{1}{3}\e$ and $|g(p)-g(q)| < \frac{1}{3}\e$. In particular, this implies that for any $p \in B_q$,
    \[
    |f_q(p) - g(p)| \leq |f_q(p)-f_q(q)|+|f_q(q)-g(q)|+|g(q)-g(p)| < \e.
    \]
    
    By compactness, we can find a finite sequence $q_1,\dots,q_k$ such that $B_{q_1},\dots,B_{q_{k}}$ covers $X$. Setting $f_i = f_{q_i}$ and $A_i = B_{q_i} \backslash \bigcup_{j < i} B_{q_j}$, we have the required functions and partition (after possibly decreasing $k$ and discarding any empty $A_i$).
\end{proof}

Note that $\mu\in\kM_x(\cU)$ is \fam\ over $M\prec\cU$ if and only if it satisfies the conditions of the previous proposition with $k=1$.

Next we review basic properties about approximations.

\begin{definition}
Given a measure $\mu\in\kM_x(\cU)$, an $\cL_{\cU}$-formula $\phi(x,y)$, an integer $n\geq 1$, and some $C\seq[0,1]$, let $\xbar=(x_1,\ldots x_n)$ and define 
\[
\Avd{n}{C}{\mu}{\phi} = \{\abar\in \cU^{\xbar}:|\mu(\phi(x,b))-\Av(\abar)(\phi(x,b))|\in C\text{ for all $b\in\cU^y$}\}.
\]
\end{definition}

Note that a measure $\mu\in\kM_x(\cU)$ is \fam\ over $M\prec\cU$ if and only if, for any $\cL$-formula $\phi(x,y)$ and $\epsilon>0$, there is some $n\geq 1$ such that $\Avo{n}{\epsilon}{\mu}{\phi}\cap (M^x)^n\neq\emptyset$.

\begin{lemma}\label{lem:famints}
Suppose $\mu\in\kM_x(\cU)$ is Borel definable and $\varphi(x,y)$ is an $\cL_{\cU}$-formula. Then for any $\abar\in \Avo{n}{\epsilon}{\mu}{\phi}$ and any $\nu\in\kM_y(\cU)$, we have
\[
(\mu\otimes\nu)(\phi(x,y))\approx_\epsilon (\Av(\abar)\otimes\nu)(\phi(x,y))=\frac{1}{n}\sum_{n=1}^n\nu(\phi(a_i,y)).
\]
\end{lemma}
\begin{proof}
This is a straightforward calculation (integrate over $S_y(M)$, where $M\prec\cU$ is such that $\phi(x,y)$ is over $M$, $\mu$ is Borel definable over $M$, and $\abar\in M^{\xbar}$).
\end{proof}

Next we work toward a characterization of \fim\ (Proposition \ref{prop:fimtest} below), which will be useful in several later results. Recall that a set (in $\cU$) is \emph{co-type-definable} if its complement is type-definable. 

\begin{lemma}\label{lem:defapprox}
Suppose $\mu\in\kM_x(\cU)$ is definable over $A\subset\cU$. Then, for any $\cL_A$-formula $\phi(x,y)$ and any $n\geq 1$ and $\epsilon>0$, $\Avc{n}{\epsilon}{\mu}{\phi}$ is type-definable over $A$, and $\Avo{n}{\epsilon}{\mu}{\phi}$ is  co-type-definable over $A$.  
\end{lemma}
\begin{proof}
Fix $\phi(x,y)$, $n\geq 1$, and $\epsilon>0$. For any closed set $C\seq[0,1]$, we have a well-defined partial $y$-type ``$\mu(\phi(x,y))\in C$" over $A$. 
 Let $q_C(\xbar,y)$ be the type defined by
\[
\bigwedge_{i=0}^n\left((\Av(\xbar)(\phi(x,y))=\textstyle\frac{i}{n})\rightarrow \left( \mu(\phi(x,y))\in C+\textstyle\frac{i}{n}\right)\vee\left(\mu(\phi(x,y))\in \textstyle\frac{i}{n}-C\right)\right).
\]
Then $(\abar,b)\models q_C$ if and only if $|\mu(\phi(x,b))-\Av(\abar)(\phi(x,b))|\in C$. Now $\Avc{n}{\epsilon}{\mu}{\phi}$ is defined by $\forall y\,q_{[0,\epsilon]}(\xbar,y)$, and $\neg \Avo{n}{\epsilon}{\mu}{\phi}$ is defined by $\exists y\,q_{[\epsilon,1]}(\xbar,y)$, both of which are types over $A$.
\end{proof}

\begin{definition}
Suppose $\mu\in\kM_x(\cU)$ and $\phi(x,y)$ is an $\cL$-formula. Given $\epsilon>0$, we say that a sequence $(\chi_n(x_1,\ldots,x_n))_{n=1}^\infty$ of $\cL_{\cU}$-formulas is a \textbf{$(\phi,\epsilon)$-approximation sequence for $\mu$} if, for all $n\geq 1$, we have $\Avc{n}{\epsilon/2}{\mu}{\phi}\seq \chi_n(\cU^{\xbar})\seq \Avo{n}{\epsilon}{\mu}{\phi}$. We also say that such a sequence is \textbf{over $A\subset\cU$} if each $\chi_n$ is an $\cL_A$-formula.
\end{definition}

Note that, by Lemma \ref{lem:defapprox}, if $\mu\in\kM_x(\cU)$ is definable over $A\subset\cU$, then for any $\cL$-formula $\phi(x,y)$ and $\epsilon>0$, there is a $(\phi,\epsilon)$-approximation sequence for $\mu$ over $A$ (but the formulas in the sequence may be unsatisfiable). 

\begin{proposition}\label{prop:fimtest}
Suppose $\mu\in\kM_x(\cU)$ is definable over $M\prec\cU$. Then the following are equivalent.
\begin{enumerate}[$(i)$]
    \item $\mu$ is \fim\ over $M$.
    \item For any $\cL$-formula $\phi(x,y)$ and $\epsilon>0$, there is a $(\phi,\epsilon)$-approximation sequence $(\chi_n)_{n=0}^\infty$ for $\mu$ over $M$ such that $\lim_{n\to\infty}\mu^{(n)}(\chi_n)=1$.
    \item For any $\cL$-formula $\phi(x,y)$ and $\epsilon>0$, if $(\chi_n)_{n=0}^\infty$ is a $(\phi,\epsilon)$-approximation sequence  for $\mu$ over $M$, then $\lim_{n\to\infty}\mu^{(n)}(\chi_n)=1$.
\end{enumerate}
\end{proposition}
\begin{proof}
$(i)\Rightarrow (iii)$. Assume $\mu$ is \fim\ over $M$. Fix an $\cL$-formula $\phi(x,y)$. Since $\mu$ is \fim, there are formulas $(\theta_n(x_1,\ldots,x_n))_{n=1}^\infty$ such that $\lim_{n\to\infty}\mu^{(n)}(\theta_n)=1$ and, for all $\epsilon>0$, we have $\theta_n(\cU^{\xbar})\seq \Avc{n}{\epsilon/2}{\mu}{\phi}$ for sufficiently large $n$. Now fix $\epsilon>0$, and let $(\chi_n)_{n=1}^\infty$ be a $(\phi,\epsilon)$-approximation sequence for $\mu$ over $M$. Then $\theta_n(\cU^{\xbar})\seq \Avc{n}{\epsilon/2}{\mu}{\phi}\seq \chi_n(\cU^{\xbar})$ for sufficiently large $n$, and so $\lim_{n\to\infty}\mu^{(n)}(\chi_n)=1$.  

$(iii)\Rightarrow (ii)$ is trivial since approximation sequences for $\mu$ exist by Lemma \ref{lem:defapprox}.

$(ii)\Rightarrow (i)$. Assume $\mu$ satisfies $(ii)$, and fix an  $\cL$-formula $\phi(x,y)$. For all $i\geq 1$, we have an $\cL_M$-formula $\chi_{n_i}(x_1,\ldots,x_{n_i})$ such that $\chi_{n_i}(\cU^{x_1\ldots x_{n_i}})\seq \Avo{n_i}{1/i}{\mu}{\phi}$ and $\mu^{(n_i)}(\chi_{n_i})\geq 1-1/i$. This suffices to prove that $\mu$ is \fim\ over $M$ (similar to the proof of \cite[Proposition 3.2]{CoGa}).
\end{proof}

We now summarize the situation concerning the analogue of Corollary \ref{cor:BDdrop} for various properties of measures.

\begin{proposition}\label{prop:alldrop}
Suppose $\mu\in\kM_x(\cU)$ is \fim\ (resp., \fam, finitely satisfiable in some small model, definable, or Borel definable), and also invariant over $M\prec\cU$. Then $\mu$  is \fim\ over $M$ (resp., \fam\ over $M$, finitely satisfiable in $M$, definable over $M$, or Borel definable over $M$). 
\end{proposition}
\begin{proof}
Corollary \ref{cor:BDdrop} provides the Borel definable case, and the definable case is similar (see Remark \ref{rem:BDdrop}). The finitely satisfiable case is a straightforward modification of the proof for types, e.g., as in \cite[Lemma 2.18]{Sibook}. See  \cite[Proposition 2.18]{Ganthesis} for details. 
It remains to consider \fim\ and \fam.

Suppose $\mu$ is \fam. Fix an $\cL$-formula $\phi(x,y)$ and $\epsilon>0$. We want to find $n\geq 1$ such that $\Avo{n}{\epsilon}{\mu}{\phi}\cap M^{\xbar}\neq\emptyset$. By assumption, there is some $n\geq 1$ and $\abar^*\in \Avc{n}{\epsilon/2}{\mu}{\phi}\cap \cU^{\xbar}$. Since $\mu$ is definable and $M$-invariant, it is definable over $M$. So $\Avc{n}{\epsilon/2}{\mu}{\phi}$ is type-definable over $M$, and contained in $\Avo{n}{\epsilon}{\mu}{\phi}$, which is co-type-definable over $M$. Therefore we may find an $\cL_M$-formula $\chi(x_1,\ldots,x_n)$ such that $\Avc{n}{\epsilon/2}{\mu}{\phi}\seq \chi(\cU^{\xbar})\seq \Avo{n}{\epsilon}{\mu}{\phi}$. Then $\cU\models\chi(\abar^*)$ and so, since $M\prec \cU$, there is $\abar\in (M^x)^n$ such that $M\models\chi(\abar)$. So $\abar\in \Avo{n}{\epsilon}{\mu}{\phi}\cap M^{\xbar}$. 

Finally, suppose $\mu$ is \fim. Fix an $\cL$-formula $\phi(x,y)$ and $\epsilon>0$. Then there is a $(\phi,\epsilon/2)$-approximation sequence  $(\chi^*_n)_{n=0}^\infty$ for $\mu$ such that $\lim_{n\to\infty}\mu^{(n)}(\chi^*_n)=1$. As before, $\mu$ is definable over $M$. Let $(\chi_n)_{n=0}^\infty$ be  $(\phi,\epsilon)$-approximation sequence for $\mu$ over $M$. Then $\chi^*_n(\cU^{\xbar})\seq \chi_n(\cU^{\xbar})$ for all $n\geq 1$, which implies $\lim_{n\to\infty}\mu^{(n)}(\chi_n)=1$. So $\mu$ is \fim\ over $M$ by Proposition \ref{prop:fimtest}. 
\end{proof}

Next we summarize what is known about the preservation of various properties with respect to Morley products.

\begin{proposition}\label{prop:MPbasic}
Fix $\mu\in\kM_x(\cU)$,  $\nu\in\kM_y(\cU)$, and $M\prec\cU$.
\begin{enumerate}[$(a)$]
\item If $\mu$ and $\nu$ are \fam\ (resp., definable) over $M$, then $\mu\otimes\nu$ is \fam\ (resp., definable) over $M$.
\item $\mu$ and $\nu$ are finitely satisfiable in $M$, and $\mu$ is Borel definable or $\nu$ is a type, then $\mu\otimes\nu$ is finitely satisfiable in $M$. 
\end{enumerate}
\end{proposition}
\begin{proof}
Part $(a)$. See Propositions 2.10 and 2.6 of \cite{CoGa}.

Part $(b)$. See \cite[Lemma 1.6]{HPS}. The authors there do not explicitly assume $\mu$ is Borel definable (which is needed for $\mu\otimes\nu$ to be well-defined). Instead, they assume $T$ is NIP so that this becomes automatic. See also  \cite[Proposition 2.25]{Ganthesis}.  
\end{proof} 

In contrast to the previous result, we have already seen that Borel definable measures are not necessarily closed under Morley products. In \cite{CoGa} it is claimed that this is also the case for \fim\ measures, due to an example from \cite{ACPgs}. However, that example turns out not to work (see Section \ref{sec:PER}) and it remains an open question whether \fim\ measures are closed under Morley products (see the end of Section \ref{sec:fim} for further discussion).

Finally, in preparation for the main result in Section \ref{sec:fim}, we make some easy observations about convex combinations. 

\begin{proposition}\label{prop:CCbasic}
Fix $M\prec\cU$, and suppose $\mu,\nu\in\kM_x(\cU)$ are \fam\ over $M$ (resp., finitely satisfiable in $M$, definable over $M$, or Borel definable over $M$). Then, for any $r\in[0,1]$, $\lambda\coloneqq r\mu+(1-r)\nu$ is \fam\ over $M$ (resp, finitely satisfiable in $M$, definable over $M$, or Borel definable over $M$).
\end{proposition}
\begin{proof}
Fix a formula $\phi(x,y)$. Note that, in any case, $\mu$ is invariant over $M$.  If  $\mu$ and $\nu$ are definable (resp., Borel definable), then $F^\phi_{\mu,M}$ and $F^\phi_{\nu,M}$ are continuous (resp., Borel), and so $F^\phi_{\lambda,M}=rF^\phi_{\mu,M}+(1-r)F^\phi_{\nu,M}$ is continuous (resp., Borel), which implies $\lambda$ is definable (resp., Borel definable) with respect $\phi(x,y)$. Also, if $\mu$ and $\nu$ are both finitely satisfiable in $M$, then it is clear that $\lambda$ is too. 

Finally, suppose $\mu$ and $\nu$ are \fam\ over $M$. Fix $\epsilon>0$. Then there are $\abar,\bbar\in (M^x)^{<\omega}$ such that $\mu\approx^\phi_\epsilon \Av(\abar)$ and $\nu\approx^\phi_\epsilon \Av(\bbar)$. Let $\eta=r\Av(\abar)+(1-r)\Av(\bbar)$. Then $\lambda\approx^\phi_{\epsilon} \eta$, and it is easy to find some $\cbar\in (M^x)^{<\omega}$ such that $\eta\approx^\phi_\epsilon \Av(\cbar)$. So $\lambda\approx^\phi_{2\epsilon}\Av(\cbar)$. This shows $\lambda$ is \fam\ over $M$. 
\end{proof}

Once again, \fim\ measures are missing from the previous result. We will show in Theorem \ref{thm:fimCC} that \fim\ measures are also closed under convex combinations.  

\section{Commuting measures}\label{sec:com}
Let $T$ be a complete $\cL$-theory with monster model $\cU$.
In this section, we investigate pairs  of measures that commute. Let us start with some results from the literature. The first is an easy exercise. 

\begin{proposition}\label{prop:dfscom}
If $p\in S_x(\cU)$ is definable and $\nu\in \kM_y(\cU)$ is finitely satisfiable in some small model, then $p\otimes \nu=\nu\otimes p$.
\end{proposition}
\begin{proof}
The argument is similar to that of \cite[Lemma 3.4]{HP} (which assumes $\nu$ is a type). Fix an $\cL_{\cU}$-formula $\phi(x,y)$ and let $M\prec\cU$ be such that $\phi(x,y)$ is over $M$, $p$ is definable over $M$, and $\nu$ is finitely satisfiable in $M$. Choose an $\cL_M$-formula $\psi(y)$ such that $\phi(x,b)\in p$ if and only if $\cU\models\psi(b)$.  Let $a\models p|_M$. Then we have $\phi(a,M^y)=\psi(M^y)$, and so $\nu(\phi(a,y)\smd \psi(y))=0$ since $\nu$ is finitely satisfiable in $M$. Therefore $(p\otimes\nu)(\phi(x,y))=\nu(\psi(y))=\nu(\phi(a,y))=(\nu\otimes p)(\phi(x,y))$. 
\end{proof}

An obvious question is whether the previous result holds when $p$ is replaced by a definable global measure (and $\nu$ is also Borel definable so that both products are well-defined). We will show later on that this is not the case (see Example \ref{ex:com1}). However, Hrushovski, Pillay, and Simon \cite{HPS} proved the following generalization and elaboration of Proposition \ref{prop:dfscom} in the setting of  NIP theories.

\begin{theorem}[Hrushovski, Pillay, Simon \cite{HPS}]\label{thm:HPSc}
Assume $T$ is NIP. 
\begin{enumerate}[$(a)$]
\item If $\mu\in\kM_x(\cU)$ is definable, and $\nu\in\kM_y(\cU)$ is finitely satisfiable in some small model, then $\mu\otimes\nu=\nu\otimes\mu$. 
\item If $\mu\in\kM_x(\cU)$ is \dfs, then $\mu\otimes\nu=\nu\otimes\mu$ for any invariant $\nu\in\kM_y(\cU)$. 
\item If $\mu\in\kM_x(\cU)$ is invariant, then it is \dfs\ if and only if $\mu_x\otimes\mu_{x'}=\mu_{x'}\otimes\mu_{x}$. 
\end{enumerate}
\end{theorem}

In this section, we pursue  results along the lines of adapting Theorem \ref{thm:HPSc} to arbitrary theories. First, we briefly note that outside of NIP, self-commuting measures (in the sense of Theorem \ref{thm:HPSc}$(c)$) need not have any special properties.

\begin{example}
Let $T$ be the theory of the random graph. Then any invariant global type in a \emph{one} free variable commutes with itself. On other hand, for any $M\prec\cU$ and $Z\seq S_1(M)$, there is a unique non-algebraic $M$-invariant type $p_Z\in S_1(\cU)$ such that $E(x,b)\in p_Z$ if and only if $\tp(b/M)\in Z$. So $p_Z$ is Borel definable (resp., definable) if and only if $Z$ is Borel (resp., clopen). Note also that the `generic' definable types $p_\emptyset$ and $p_{S_1(M)}$ are not finitely satisfiable in $M$. In fact, $T$ has no nontrivial \dfs\ global measures (see \cite[Theorem 4.9]{CoGa}).  
\end{example}

The first goal this section is a suitable generalization of Theorem \ref{thm:HPSc}$(a)$ for arbitrary theories. The original proof of this theorem relied on a fundamental property of measures in the NIP setting, namely, that any Keisler measure can be \textit{locally uniformly approximated} by averaging on a finite collection of types in the support of the given measure. Using this, the authors of \cite{HPS} were able to reduce the problem of whether a finitely satisfiable \textit{measure} commutes with a definable measure to the question of whether a finitely satisfiable \textit{type} commutes with a definable measure (note the duality to Proposition \ref{prop:dfscom} in this statement). That being said, the proof of this `easier' problem remained nontrivial and still required the use of NIP, along with the weak law of large numbers.  Unfortunately there are two major obstacles one finds when trying to directly adapt this proof of Theorem \ref{thm:HPSc}$(a)$ to the general setting. First, Keisler measures in the wild do not admit approximations by types as discussed above. Secondly, and more importantly, the statement in total generality is false. Indeed, Proposition \ref{prop:nocom} gives an example of  a $\dfs$ type and a definable measure that do not commute. Therefore the dual version of Proposition \ref{prop:dfscom} alluded to above fails outside of NIP.

 Fortunately however, one can give a simpler proof of Theorem \ref{thm:HPSc}$(a)$ in the NIP context by treating smooth extensions of measures as analogous to realizations of types, along with some elementary topology (see \cite[Proposition 3.6]{GanSA}). By embracing this ideology, and widening the focus to commuting extensions of measures, we will recover a `deviant' generalization of Theorem \ref{thm:HPSc}$(a)$, which applies  to general theories and has a purely topological proof. This generalization is given in  Theorem \ref{thm:gencom} below. We start with some topological lemmas.

 Recall that $\kM_x(\cU)$ is a compact Hausdorff space under the subspace topology induced from $[0,1]^{\Def_x(\cU)}$. 
 
 \begin{lemma}\label{lem:defcont}
 Suppose $\mu\in\kM_x(\cU)$ is definable. Then for any $\cL_{\cU}$-formula $\phi(x,y)$, the map $\nu\mapsto (\mu\otimes\nu)(\phi(x,y))$ is continuous from $\kM_y(\cU)$ to $[0,1]$. 
 \end{lemma}
 \begin{proof}
 This involves similar calculations as in the proofs of \cite[Proposition 6.3]{ChGan} and \cite[Proposition 2.6]{CoGa}. Fix an $\cL_{\cU}$-formula $\phi(x,y)$, and fix $A\subset\cU$ such that $\phi(x,y)$ is over $A$ and $\mu$ is definable over $A$. Fix $\epsilon>0$. By Fact \ref{fact:defdefs}, there are $\cL_A$-formulas $\psi_1(y),\ldots ,\psi_n(y)$ and $r_1,\ldots,r_n\in [0,1]$ such that $\|F_{\mu,A}^{\varphi} - \sum_{i=1}^{n} r_i \mathbf{1}_{\psi_{i}(y)}\|_\infty < \epsilon$. 
Hence, for arbitrary $\nu \in \kM_{y}(\mathcal{U})$, we have 
\[
(\mu \otimes \nu)(\varphi(x,y)) = \int_{S_{y}(A)} F_{\mu}^{\varphi}\, d\nu \approx_{\epsilon} \int_{S_{y}(A)} \sum_{i=1}^{n} \mathbf{1}_{\psi_{i}(y)}\, d\nu = \sum_{i=1}^{n} r_i \nu(\psi_{i}(y)). 
\]
By definition of the topology on $\kM_y(\cU)$, the map $\nu \mapsto \nu(\psi_{i}(y))$ is continuous. So the map $\nu \mapsto \sum_{i=1}^{n} r_i \nu(\psi_{i}(y))$ is also continuous since it is a linear combination of continuous functions. Since $\nu$ was arbitrary we have
\[
\sup_{ \nu \in \kM_{y}(\mathcal{U})} \left|(\mu \otimes \nu)(\varphi(x,y)) - \sum_{i=1}^{n} r_i \nu(\psi_{i}(y))\right| < \epsilon. 
\]
Therefore $\nu \mapsto (\mu \otimes \nu)(\varphi(x,y))$ is a uniform limit of continuous functions, and hence is continuous. 
\end{proof}

  In the subsequent results, we will consider pairs of global measures in the same variable sort, which have the same restriction to some small model. Thus we take a moment to point out various subtleties that arise. In particular, suppose $\mu\in\kM_x(\cU)$ is Borel definable over $A\subset\cU$, and $\nu,\hat{\nu}\in\kM_y(\cU)$ are such that $\nu|_A=\hat{\nu}|_A$. Then we trivially have $(\mu\otimes\nu)|_A=(\mu\otimes\hat{\nu})|_A$.  But note that this can fail if $\mu$ is only Borel definable over some larger $B\supseteq A$, since in this case the Morley products with $\mu$ must be computed with respect to $\nu|_B$ and $\hat{\nu}|_B$ (even when applied to $\cL_A$-formulas). It is also important to point out that if we instead have $\mu,\hat{\mu}\in\kM_x(\cU)$ Borel definable over $A$, with $\mu|_A=\hat{\mu}|_A$, then one \emph{cannot} necessarily conclude $(\mu\otimes\nu)|_A=(\hat{\mu}\otimes\nu)|_A$ for a given $\nu\in\kM_y(\cU)$. 
 
 \begin{definition}
 Suppose $\mu\in\kM_x(\cU)$ is Borel definable over $A\subset\cU$. Then a Borel definable measure $\nu\in\kM_y(\cU)$ \textbf{$A$-commutes} with $\mu$ if $(\mu\otimes\nu)|_A=(\nu\otimes\mu)|_A$. Define $C^\mu_y(A)$ to be the set of measures in $\kM_y(\cU)$ that are Borel definable over $A$ and $A$-commute with $\mu$.
\end{definition}

As we will see below,  Theorem \ref{thm:HPSc}$(a)$ can be viewed as  a question of when $C^\mu_y(A)$ contains certain limit points.  The next lemma describes a technical scenario in which this can happen.
 
 \begin{lemma}\label{lem:comtech}
 Suppose $\mu\in\kM_x(\cU)$ is definable over $A$, and $\nu\in\kM_y(\cU)$ is a Borel definable measure, which is the limit of a net $(\nu_i)_{i\in I}$ from $C^\mu_y(A)$. If $\mu|_A$ has a definable global extension $\hat{\mu}\in\kM_x(\cU)$, which $A$-commutes with $\nu$ and $\nu_i$ for all $i\in I$, then $\nu\in C^\mu_y(A)$.
 \end{lemma}
  \begin{proof}
  We first note that $\nu$ is $A$-invariant, and thus Borel definable over $A$ by Corollary \ref{cor:BDdrop}. 
Now fix an $\cL_A$-formula $\phi(x,y)$. Then we have the following calculations (individual steps are justified afterward):
 \begin{multline*}
 (\mu\otimes\nu)(\phi(x,y))=\lim_{i\in I}(\mu\otimes\nu_i)(\phi(x,y))=\lim_{i\in I}(\nu_i\otimes\mu)(\phi(x,y))\\
 =\lim_{i\in I}(\nu_i\otimes\hat{\mu})(\phi(x,y))=\lim_{i\in I}(\hat{\mu}\otimes\nu_i)(\phi(x,y))=(\hat{\mu}\otimes\nu)(\phi(x,y))\\
 = (\nu\otimes\hat{\mu})(\phi(x,y))=(\nu\otimes\mu)(\phi(x,y)).
 \end{multline*}
 The first and fifth equalities above use Lemma \ref{lem:defcont}; the second equality uses the assumption that $\nu_i$ is in $C^\mu_y(A)$; the third and seventh equalities use $\mu|_A=\hat{\mu}|_A$; and the fourth and sixth equalities use the commutativity assumptions on $\hat{\mu}$.
 \end{proof}

 Note that in the statement of Lemma \ref{lem:comtech}, we do not need to assume that $\hat{\mu}$ is definable \emph{over $A$}. For example, if $\mu$ is a type then such a $\hat{\mu}$ exists as in Remark \ref{rem:thirdtype}. So we see  that if $p\in S_x(\cU)$ is definable over $A\subset\cU$, then the set of $A$-invariant measures in $\kM_y(\cU)$ that $A$-commute with $p$ is closed (as usual, when working with types, Borel definability can be weakened to invariance). However, when $\mu$ is a measure, the existence of $\hat{\mu}$ as in Lemma \ref{lem:comtech} is a nontrivial assumption. 
 
 We can now prove a generalization of Theorem \ref{thm:HPSc}$(a)$ for arbitrary theories.
 
 \begin{theorem}\label{thm:gencom}
 Suppose $\mu\in\kM_x(\cU)$ is definable over $M\prec\cU$, and $\nu\in\kM_y(\cU)$ is Borel definable and finitely satisfiable in $M$. If $\mu|_M$ has a  definable global extension that $M$-commutes with $\nu$, then $(\mu\otimes\nu)|_M=(\nu\otimes\mu)|_M$.
 \end{theorem}
 \begin{proof}
 Let $X$ denote the convex hull of $\{\delta_a:a\in M^y\}$ in $\kM_y(\cU)$. Then it is not hard to show that $\nu$ is in the closure of $X$ (see also \cite[Proposition 2.11]{ChGan}). Moreover, by an easy calculation, a measure in $X$ commutes with \emph{every} invariant measure. Thus the hypotheses of Lemma \ref{lem:comtech} are satisfied, and so we have $\nu\in C^\mu_y(M)$.
 \end{proof}

 \begin{remark}\label{rem:hpsca}
Theorem \ref{thm:HPSc}$(a)$ is a consequence of Theorem \ref{thm:gencom} together with the fundamental results on NIP theories discussed before Corollary \ref{cor:NIPassoc}. Indeed, suppose $T$ is NIP, $\mu\in\kM_x(\cU)$ is definable, and $\nu$ is finitely satisfiable in some $M\prec\cU$. Without loss of generality, $\mu$ is definable over $M$. Moreover, $\nu$ is $M$-invariant and hence Borel definable over $M$. Finally, $\mu|_M$ has a global extension that is smooth, and thus is definable and commutes with $\nu$. Since all of this works over any $N\succeq M$, we have $\mu\otimes\nu=\nu\otimes\mu$ by Theorem \ref{thm:gencom}. 
 \end{remark}

In light of Proposition \ref{prop:dfscom}, it is natural to ask whether the assumption on $\mu|_M$ in Theorem \ref{thm:gencom} is necessary. A counterexample, which we only mention now, will be given later in the paper.

\begin{example}\label{ex:com1}
There is a complete theory $T$, a definable measure $\mu\in\kM_x(\cU)$ and a \dfs\ type $q\in S_y(\cU)$, such that $\mu\otimes q\neq q\otimes\mu$.  See Section \ref{sec:com1} for details.
\end{example}

On the other hand, the following question remains open.

\begin{question}\label{ques:dfscom}
Do any two \dfs\ global measures commute? (Note that for types this is a special case of Proposition \ref{prop:dfscom}.)
\end{question}

The next goal of this section is to show that \fim\ measures commute with Borel definable measures. In other words, Theorem \ref{thm:HPSc}$(b)$ generalizes to arbitrary theories, provided that \dfs\ is replaced by \fim\ (which is equivalent in NIP). This result also generalizes the easier fact that \emph{smooth} measures commute with Borel definable measures \cite{HPS}. In analogy to the comparison between Proposition \ref{prop:dfscom} and Theorem \ref{thm:gencom}, we will also see that \fim\ \emph{types} commute with \emph{invariant} measures. However, in this case the overall structure of the proof for types is not that much different than for measures. So to avoid repetitive arguments, we will use the relative notion of measurability from Section \ref{sec:bern}.

\begin{definition}\label{def:mumeas}
Suppose $\mu\in\kM_x(\cU)$ is invariant over $A\subset\cU$,  and $\nu\in\kM_y(\cU)$. Then $\mu$ is \textbf{$\nu$-measurable over $A$} if, for any $\cL_A$-formula $\phi(x,y)$, the map $F^\phi_{\mu,A}$ is $\nu|_A$-measurable, i.e., $(F^\phi_{\mu,A})\inv(U)$ is $\nu|_A$-measurable for any open $U\seq[0,1]$.
\end{definition}

Let us note the two examples of interest.

\begin{example}\label{ex:mu-meas} Fix $\mu\in\kM_x(\cU)$ and $A\subset\cU$.
\begin{enumerate}[$(1)$]
\item If $\mu$ is Borel definable over $A$ then it is $\nu$-measurable over $A$ for any $\nu\in\kM_y(\cU)$.
\item If $\mu$ is invariant over $A$ then it is $q$-measurable over $A$ for any $q\in S_y(\cU)$. 
\end{enumerate}
\end{example}

Suppose $\mu\in\kM_x(\cU)$ is invariant over $A\subset\cU$, and $\nu\in\kM_y(\cU)$, and $\mu$ is $\nu$-measurable over $A$. Given an $\cL_A$-formula $\phi(x,y)$, we set $(\mu\otimes_A\nu)(\phi(x,y))=\int_{S_y(A)}F^\phi_{\mu}\, d\nu$. This yields a well-defined Keisler measure $\mu\otimes_A\nu$ in $\kM_{xy}(A)$.  Note that if $\mu$ is either Borel definable over $A$, or $\nu$ is a type in $S_y(\cU)$, then $\mu\otimes_A\nu=(\mu\otimes\nu)|_A$. However, unlike the situation with Borel definability, it is possible for a measure $\mu$ to be $\nu$-measurable over some $A\subset\cU$, but not $\nu$-measurable over any proper $B\supset A$ (see Proposition \ref{prop:bern} for an example). So we may not have a well-defined global product  $\mu\otimes\nu$.

Next, we recall the weak law of large numbers (a special case of Chebyshev's inequality). Our formulation of this result follows \cite[Proposition B.4]{Sibook}, except we have sharpened the bound slightly (in a way that is evident from how Chebyshev is applied).

\begin{fact}\label{fact:wlln}
Suppose $(\Omega,\cB,\mu)$ is a probability space, and fix $X\in\cB$ and $\epsilon>0$. Given $n\geq 1$, let $\mu^n$ denote the usual product measure on $\Omega^n$, and let 
\[
X_{n,\epsilon}=\{\abar\in \Omega^n:\mu(X) \approx_\epsilon\Av(\abar)(X)\}.
\]
Then $X_{n,\epsilon}$ is $\mu^n$-measurable, and $\mu^n(X_{n,\epsilon})\geq 1-\frac{\mu(X)(1-\mu(X))}{\epsilon^2n}$. 
\end{fact}

The next lemma uses Fact \ref{fact:wlln} to highlight the leverage obtained when working with \fim\ measures. This distinction is further discussed after Theorem \ref{thm:fimcom}.

\begin{lemma}\label{lem:fimcom}
Suppose $\mu\in \kM_x(\cU)$ is \fim\ over $M\prec\cU$. Fix an $\cL$-formula $\phi(x,y)$ and $\mu|_M$-measurable sets $X_1,\ldots,X_n\seq S_x(M)$. Then for any $\epsilon>0$, there is an integer $k\geq 1$ and  a sequence $(a_1,\ldots,a_k)\in(\cU^x)^k$ such that $\mu\approx_\epsilon^\phi \Av(\abar)$ and $\mu|_M(X_i)\approx_\epsilon \Av(\abar)|_M(X_i)$ for all  $1\leq i\leq n$.
\end{lemma}
\begin{proof}
The argument is similar to various parts of Section 3 in \cite{HPS} (see, e.g., \cite[Lemma 3.6]{HPS}). 
Fix $\epsilon>0$. Choose $\cL_M$-formulas $\theta_k(x_1,\ldots,x_k)$ such that $\lim_{k\to\infty}\mu^{(k)}(\theta_k)=1$ and, for $k$ sufficiently large, if $\theta_k(\abar)$ holds then $\mu\approx_\epsilon^\phi \Av(\abar)$. 
For $1\leq i\leq n$ and $k\geq 1$, define
\begin{align*}
X_{i,k} &= \{(p_1,\ldots,p_k)\in S_x(M)^k:\mu|_M(X_i)\approx_\epsilon \Av(\bar{p})(X_i)\},\text{ and }\\
Y_{i,k} &= \{p\in S_{x_1,\ldots,x_k}(M):(p|_{x_1},\ldots,p|_{x_k})\in X_{i,k}\}.
\end{align*}
Then each set $Y_{i,k}$ is $\mu^{(k)}|_M$-measurable and $\mu^{(k)}|_M(Y_{i,k})=(\mu|_M)^k(X_{i,k})$.  
So we have $\lim_{k\to\infty}\mu^{(k)}(Y_{i,k})=1$ by Fact \ref{fact:wlln}. Choose $k$ large enough so that $\mu^{(k)}(\theta_k)$, $\mu^{(k)}(Y_{1,k}),\ldots,\mu^{(k)}(Y_{n,k})$ are each strictly greater than $\frac{n}{n+1}$. Then there is some $p\in [\theta_k]\cap Y_{1,k}\cap\ldots\cap Y_{n,k}$. Let $\abar\in\cU^k$ realize $p$. Then $\abar$ satisfies the desired conditions. 
\end{proof}

We now prove a proposition that provides the heart of the result that \fim\ measures commute with Borel definable measures.

\begin{proposition}\label{prop:fimcom}
Suppose $\mu\in\kM_x(\cU)$ is \fim\ over $M\prec\cU$ and $\nu\in\kM_y(\cU)$ is $\mu$-measurable over $M$. Then $(\mu\otimes\nu)|_M=\nu\otimes_M\mu$.
\end{proposition}
\begin{proof}
Fix an $\cL_M$-formula $\phi(x,y)$ and some $\epsilon>0$. Let $\varphi^*(y,x)$ denote the same formula $\varphi(x,y)$, but with the roles of object and parameter variables exchanged. Since $F^{\phi^*}_{\nu,M}$ is bounded and $\mu|_M$-measurable, it can be approximated uniformly by simple $\mu|_M$-measurable functions (see Fact \ref{fact:Bapprox}). So there are $\mu|_M$-measurable sets $X_1,\ldots,X_n\seq S_x(M)$, and $r_1,\ldots,r_n\in [0,1]$ such that $\|F^{\phi^*}_{\nu,M} -\sum_{i=1}^n r_i\boldsymbol{1}_{X_i}\|_\infty<\epsilon$. By Lemma \ref{lem:fimcom}, there is some $\abar\in (\cU^x)^k$ such that $\mu\approx^\phi_{\epsilon}\Av(\abar)$ and $\mu(X_i)|_M\approx_{\epsilon/n}\Av(\abar)|_M(X_i)$ for all $1\leq i\leq n$. Let $p_j=\tp(a_j/M)$. Then
\begin{multline*}
(\mu\otimes\nu)(\phi(x,y)) \approx_\epsilon (\Av(\abar)\otimes\nu)(\phi(x,y))=\frac{1}{k}\sum_{j=1}^k\nu(\phi(a_j,y))=\frac{1}{k}\sum_{j=1}^kF^{\phi^*}_{\nu,M}(p_j)\\
\approx_\epsilon \frac{1}{k}\sum_{j=1}^k\sum_{i=1}^nr_i\boldsymbol{1}_{X_i}(p_j)
= \frac{1}{k}\sum_{i=1}^n\sum_{j=1}^k r_i\delta_{a_j}|_M(X_i)
=\sum_{i=1}^n r_i\Av(\abar)|_M(X_i)\\
 \approx_\epsilon \sum_{i=1}^n r_i\mu|_M(X_i)=\int_{S_x(M)}\sum_{i=1}^n r_i\boldsymbol{1}_{X_i}\, d\mu\approx_\epsilon \int_{S_x(M)} F^{\phi^*}_{\nu}\, d\mu=(\nu\otimes_M\mu)(\phi(x,y)). 
\end{multline*}
Since $\epsilon>0$ was arbitrary, we have the desired result. 
\end{proof}

\begin{theorem}\label{thm:fimcom}$~$
\begin{enumerate}[$(a)$]
\item If $\mu\in\kM_x(\cU)$ is \fim\ and $\nu\in\kM_y(\cU)$ is Borel definable, then $\mu\otimes\nu=\nu\otimes\mu$.
\item If $p\in S_x(\cU)$ is \fim\ and $\nu\in\kM_y(\cU)$ is invariant, then $p\otimes\nu=\nu\otimes p$.
\end{enumerate}
\end{theorem}
\begin{proof}
In light of Example \ref{ex:mu-meas}, this follows immediately from Proposition \ref{prop:fimcom}.
\end{proof}

A natural question is whether  one can get away with using \fam\ measures in place of \fim\ measures in Theorem \ref{thm:fimcom}.  However, note that in the proof of Proposition \ref{prop:fimcom}, we need to approximate $\mu$ simultaneously on instances of the formula $\varphi(x,y)$ \textit{as well as} finitely many Borel sets. \emph{A priori}, the assumption of \fam\ is not sufficient to obtain this level of approximation. That being said, it is perhaps worth emphasizing that this kind of approximation for a measure $\mu$ is all that is needed to prove that $\mu$ commutes with any Borel definable measure. Thus it may be worth pursuing the question of whether this is strictly weaker than the \fim\ assumption.

On the other hand, if in Proposition \ref{prop:fimcom} we restrict to the case where $\nu$ is definable, then we only need to approximate $\mu$ on instances of $\varphi(x,y)$ and finitely many \emph{clopen} sets (i.e., formulas). In this case, such an approximation is possible when $\mu$ is only \fam. These observations yield the next result, which was first shown by the second author using different methods (see \cite[Corollary 3.8]{GanSA}).

\begin{proposition}\label{prop:fdcom}
Suppose $\mu\in\kM_x(\cU)$ is \fam\ and $\nu\in\kM_y(\cU)$ is definable. Then 
$\mu\otimes\nu=\nu\otimes\mu$. 
\end{proposition}
\begin{proof}
Following the proof of Proposition \ref{prop:fimcom}, we can use definability of $\nu$ (and Fact \ref{fact:defdefs}) to ensure each set $X_i$ is clopen (say given by the formula $\psi_i(x)$). Then, since $\mu$ is \fam, we may choose an $\frac{\epsilon}{n}$-appproximation  $(a_1,\ldots,a_k)$ for the  finite set of \emph{formulas} $\{\phi(x,y),\psi_1(x),\ldots,\psi_n(x)\}$. The rest of the calculations are the same.
\end{proof}

  In \cite{CoGa}, it is shown that a type $p\in S_x(\cU)$ is \fim\ over $M\prec\cU$ if and only if it is \textbf{generically stable over $M$}, i.e., $p$ is $M$-invariant and there does not exist an $\cL$-formula $\phi(x,y)$, a sequence $(b_i)_{i<\omega}$ from $\cU^y$, and a Morley sequence $(a_i)_{i<\omega}$ in $p|_M$, such that $\cU\models\phi(a_i,b_j)$ if and only if $i\leq j$.
So we have shown that generically stable types commute with invariant measures. 

\begin{remark}\label{rem:gstype}
There is a more direct proof that generically stable types commute with invariant \emph{types}. For the sake of completeness, we include the argument  (which is similar to the proof that in NIP theories, \dfs\ types commute with invariant types \cite[Proposition 2.33]{Sibook}). Suppose $p\in S_x(\cU)$ is generically stable, and $q\in S_y(\cU)$ is invariant. Fix $\phi(x,y)\in q\otimes p$. Let $M\prec\cU$ be such that $\phi(x,y)$ is over $M$, $p$ is generically stable over $M$, and $q$ is $M$-invariant. Let $(a_i)_{i<\omega}$ be a Morley sequence in $p$ over $M$, and fix $b\models q|_{Ma_{<\omega}}$. Then, for all $i<\omega$, we have $(a_i,b)\models (q\otimes p)|_M$ and so $\phi(a_i,b)$ holds. By generic stability, $p$ coincides with the average type of $(a_i)_{i<\omega}$ (see \cite[Section 3]{CoGa}), and so $\phi(x,b)\in p$. Thus $\phi(x,y)\in p\otimes q$ since $b\models q|_M$.  
\end{remark}

It is natural to ask at this point whether commuting with all Borel definable measures characterizes \fim. This turns out to not be the case.

\begin{example}\label{ex:com2}
There is a complete theory $T$ and a \fam\ (but not \fim) global type that commutes with every invariant measure. See Section \ref{sec:Henson} for details.
\end{example}

On the other hand, we know from Example \ref{ex:com1} that \dfs\ is not sufficient to ensure commuting with Borel definable measures. So this leaves the following questions. 

\begin{question}\label{ques:famcom}
Let $T$ be a complete theory, and fix $\mu\in\kM_x(\cU)$.
\begin{enumerate}[$(1)$]
\item Suppose $\mu$ is \fam. Does $\mu$ commute with every Borel definable measure?
\item Suppose $\mu$ commutes with every Borel definable measure. Is $\mu$ \fam?
\end{enumerate}
\end{question}

\section{Closure properties of \fim\ measures}\label{sec:fim}

In this section, we focus specifically on frequency interpretation measures, which seem to provide the `right' generalization of the notion of generically stable types (recall the discussion before Remark \ref{rem:gstype}) to the setting of Keisler measures. Our goal is to investigate preservation of \fim\ under natural operations on Keisler measures. We first consider convex combinations.

 One fundamental difference between the space of types and the space of Keisler measures is that the latter admits a convex structure. More explicitly, given any two Keisler measures $\mu$ and $\nu$ in $\mathfrak{M}_{x}(\mathcal{U})$ and any real number $r$ in the interval $[0,1]$, one can construct the measure $r\mu + (1-r)\nu\in \kM_x(\cU)$. Thus the question arises as to which collections of measures are preserved under this construction, or which subsets of the space $\mathfrak{M}_{x}(\mathcal{U})$ are \textit{convex}. While it is easily observed that the classes of \dfs\ and \fam\ measures form convex sets (see Proposition \ref{prop:CCbasic}), this property does not obviously hold for the class of \fim\ measures. In this section, we demonstrate that the class of \fim\ measures is also convex. This result provides some fundamental geometric information about the space of \fim\ measures, and also provides a process for making new \fim\ measures from old ones. For example, the average of finitely many \fim\ measures is still \fim. In fact, showing this even in the case of \fim\ types is nontrivial.

The proof that \fim\ is preserved under convex combinations will require the following tail bound for a binomial distribution. Note that, for any real number $r$ and integer $n\geq 1$, we have $\sum_{X\seq[n]}r^{|X|}(1-r)^{n-|X|}=1$. Given $n\geq 1$, $\epsilon>0$, and $r\in[0,1]$, let $\cP_{r,\epsilon}(n)$ denote the collection of subsets $X\seq [n]$ such that $|X|/n\approx_\epsilon r$.

 \begin{fact}\label{fact:binomial}
If $r\in [0,1]$ and $\epsilon>0$ then 
\[
\sum_{X\in \cP_{r,\epsilon}(n)}r^{|X|}(1-r)^{n-|X|}\geq 1-\frac{r(1-r)}{\epsilon^2n}.
\] 
 \end{fact}
 \begin{proof}
 This follows from Chebyshev's inequality applied to the biniomial distribution $B(n,r)$ (the sum above is precisely $P(|B(n,r)-rn|<\epsilon n))$. Equivalently, apply Fact \ref{fact:wlln} with $\Omega=\{0,1\}$, $\cB=\cP(\Omega)$,  $X=\{1\}$, and $\mu(X)=r$. 
 \end{proof}

 \begin{theorem}\label{thm:fimCC}
Suppose $\mu,\nu\in\kM_x(\cU)$ are \fim\ over $M\prec\cU$, and fix $r\in[0,1]$. Then $\lambda=r\mu+(1-r)\nu$ is \fim\ over $M$. 
\end{theorem}
\begin{proof}
Note that $\mu$ and $\nu$ are definable over $M$, and thus so is $\lambda$ by Proposition \ref{prop:CCbasic}.
Fix a formula $\phi(x,y)$. For each $\e > 0$, let $(\chi^{\mu}_{n,\e})_{n=0}^\infty$, $(\chi^{\nu}_{n,\e})_{n=0}^\infty$, and $(\chi^{\lambda}_{n,\e})_{n=0}^\infty$  be $(\phi,\e)$-approximation sequences over $M$ for $\mu$, $\nu$, and $\lambda$, respectively.
By Proposition \ref{prop:fimtest}, we have that for every $\e>0$, $\lim_{n\to \infty}\mu^{(n)} (\chi^{\mu}_{n,\e})= 1$, and likewise for $\nu$. We need to show $\lim_{n\to\infty}\lambda^{(n)}(\chi^\lambda_{n,\epsilon})=1$ for all $\epsilon>0$. So fix some $\epsilon>0$. Without loss of generality, assume $\epsilon<\min\{\frac{r}{2},\frac{1-r}{2}\}$. What we will end up showing is that for any $\delta>0$, there is an integer $n(\delta)\geq 1$ such that, if $n\geq n(\delta)$ then $\lambda^{(n)}(\chi^{\lambda}_{n,6\epsilon})>(1-\delta)^3$. Given this, we can conclude $\lim_{n\to\infty}\lambda^{(n)}(\chi^\lambda_{n,6\epsilon})=1$. Since $\epsilon>0$ is arbitrarily small, this suffices to yield the desired result. 

Given $n\geq 1$ and $X\seq [n]$, let 
   \[
   \lambda_{n,X}=\bigotimes_{i=1}^{n} \left\{\begin{matrix}\mu, & i \in X \\ \nu, & i \notin X \end{matrix}\right\}.
   \]
   Note that $\lambda_{n,X}$ is well-defined by associativity for definable measures.
By linearity of the Morley product, we have that for any $n\geq 1$,
    \[
    \lambda^{(n)} = \sum_{X \subseteq [n]} r^{|X|}(1-r)^{n-|X|} \lambda_{n,X}.
    \]

Now fix some $\delta>0$.  Choose $n_*\geq 1$ so that if $n\geq\frac{r}{2}n_*$ then $\mu^{(n)}(\chi^\mu_{n,\epsilon})>1-\delta$, and if $n\geq \frac{1-r}{2}n_*$ then $\nu^{(n)}(\chi^\nu_{n,\epsilon})>1-\delta$. Since $\epsilon<\min\{\frac{r}{2},\frac{1-r}{2}\}$, we have that for any $n \geq n_*$ and any $m \leq n$, if $\frac{m}{n}\approx_\epsilon r$ then $m \geq \frac{r}{2} n_*$ and $n-m \geq \frac{1-r}{2} n_*$. 
   
   Suppose $n\geq n_*$ and $X\in\cP_{r,\epsilon}(n)$. We will show $\lambda_{n,X}(\chi^\lambda_{n,6\epsilon})>(1-\delta)^2$. Let $m=|X|$. By construction, we have $m \geq \frac{r}{2} n_*$ and $n - m \geq \frac{1-r}{2} n_*$. Enumerate
   \[
   X=\{i_1,\dots,i_{m}\}\makebox[.5in]{and} [n]\backslash X=\{j_1,\dots,j_{n-m}\}.
   \]
 Consider the formula $\Phi(x_1,\ldots,x_n)\coloneqq \chi^\mu_{m,\e}(x_{i_1},\dots,x_{i_{m}}) \wedge \chi^\nu_{n-m,\e}(x_{j_1},\dots,x_{j_{n-m}})$. Then we have
 \[
 \lambda_{n,X}(\Phi)=\mu^{(m)}(\chi^\mu_{m,\epsilon})\cdot\nu^{(n-m)}(\chi^\nu_{n-m,\epsilon})>(1-\delta)^2.
 \]
Furthermore, for any $\abar\models\Phi$ and $b\in \cU^y$, we have 
    \begin{align*}
    \lambda( \phi(x,b)) &= r\mu(\phi(x,b)) + (1-r) \nu(\phi(x,b)) \\
    & \approx_{2\e} \frac{m}{n} \mu (\phi(x,b)) + \frac{n-m}{n} \nu(\varphi(x,b)) \\
    & \approx_{\e} \frac{m}{n} \Av(a_{i_1},\dots,a_{i_{m}})( \phi (x,b)) + \frac{n - m}{n} \Av(a_{j_1},\dots,a_{j_{n-m}})( \phi (x,b)) \\
    & = \Av(a_1,\dots,a_n)( \phi (x,b)).
    \end{align*}
Thus  $\Phi(\cU^{\xbar})\seq \Avo{n}{3\epsilon}{\lambda}{\phi}\seq\chi^\lambda_{n,6\epsilon}(\cU^{\xbar})$, and so $\lambda_{n,X}(\chi^\lambda_{n,6\epsilon})\geq\lambda_{n,X}(\Phi)>(1-\delta)^2$.

Finally, fix $n\geq n(\delta)\coloneqq \max\{n_*,\frac{r(1-r)}{\epsilon^2\delta}\}$. 
We show  $\lambda^{(n)}(\chi^\lambda_{n,6\e}) > (1-\delta)^3$. Indeed, we have just shown  that $\lambda_{n,X}(\chi^\lambda_{n,6\epsilon})>(1-\delta)^2$ for all $X\in \cP_{r,\epsilon}(n)$. So
    \begin{align*}
    \lambda^{(n)}(\chi^\lambda_{n,6\epsilon}) &= \sum_{X\seq[n]} r^{|X|}(1-r)^{n-|X|}\lambda_{n,X}(\chi^\lambda_{n,2\epsilon})\\
    &\geq \sum_{X\in \cP_{r,\epsilon}(n)}r^{|X|}(1-r)^{n-|X|}\lambda_{n,X}(\chi^\lambda_{n,2\epsilon})\\
    &>(1-\delta)^2\left(\sum_{X\in \cP_{r,\epsilon}(n)}r^{|X|}(1-r)^{n-|X|}\right)\\
    &>(1-\delta)^3,
    \end{align*}
    where the final inequality uses Fact \ref{fact:binomial} and choice of $n$.  
   \end{proof}
 
Finally, we discuss the question of whether \fim\ measures are closed under the Morley product. In \cite{CoGa}, a negative answer to this question was claimed by the first two authors, due to an example from \cite{ACPgs} of a generically stable type $p$ such that $p\otimes p$ is not generically stable. However, a gap in the proof was later noted by the third author. In fact, we will show in Section \ref{sec:PER} that the ambient theory defined in \cite{ACPgs} admits no nontrivial \dfs\ measures. In an earlier draft of this paper, we further claimed that \fim\ measures are indeed closed under Morley products, but an error in our proof was found by Silvain Rideau and Paul Wang. Thus the question remains open, and so we close this section with some remarks on the underlying subtleties.

Suppose $\mu\in\kM_x(\cU)$ and $\nu\in\kM_y(\cU)$ are \fim\ over some $M\prec\cU$. To analyze the question of \fim\ for $\mu\otimes\nu$, we consider a formula $\varphi(x,y,z)$ which, without loss of generality, is over $M$. Let $(\theta_{n}(x_1,...,x_n))_{n = 0}^{\infty}$ and $(\chi_{n}(y_1,...,y_n))_{n=0}^{\infty}$ be sequences of $\mathcal{L}_{M}$-formulas obtained by applying the definition of \fim\ to $\mu$ and $\nu$ with respect to the relevant bi-parititions of $\varphi(x,y,z)$. If we let $\psi_{n}(x_1,y_1,...,x_n,y_n)$ denote the $\cL_M$-formula $\theta_{n}(x_1,...,x_n) \wedge \chi_{n}(y_1,...,y_n)$, then it is not difficult to show that $\lim_{n \to \infty} (\mu \otimes \nu)^{(n)}(\psi_n) = 1$ and, for any $\epsilon>0$, if $n$ is sufficiently large then $\mu\otimes\nu \approx^{\varphi}_\epsilon \Av(a_i,b_j)_{i,j\leq n}$ for any $(\abar,\bbar)\models\psi_n$. However, this does not (a priori) imply the same conclusion for $\Av(a_i,b_i)_{i\leq n}$, which would be needed to conclude $\mu\otimes\nu$ is \fim\ using this argument. Indeed, despite obtaining arbitrarily good approximations along the full array $(a_i,b_j)_{i,j\leq n}$, there could be very different behavior on the `diagonal'. This suggests that perhaps the original intuition from \cite{ACPgs} is correct after all, and counterexamples may exist in sufficiently complicated theories. Such issues will be considered in future work, along with other questions about \fim\ measures and generic stability.

\section{Examples: \dfs\ and not \fam}\label{sec:dfsfam}

One of the main questions left open in \cite{CoGa} was on the existence of a global measure that is \dfs\ but not \fam. What was done in \cite{CoGa} was a local version of this phenomenon. Specifically, it was shown that for any $s> r\geq 3$, if $T_{r,s}$ is the theory of the generic $K^r_s$-free $r$-uniform hypergraph (where $K^r_s$ is the complete $r$-uniform hypergraph on $s$ vertices), then there is a formula $\phi(x,y)$ and a $\phi$-type in $S_\phi(\cU)$ that is \dfs\ and not \fam\ with respect to $\phi(x,y)$. However, it is also proved in \cite{CoGa} that this type cannot be extended to a \dfs\ global type (or measure).

The goal of this section is to construct complete theory with a \dfs\ global type that is not \fam. Our theory, denoted $\Ti$, will be much more complicated than $T_{r,s}$ (although the proof of \dfs\ and not \fam\ will be easier in some ways). Therefore, we will first construct a less complicated theory  $\Te$ with a complete \emph{quantifier-free} type $q_0$ that is \dfs\ and not \fam. We will then note some problems that arise when investigating quantifier elimination for $\Te$ which, in particular, suggest that finding a complete \dfs\ extension of $q_0$  is likely to be difficult, if not impossible.  This will motivate the construction of $\Ti$, a complicated variation of $\Te$ which  admits a global complete type $q$ that is \dfs\ and not \fam. We will also construct a definable measure $\mu$ in $\Ti$ that does not commute with  $q$ (as promised in Example \ref{ex:com1}).

\subsection{Sets that are half full}\label{sec:Te}
In this section we define $\Te$. This theory seems to represent the paradigm for the combinatorial separation of \dfs\ from \fam. We will work with the interval $\lb 0,1\rp$. Given $n\geq 1$, let $\cI_n=\{\lb\frac{i-1}{n},\frac{i}{n}\rp:1\leq i\leq n\}$.  

Let $\cL=\{P,Q,\sqin\}$ where $P$ and $Q$ are unary sorts and $\sqin$ is a binary relation on $P\times Q$. Define an $\cL$-structure $\Me$ such that
 \begin{enumerate}[\hspace{5pt}$\ast$]
 \item $P(\Me)$ is the interval $\lb 0,1\rp$,
 \item $Q(\Me)$ is the set of subsets of $\lb 0,1\rp$ obtained as the union of exactly $n$ distinct intervals in $\cI_{2n}$ for some $n\geq 1$, and 
 \item $\sqin^{\Me}$ is the membership relation.
 \end{enumerate}
Note that any set in $Q(\Me)$ has Lebesgue measure $\frac{1}{2}$. 
 
Define $\Te=\Th(\Me)$. Let $\cU\succ \Me$ be a monster model. Define
 \[
 q_0(y)=\{a\sqin y:a\in P(\cU)\}.
 \]
 Note that $q_0\models y\neq b$ for all $b\in Q(\cU)$. So $q_0$ determines a unique complete quantifier-free type, which is $\emptyset$-definable with respect to quantifier-free formulas.

\begin{proposition}\label{prop:Tetype}
$q_0$ is finitely satisfiable in $\Me$, but not finitely approximated in $\Me$ with respect to $x\sqin y$.  
\end{proposition}
\begin{proof}
We first show $q_0$ is finitely satisfiable in $\Me$. Fix $a_1,\ldots,a_n\in P(\cU)$. We need to find some $b\in Q(\Me)$ such that $a_i\sqin b$ holds for all $1\leq i\leq n$. Let $b_1,\ldots,b_k\in Q(\Me)$ enumerate the sets obtained as the union of  exactly $n$ intervals in $\cI_{2n}$ (so $k={2n\choose n}$). Then $\Me$ satisfies the $\cL_{\Me}$-sentence saying that for any $n$ elements from $P$, there is some $b_t$ containing them.

Now we show that $q_0$ is not finitely approximated in $\Me$ with respect to $x\sqin y$. Fix a tuple $\bbar\in Q(\Me)^n$. We find $a\in \lb 0,1\rp$ such that $|\{1\leq i\leq n:a\sqin b_i\}|\leq \frac{n}{2}$. Define $f\colon \lb 0,1\rp\to \{0,1,\ldots,n\}$ such that $f(a)=|\{1\leq i\leq n:a\sqin b_i\}|$, i.e., $f=\sum_{i=1}^n\boldsymbol{1}_{b_i}$. Then $f$ is integrable and we have
\[
\int_0^1 f\, dx=\int_0^1\sum_{i=1}^n\boldsymbol{1}_{b_i}\,dx=\sum_{i=1}^n\int_0^1\boldsymbol{1}_{b_i}\, dx=\frac{n}{2}.
\]
So there is some $a\in \lb 0,1\rp$ such that $f(a)\leq \frac{n}{2}$, as desired.
\end{proof}

Altogether, $q_0$ is a complete quantifier-free global type, which is both $\emptyset$-definable and finitely satisfiable in $\Me$. By $\emptyset$-definability, we conclude that $q_0$ is finitely satisfiable in any small model, but not finitely approximated in any small model. However, the theory $\Te$ does not have quantifier elimination. For example, we can define relations on $Q$ of the form $f(\xbar)=g(\ybar)$ where $f$ and $g$ are terms in the language of Boolean algebras. A natural route to a theory with reasonable quantifier elimination might involve replacing $Q(\Me)$ by a suitable Boolean algebra. But, as we show in the next section, this would cause problems for finding nontrivial \dfs\ types.  Therefore we will abandon $\Te$, and replace it with a more complicated theory that is able to sidestep the obstacles created by Boolean algebras.

 \subsection{Interlude on \dfs\ types in Boolean algebras}

\begin{proposition}\label{prop:BAbad}
Let $T$ be the complete theory of a Boolean algebra, and 
 suppose $p\in S_1(\cU)$ is \dfs\ over $M\prec\cU$. Then $p$ is realized in $M$.
 \end{proposition}
 \begin{proof}
 We use the symbols $\sqcup$, $\sqcap$,\,$\cop$, $\top$, $\bot$, and $\sqseq$ to denote the join, meet, complement, top element, bottom element, and induced partial order, respectively. Choose $\cL_M$-formulas $\phi(y)$ and $\psi(y)$ such that, given $c\in\cU$, we have $p\models x\sqseq c$ if and only if $\phi(c)$, and $p\models c\sqseq x$ if and only if $\psi(c)$. By \cite[Proposition 2.9]{CoGa}, we may assume $\phi(y)$ is a Boolean combination of formulas of the form $a\sqseq y$ for $a\in M$, and $\psi(y)$ is a Boolean combination of formulas of the form $y\sqseq a$ for $a\in M$. So we can choose $a_1,\ldots,a_k,b_1,\ldots,b_\ell\in M$ and finite sets $X_1,\ldots,X_k,Y_1,\ldots,Y_k\seq M$ such that
  \begin{align*}
  \phi(y) &= \bigvee_{i=1}^k  \left((a_i\sqseq y)\wedge\bigwedge_{m\in X_i}\neg (m\sqseq y)\right)\\
  \psi(y) &= \bigvee_{i=1}^\ell  \left((y\sqseq b_i)\wedge\bigwedge_{n\in Y_i}\neg (y\sqseq n)\right)
  \end{align*}
  Note that both $\phi(y)$ and $\psi(y)$ are consistent since we have $\phi(\top)$ and $\psi(\bot)$. So we can discard any inconsistent disjuncts in either formula. It then follows that $\phi(a_i)$ holds for all $1\leq i\leq k$, and $\psi(b_i)$ holds for all $1\leq i\leq \ell$. So $p\models x\sqseq a_i$ for all $i\leq k$, and $p\models b_i\sqseq a$ for all $i\leq\ell$.  Thus if $a=a_1\sqcap\ldots\sqcap a_k$ and $b=b_1\sqcup\ldots\sqcup b_\ell$, then we have $p\models b\sqseq x\sqseq a$. 
 
 Since $p$ is consistent, we know that $b\sqseq a$.  Moreover, for any $c\in\cU$, if $b\sqsubset c\sqsubset a$ then $\neg\phi(c)\wedge\neg\psi(c)$ holds, and so $p\models\neg(x\sqseq c\vee c\sqseq x)$. 

 Let $B=\{m\in M:b\sqseq m\sqseq a\}$. We view $B$ as a Boolean algebra, with bottom element $b$ and top element $a$. We will show that $|B|\leq 2$, from which it follows that $M\models\neg\exists x(b\sqsubset x\sqsubset a)$, and so $p$ must be realized by $a$ or by $b$. Toward a contradiction, suppose $|B|>2$.  
 
  Let $P$ be an ultrafilter over $B$. Then the type $\{b\sqsubset x\sqseq m:m\in P\}$ is finitely satisfiable, and thus realized by some $c\in \cU$. Note that $b\sqsubset c\sqsubset a$. Let $d=b\sqcup (a\sqcap c\cop)$ (i.e., $d$ is the complement of $c$ in $B(\cU)$). Then $b\sqsubset d\sqsubset a$ as well. Therefore $p\models (b\sqseq x\sqseq a)\wedge \neg(c\sqseq x)\wedge\neg (x\sqseq d)$. So we can find some $m\in M$ realizing this formula. Then $m\in B$ since $b\sqseq m\sqseq a$. Since $\neg(c\sqseq m)$, we must have $m\not\in P$. So $b\sqcup (a\sqcap m\cop)\in P$, and so $c\sqseq b\sqcup (a\sqcap m\cop)$. It follows that $m\sqseq d$, which contradicts the choice of $m$.
 \end{proof}

After relativizing the previous argument, we have the following conclusion.

\begin{corollary}\label{cor:BAbad}
Let $T$ be a complete theory, and suppose $p\in S_x(\cU)$ is a global type, which is \dfs\ over some $M\prec\cU$. Given an $\cL_M$-formula $\phi(x)\in p$, if there is an $M$-definable Boolean algebra on $\phi(M^x)$, then $p$ is realized in $\phi(M^x)$.  
\end{corollary}

\begin{remark}\label{rem:BAbad}
We briefly note that one can construct \emph{measures} in Boolean algebras that are nontrivial in the sense of Definition \ref{def:triv}. (This was posed as a question in a preliminary version of the paper.) In particular, let $\ABA$ be the theory of atomless Boolean algebras, and fix $\Uc\models\ABA$. We construct a nontrivial \dfs\ (in fact, smooth) measure $\sigma\in \kM_1(\cU)$. 

Say that an $\cL_{\Uc}$-formula $\varphi(x)$ (in a single free variable) is \emph{convex} if whenever $\varphi(a)$ holds, $\varphi(c)$ holds, and $a \sqsubseteq b \sqsubseteq c$, then $\varphi(b)$ holds. It is not too hard to show, using quantifier elimination for $\ABA$, that every formula in a single free variable is equivalent to a (finite) disjunction of convex formulas. 
    
    Let $\cH$ be the Boolean algebra of subsets of $[0,1)$ generated by half-open intervals of the form $[a,b)$ with $0\leq a<b\leq 1$. Note that $\cH$ is an atomless Boolean algebra. We will regard $\cH$ as an elementary sub-structure of $\Uc$.  Given an $\cL_\cU$-formula $\varphi(x)$, we define $X(\varphi)=\{r\in[0,1]:\cU\models\varphi(\lp 0,r\rb)\}$. 
    
    For any convex $\cL_{\Uc}$-formula $\varphi(x)$, it is easy to see that $X(\varphi)$ is an interval. Therefore, for any $\cL_{\Uc}$-formula $\psi(x)$, we have that $X(\psi)$ is a finite union of intervals and is thus Lebesgue measurable. Define a measure $\sigma$ by setting $\sigma(\psi(x))$ equal to the Lebesgue measure of $X(\psi)$. Then $\sigma$ is clearly not a trivial measure. Moreover, $\sigma$ is the unique global extension of $\sigma{\upharpoonright}\cH$ (this follows using an argument analogous to the proof that Lebesgue measure restricted to the interval $[0,1]$ is smooth in $\mathsf{DLO}$). 
\end{remark}
 
\subsection{Amalgamated Boolean algebras}\label{sec:Ti}
In this section, we construct a theory, denoted $\Ti$, which admits a complete \dfs\ global type that is not \fam. This theory will be obtained from $\Te$ by making various modifications, which are guided by two opposing forces. On the one hand, quantifier elimination is made easier by working with a richer structure in the $Q$ sort, such as a Boolean algebra. 
So we start with the Boolean algebra $\cH$ from Remark \ref{rem:BAbad}, i.e., the sub-algebra of $\cP([ 0,1))$ generated by half-open intervals $\lb a,b\rp$.
On the other hand, Corollary \ref{cor:BAbad} tells us that using $\cH$ in the $Q$ sort will destroy any chance of obtaining a nontrivial \dfs\ type. In order to fix this issue, we will replace $\cH$ by a certain algebraic lattice, denoted $\cQ$, which will be obtained by taking infinitely many disjoint copies of $\cH$ and identifying all of the top elements and all of the bottom elements. 
By doing this, we will be able to construct a complete global type, which is built from the same local instance of \dfs\ and not \fam\ found in $\Te$, but is also able to avoid concentrating on any `standard' Boolean algebra defined in $\cU$, and thus maintain \dfs\ globally.

We now work toward the definition of $\Ti$, starting with a precise description of $\cQ$ (which, in fact, will be a lattice with a complement operator). First, let  $\cH$ be as defined above, and set $\cH_0 = \cH \backslash \{ \emptyset, \lb 0,1\rp\}$. Set $\cQ=(\N\times\cH_0)\cup\{\bot,\top\}$, where $\bot$ and $\top$ are new symbols. Define the structure $(\cQ,\sqcap,\sqcup,\,\cop,\bot,\top)$ as follows.
\begin{enumerate}[\hspace{5pt}$\ast$]
\item Given $m,n\in\N$ and $X,Y\in\cH_0$,
\begin{align*}
(m,X)\sqcap(n,Y) &= \begin{cases}
(m,X\cap Y) & \text{if $m=n$}\\
\bot & \text{if $m\neq n$,}
\end{cases}\\
(m,X)\sqcup(n,Y) &=\begin{cases}
(m,X\cup Y) & \text{if $m=n$}\\
\top & \text{if $m\neq n$, and}
\end{cases}\\
(m,X)\cop &= (m,\lb 0,1\rp\backslash X).
\end{align*}
\item Given $b\in\cQ$, 
\[
\begin{tabular}{ccc}
\begin{tabular}{l}
$\bot \sqcap b = b \sqcap \bot = \bot$\\
$\top \sqcap b = b \sqcap \top = b$
\end{tabular} & and & 
\begin{tabular}{l}
$\bot \sqcup b = b \sqcup \bot = b$\\
$\top \sqcup b = b \sqcup \top = \top$.
\end{tabular}
\end{tabular}
\]
\item $\bot\cop=\top$ and $\top\cop=\bot$.
\end{enumerate}
In fact, $\cQ$ is an orthocomplemented lattice; however, general familiarity with such structures will not be necessary, and the reader will only need the definition above. 
We let $\cQ_0$ denote $\cQ\backslash\{\bot,\top\}=\N\times\cH_0$.

The next goal is to describe the set that will play the same role held by $\lb 0,1\rp$  in the $P$ sort of $\Te$. Since $\cQ$ is not a Boolean algebra of subsets of some ground set, we first need to define an analogous notion of `membership' in elements of $\cQ$.

\begin{definition}\label{def:sqin}
Define a binary relation $\sqin$ on $\lb 0,1\rp^{\N}\times\cQ$ such that $a\sqin  b$ if and only if either $b=\top$ or $b=(n,X)\in\cQ_0$ and $a(n)\in X$. 
\end{definition}

In order to obtain quantifier elimination for $\Ti$, we will not use $\lb 0,1\rp^{\N}$ in the $P$ sort, but rather a countable subset $S$ satisfying certain properties (described in Lemma \ref{lem:dense-seq} below). We first need some terminology.

 \begin{definition}
   A \textbf{cube} is a nonempty  subset of $\lb 0,1\rp^\N$ of the form $C=\prod_{n\in\N}\lb a_n,b_n\rp$ with $a_n = 0$ and $b_n = 1$ for all but finitely many $n\in \N$. If, moreover, $a_n$ and $b_n$ are rational for all $n$, then $C$ is a \textbf{rational cube}. 
 \end{definition}
 
  The following facts, which will use later on, are easy to verify.
 \begin{fact}\label{fact:cubes-for-babies}
 Let $B$ be the Boolean algebra on $\lb 0,1\rp^{\N}$ generated by sets of the form $\{a \in \lb0,1\rp^\N : a\sqin b\}$ for all $b \in \cQ$.
 \begin{enumerate}[$(a)$]
   \item $B$ is the same as the Boolean algebra generated by cubes.
   \item Every element of $B$ can be written as a disjoint union of finitely many cubes.
   \item Suppose $f: \lb 0,1\rp^\N\to \R$ is a finite linear combination of indicator functions of sets in $B$. Then $f$ can be written as $\sum_{i<k} r_i \boldsymbol{1}_{C_i}$, where $\{C_i\}_{i<k}$ is a partition of $\lb 0,1\rp^\N$ into disjoint cubes and each $r_i$ is in the image of $f$.
   \end{enumerate}
 \end{fact}

   \begin{lemma}\label{lem:dense-seq}
   There is a countably infinite set $S\seq\lb 0,1\rp^\N$ such that:
   \begin{enumerate}[$(i)$]
   \item $S$ nontrivially intersects any cube in $\lb 0,1\rp^\N$, and
   \item the map $(a,n)\to a(n)$ from $S\times\N$ to $\lb 0,1\rb$ is injective.
   \end{enumerate}
   \end{lemma}
   \begin{proof}
     Let $\{C_\ell\}_{\ell \in \N}$ be an enumeration of all rational cubes. We will build the set $S=\{a_\ell\}_{\ell\in\N}$ in stages. At stage $\ell$, pick $a_\ell \in \lb 0,1\rp^\N$ so that
     \begin{enumerate}[$(1)$]
     \item $a_\ell \in C_\ell$,
     \item $a_\ell$ is an injective map whose range  is disjoint from those of $a_i$ for each $i<\ell$.
     \end{enumerate}
      Since at each stage we have only used countably many elements of $\lb 0,1\rp$, it is clear that we can always make such a choice of $a_\ell$. Now, since any cube in $\lb 0,1\rp^\N$ contains a rational cube, we have condition $(i)$ of the lemma by part $(1)$ of the construction. Condition $(ii)$ of the lemma follows from part $(2)$.  
   \end{proof}

We now have all of the ingredients necessary to define $\Ti$. In order to mimic the behavior  in Proposition \ref{prop:Tetype}, we need to be able to pick out sets in $\cQ$ of `measure' $\frac{1}{2}$. With an eye toward  quantifier elimination, we will add a third sort $R$ for the ordered group of reals, and a unary function from $Q$ to $R$ for the appropriate measure. Altogether, we 
define   $\cL=\{P,Q,R,\sqin,\sim,\ell,\sqcap,\sqcup,\,\cop,\bot,\top,+,<,0,1\}$ where 
 \begin{enumerate}[\hspace{5pt}$\ast$]
 \item $P$, $Q$, and $R$ are unary sorts, 
 \item $\sqin$ is a binary relation on $P\times Q$,  
 \item $\sim$ is a binary relation on $Q$, 
 \item $\ell$ is a unary function from $Q$ to $R$,  
 \item $\{\sqcap,\sqcup,\,\cop,\bot,\top\}$ is the language of `lattices with complements' on $Q$, and
 \item $\{+,<,0,1\}$ is the language of ordered abelian groups on $R$, with an additional constant symbol $1$.
 \end{enumerate}
 Now let $\lambda$ denote the Lebesgue measure on $\lb 0,1\rp$. We define an $\cL$-structure $\Mi$ via the following interpretation of $\cL$.
 \begin{enumerate}[\hspace{5pt}$\ast$]
 \item $P(\Mi)$ is a fixed set $S\seq\lb 0,1\rp^\N$ as in  Lemma~\ref{lem:dense-seq}.
 \item $(Q(\Mi),\sqcap,\sqcup,\,\cop,\bot,\top)$ is $(\cQ,\sqcap,\sqcup,\,\cop,\bot,\top)$.
 \item $(R(\Mi),+,<,0,1)$ is $(\R,+,<,0,1)$.
 \item $\sqin$ is as described in Definition \ref{def:sqin} (but restricted to $P(\Mi)\times\cQ$).
 \item If $b,c\in  Q(\Mi)$, then $b\sim c$ holds if and only if $b$ and $c$ are both in $\cQ_0$ and have the same first coordinates.
 \item If $b\in Q(\Mi)$ then $\ell(b)=\begin{cases}
\lambda(X) & \text{if $b=(n,X)\in\N\times\cH_0$}\\
 0 & \text{if $b=\bot$, and}\\
 1 & \text{if $b=\top$.}
 \end{cases}$
 \end{enumerate}
Finally, we define $\Ti=\Th(\Mi)$.

Let $\cU\succ \Mi$ be a monster model. For the rest of this section, we let $y$ denote a variable (of length one) in the $Q$ sort. Define 
\[
\textstyle q_1(y)=\{a\sqin y:a\in P(\cU)\}\cup\{\ell(y)=\frac{1}{2}\}\cup\{y\not\sim b:b\in Q(\cU)\}.
\]
We now observe that $q_1$ is a partial global type exhibiting the same behavior found in Proposition \ref{prop:Tetype}.

\begin{proposition}\label{prop:Titype}
$q_1$ is finitely satisfiable in $\Mi$, but not finitely approximated in $\Mi$ with respect to the formula $\psi(y,x)\coloneqq (x\sqin y) \wedge(\ell(y)=\frac{1}{2})$. 
\end{proposition}
\begin{proof}
We first show $q_1$ is finitely satisfiable in $\Mi$. Fix $a_1,\ldots,a_n\in P(\cU)$ and $b_1,\ldots,b_m\in Q(\cU)$.  We need to find some $c\in \cQ$ such that $\ell(c)=\frac{1}{2}$, $a_i\sqin c$ holds for all $1\leq i\leq n$, and $c\not\sim b_i$ for all $1\leq i\leq m$. Fix $s\in\N$ such that, for all $1\leq i\leq m$, $b_i\not\sim (s,X)$ for any $X\in\cH_0$. Let $X_1,\ldots,X_k\in\cH_0$ enumerate the sets obtained as the union of exactly $n$ intervals in $\cI_{2n}$, and consider $c_t=(s,X_t)\in\cQ$  for $t\leq k$. Then by construction there exists some $t_*\leq k$ such that $\cU\models a_i\sqin c_{t_*}$ for all $1\leq i\leq n$, as desired.

Now we show that $q_1$ is not finitely approximated in $\Mi$ with respect to $\psi(y,x)$. Fix a tuple $\bbar\in \cQ^n$. We find $a\in P(\Mi)$ such that $|\{1\leq i\leq n:\psi(b_i,a)\}|\leq \frac{n}{2}$. Without loss of generality, we may assume $\ell(b_i)=\frac{1}{2}$ for all $1\leq i\leq n$. Let $\lambda^{\N}$ denote the measure on $\lb 0,1\rp^{\N}$ obtained from the product of $\lambda$ on $\lb 0,1\rp$.  
Define $f\colon \lb 0,1\rp^{\N}\to\{0,1,\ldots,n\}$ such that $f(a)=|\{1\leq i\leq n: a\sqin b_i\}|$. Given $1\leq i\leq n$, set $B_i=\{a\in \lb 0,1\rp^{\N}:a\sqin b_i\}$. 
Then each $B_i$ is $\lambda^{\N}$-measurable, with $\lambda^{\N}(B_i)=\ell(b_i)=\frac{1}{2}$, and $f=\sum_{i=1}^n\boldsymbol{1}_{B_i}$. Therefore $f$ is $\lambda^{\N}$-integrable  and
\[
\int_{\lb 0,1\rp^{\N}} f\, d\lambda^{\N}=\int_{\lb 0,1\rp^{\N}}\sum_{i=1}^n\boldsymbol{1}_{B_i}\,d\lambda^{\N}=\sum_{i=1}^n\int_{\lb 0,1\rp^{\N}} \boldsymbol{1}_{B_i}\, d\lambda^{\N}=\frac{n}{2}.
\]
By Fact~\ref{fact:cubes-for-babies}, $f=\sum_{j<k}r_j \boldsymbol{1}_{C_j}$ where $\{C_j\}_{j<k}$ is a partition of $\lb 0,1\rp^\N$ into disjoint cubes and each $r_j$ is in $\{0,1,\ldots,n\}$.
 So it must be the case that $r_j \leq \frac{n}{2}$ for some $j<k$. By our choice of $P(\Mi)$, there must be some $a \in P(\Mi) \cap C_j$, which therefore has $f(a)\leq \frac{n}{2}$, as desired.
\end{proof}

We now turn to the main goal, which is to show that $q_1$ extends to a complete global type in $S_Q(\cU)$ that is \dfs\ and not \fam\ (here $S_Q(\cU)$ denotes the space of complete global types concentrating on the $Q$ sort). In fact, we will show that $q_1$ determines a unique complete type in $S_Q(\cU)$ and that this type has the desired properties. The first step is quantifier elimination.

\begin{theorem}\label{thm:TiQE}
$\Ti$ has quantifier elimination.
\end{theorem}

The proof of the previous theorem is rather involved and so, to avoid stalling the primary exposition, we have cordoned off the details in Section \ref{app:QE} of the appendix.  So let us now continue toward the main goal.

\begin{lemma}\label{lem:0def}
For any $\cL$-formula $\phi(y,\zbar)$, there is an $\cL$-formula $\psi(\zbar)$ such that, for any $\bbar\in\cU^{\zbar}$, if $\cU\models\psi(\bbar)$ then $q_1\models \phi(y,\bbar)$, and if $\cU\models\neg\psi(\bbar)$ then $q_1\models\neg\phi(y,\bbar)$.
\end{lemma}
\begin{proof}
Let $\cL_Q$ and $\cL_R$ denote the restrictions of $\cL$ to the $Q$ and $R$ sorts, respectively. 
We first claim that, for any $\cL_Q$-term $t(y,\zbar)$, there is a term $s(y,\zbar)$ of the form $y$, $y\cop$, or $u(\zbar)$ such that $q\models t(y,\bbar)=s(y,\bbar)$ for all $\bbar\in \cU^{\zbar}$. Indeed, this can be proved by induction on terms. The main point is that, for any $b\in \cU$, $q_1\models y\not\sim b$, and so $q_1\models (y\sqcap b=\bot)\wedge (y\sqcup b=\top)$.  

Now we prove the lemma. It suffices to assume $\phi$ is atomic. We consider cases.

Suppose $\varphi(y,\zbar,x)$ is $x\sqin t(y,\zbar)$ where $t(y,\zbar)$ is an $\cL_Q$-term. Let $s$ be as in the initial claim. If $s$ is $y$, let $\psi(\zbar,x)$ be $x=x$. If $s$ is $y\cop$,  let $\psi(\zbar,x)$ be $x\neq x$. If $s$ is $u(\zbar)$, let $\psi(\zbar,x)$ be $x\sqin u(\zbar)$. 

Suppose $\varphi(y,\zbar)$ is $t_1(y,\zbar)\asymp t_2(y,\zbar)$ where $t_1,t_2$ are $\cL_Q$-terms and $\asymp$ is $=$ or $\sim$. Let $s_1$ and $s_2$ be as in the initial claim. If $s_1,s_2\in\{y,y\cop\}$, then let $\psi(\zbar)$ be $\zbar=\zbar$ if $s_1$ and $s_2$ are the same or $\asymp$ is $\sim$, and let $\psi(\zbar)$ be $\zbar\neq\zbar$ otherwise. If $s_1$ and $s_2$ are $u_1(\zbar)$ and $u_2(\zbar)$, let $\psi(\zbar)$ be $u_1(\zbar)\asymp u_2(\zbar)$. Otherwise, let $\psi(\zbar)$ be $\zbar\neq\zbar$.

Finally, suppose $\varphi(y,\zbar,\wbar)$ is $f(\ell(t_1(y,\zbar)),\ldots,\ell(t_n(y,\zbar)),\wbar)\asymp 0$ 
 where $\asymp$ is $=$ or $<$, each $t_i$ is an $\cL_Q$-term, and $f(v_1,\ldots,v_n,\wbar)$  is an $\cL_R$-term.
Let $s_1,\ldots,s_n$ be as in the initial claim. Recall that $q\models \ell(y)=\ell(y\cop)=\frac{1}{2}$. Let $v_i(\zbar)$ be either $\frac{1}{2}$, if $s_i$ is $y$ or $y\cop$, or  $\ell(u_i(\zbar))$ if $s_i$ is $u_i(\zbar)$. Then we may take $\psi(\zbar,\wbar)$ to be $f(v_1(\zbar),\ldots,v_n(\zbar),\wbar)\asymp 0$.
\end{proof}
 
 \begin{corollary}\label{cor:dfsnotfam}
 There is a unique complete type $q\in S_Q(\cU)$ extending $q_1$. Moreover, $q$ is $\emptyset$-definable and finitely satisfiable in any small model, but is not \fam. 
 \end{corollary}
 \begin{proof}
 By Lemma \ref{lem:0def}, there is a unique type $q\in S_Q(\cU)$ extending $q_1$, and $q$ is $\emptyset$-definable. By Proposition \ref{prop:Titype}, $q$ is finitely satisfiable in $\Mi$ (in particular, for any $\phi(x)\in q$ there is some $\psi(x)\in q_1$ such that $\psi(x)$ implies $\phi(x)$), but not \fam\ over $\Mi$. 
 Since $q$ is $\emptyset$-invariant, the same is true over any small model. 
 \end{proof}

 \subsection{Failure to commute}\label{sec:com1}
Let $\cU\models \Ti$. In this section, we construct a definable measure $\mu\in\kM_P(\cU)$ that does not commute with the \dfs\ type $q\in S_Q(\cU)$ from the last section. This justifies the claim made in Example \ref{ex:com1}.  Throughout this section, $x$ and $y$ denote tuples of variables of length one in the $P$ and $Q$ sorts, respectively.

Recall that in the standard model $\Mi$ of $\Ti$, the set $P(\Mi)$ is a subset of $\lb 0,1\rp^{\N}$, which is equipped with the product-Lebesgue measure $\lambda^{\N}$. Now suppose $\phi(x)$ is an $\cL_{\Mi}$-formula in the $P$ sort. By quantifier elimination, the subset of $P(\Mi)$ defined by $\phi(x)$ differs by finitely many elements from a set of the form $P(\Mi)\cap \bigcup_{i<n}C_i$, with each $C_i$ a cube. Furthermore, the set $\bigcup_{i<n}C_i$ is uniquely determined by $\varphi(x)$. The map from $\Mi$-definable subsets of $P(\Mi)$ to subsets from $[0,1)^{\mathbb{N}}$ defined by taking each $\varphi(x)$ to the corresponding $\bigcup_{i < n} C_i$ is a Boolean algebra homomorphism. Let $\lambda^\ast$ be the pullback of $\lambda^\N$ along this homomorphism. Then $\lambda^\ast$ is automatically a finitely additive probability measure on the $\Mi$-definable subsets of $P(\Mi)$, so it is an element of $\mathfrak{M}_P(\Mi)$.

\begin{lemma}\label{lem:lebdef}
There is a unique $\emptyset$-definable measure $\mu\in\kM_P(\cU)$ extending $\lambda^*$.
\end{lemma}
\begin{proof}
Let $M=\Mi$.
The proof amounts to showing that $\lambda^*\in \kM_P(M)$ is ``$\emptyset$-definable". This notion of definability for Keisler measures over small models is similar to that for global measures. See Section \ref{app:heir} of the appendix for details. We will apply Remark \ref{rem:heirQE} and Theorem \ref{thm:heir}. Fix an $\cL$-formula $\Phi(x;\wbar,\ybar,\zbar)$ which is a conjunction of atomic and negated atomic formulas, where $\wbar$ is of sort $P$, $\ybar$ is of sort $Q$, and $\zbar$ is of sort $R$. Define the map $f^{\Phi}_{\lambda^*}\colon (\bbar,\cbar,\dbar)\mapsto\lambda^*(\Phi(x;\bbar,\cbar,\dbar))$ from $M^{\wbar\ybar\zbar}$ to $[0,1]$. We want to show that $f^{\Phi}_{\lambda^*}$ has an $\emptyset$-invariant continuous extension to $S_{\wbar\ybar\zbar}(M)$ (see also Definition \ref{def:babyfiber} and surrounding discussion).

Write $\Phi(x;\wbar,\ybar,\zbar)$ as $\phi(x,\ybar)\wedge \theta(x,\wbar)\wedge \chi(\wbar,\ybar,\zbar)$ where:
\begin{enumerate}[\hspace{5pt}$\ast$]
\item $\chi(\wbar,\ybar,\zbar)$ is some $\cL$-formula not mentioning $x$,
\item $\theta(x,\wbar)\coloneqq \bigwedge_{i=1}^m (x =_i w_i)$, where $=_i$ is  $=$ or $\neq$, and
\item  $\varphi(x,\ybar)\coloneqq  \bigwedge_{i=1}^n (x\sqin_i t_i(\ybar))$,
where $\sqin_i$ is $\sqin$ or $\not\sqin$, and each $t_i(\xbar)$ is an $\cL_Q$-term. 
\end{enumerate}
If some $=_i$ is $=$ then $f^{\Phi}_{\lambda^*}$ is identically $0$, in which case our task is trivial. So we may assume that each $=_i$ is $\neq$. In this case, we have $f^{\Phi}_{\lambda^*}=(f^{\phi}_{\lambda^*}\circ \rho)\boldsymbol{1}_{\chi}$, where $\rho\colon M^{\wbar\ybar\zbar}\to M^{\ybar}$ is the subtuple map. So it suffices to show that $f^{\phi}_{\lambda^*}$ has an $\emptyset$-invariant continuous extension to $S_{\ybar}(M)$.

Let $\sim^*$ denote the equivalence relation on $Q$ obtained from $\sim$ by making $\top$ and $\bot$ singleton classes, i.e., $y\sim^* y'$ iff $(y\sim y')\vee (y=y'=\top)\vee (y=y'=\bot)$. Note that $\sim^*$ is $\emptyset$-definable. Let $\Sigma$ denote the set of partitions of $[n]$. Given $\sigma\in\Sigma$, let $\theta_\sigma(\ybar)$ be the conjunction of $t_i(\ybar)\sim^* t_j(\ybar)$ for all $\sigma$-related $i,j\in[n]$, and $t_i(\ybar)\not\sim^* t_j(\ybar)$ for all $\sigma$-unrelated $i,j\in[n]$. Let $\phi_\sigma(x,\ybar)$ be the $\cL$-formula $\phi_\sigma(x,\ybar)\wedge\theta_\sigma(\ybar)$. Then $f^\phi_{\lambda^*}=\sum_{\sigma\in\Sigma}f^{\phi_\sigma}_{\lambda^*}$. So it suffices to show that each $f^{\phi_\sigma}_{\lambda^*}$ has an $\emptyset$-invariant continuous extension to $S_{\ybar}(M)$. 

Now fix $\sigma\in\Sigma$, and let $\sigma=\{\sigma_1,\ldots,\sigma_k\}$. For $1\leq j\leq k$, define the $\cL_Q$-term 
\[
u_j(\ybar)\coloneqq \bigsqcap_{i\in\sigma_j} \left\{\begin{matrix}t_i(\ybar), & \text{$\sqin_i$ is $\sqin$} \\ t_i(\ybar)\cop, & \text{$\sqin_i$ is $\not\sqin$} \end{matrix}\right\},
\]
and let $\zeta_j(x,\ybar)$ be the $\cL$-formula $x\sqin u_j(\ybar)$. Then, for any $\bbar\in M^{\ybar}$, we have  
\[
f^{\phi_\sigma}_{\lambda^*}(\bbar)=\lambda^*(\phi(x,\bbar))=\prod_{j=1}^k \ell(u_j(\bbar))=\prod_{j=1}^k f^{\zeta_j}_{\lambda^*}(\bbar)
\]
(note that here we are suppressing compositions with ``subtuple" functions, as in the first reduction from $\Phi$ to $\phi$).
So it suffices to show that each $f^{\zeta_j}_{\lambda^*}$ has an $\emptyset$-invariant continuous extension to $S_{\ybar}(M)$. For this, we apply Fact \ref{fact:defdefs2}. In particular, fix $1\leq j\leq k$ and $\epsilon<\delta$ in $[0,1]$. Let $\psi_j(\ybar)$ be the $\cL$-formula $\ell(u_j(\ybar))<\alpha$, where $\alpha\in (\epsilon,\delta)$ is rational.  Then 
\[
\left\{\bbar\in M^{\ybar}:f^{\zeta_j}_{\lambda^*}(\bbar)\leq\epsilon\right\}\seq\psi_j(M)\seq \left\{\bbar\in M^{\ybar}:f^{\zeta_j}_{\lambda^*}(\bbar)<\delta\right\},
\]
as desired.
\end{proof}

Now, since $p$ and $\mu$ are both definable, we have the Morley products $\mu\otimes q$ and $q\otimes\mu$, which are also both definable.

\begin{proposition}\label{prop:nocom}
$(\mu\otimes q)(x\sqin y)=\frac{1}{2}$ and $(q\otimes\mu)(x\sqin y)=1$.
\end{proposition}
\begin{proof}
Note that $F^{\sqin}_{q,M}$ is the constant function $1$, and so $(q\otimes\mu)(x\sqin y)=1$. On the other hand, we have $(\mu\otimes q)(x\sqin y)=\mu(x\sqin b)$ where $b\models q|_{\Mi}$. So $\ell(b)=\frac{1}{2}$. For any $c\in \cQ$, if $\ell(c)=\frac{1}{2}$ then $\mu(x\sqin c)=\lambda^*(x\sqin c)=\frac{1}{2}$. By Remark \ref{rem:heir}, this also holds for all $c\in Q(\cU)$, and thus $\mu(x\sqin b)=\frac{1}{2}$. 
\end{proof}

With Question \ref{ques:dfscom} in mind, we point out that  $\mu$ is not finitely satisfiable in $\Mi$ (and hence not in any small model). Indeed, if $b\models q|_{\Mi}$ then $\mu(x\not\sqin b)=\frac{1}{2}$ by the previous proof, but $x\not\sqin b$ has no solution in $\Mi$ by definition of $q$. In fact, any \dfs\ measure in $\kM_P(\cU)$ must be a sum of countably many weighted Dirac measures at points in $P(\cU)$. This assertion follows from an argument nearly identical to the proof of \cite[Theorem 4.9]{CoGa}, together with the following genericity property in $\Mi$: For any finite disjoint sets $A,B\subset P(\Mi)$, there is some $c\in Q(\Mi)$ such that $a\sqin c$ for all $a\in A$ and $b\not\sqin c$ for all $b\in B$ (indeed, by construction of $P(\Mi)$, such an element $c$ can be found in $\{n\}\times\cH_0$ for any $n\in\N$).

Note that Lemma \ref{lem:lebdef}, Proposition \ref{prop:nocom}, and Proposition \ref{prop:fdcom} provide another demonstration that the type $q$ is not \fam. By Theorem \ref{thm:gencom}, we also have the following conclusion for $\lambda^*$. 

\begin{corollary}
No definable extension of $\lambda^*$ in $\kM_P(\cU)$ commutes with $q$. In particular, $\lambda^*$ has no smooth (or even \fim) global extensions in $\kM_P(\cU)$. 
\end{corollary}

\begin{remark} Furthermore, the restriction of $\lambda^{*}$ to the language $\mathcal{L}_{PQ}$ is an example of a measure with no definable extension (and is the first such example that we are aware of). Let $\lambda^\ast_{PQ}$ be this restriction.  Assume that $\lambda^\ast_{PQ}$ has some definable extension $\nu$ over some $N \succeq M_{PQ}$.  Then we must have some $\cL_N$-formula $\psi(y)$ such that for any $b \in P^N$, if $\nu(x \sqin b) < \frac{1}{3}$ then $\psi(b)$ holds, and if $\nu(x \sqin b) > \frac{2}{3}$ then $\psi(b)$ fails to hold. This implies that for any $b \in P^{M_{PQ}}$,  if $\lambda^\ast_{PQ}(x \sqin b) < \frac{1}{3}$ then $\psi(b)$ holds, and if $\nu(x \sqin b) > \frac{2}{3}$ then $\psi(b)$ fails to hold.

  Find a $\sim$-class $C$ in $M_{PQ}$ such that no parameter of $\psi$ is contained in $C$. Let $A$ be the set of $P$-parameters in $\psi$. For each $a \in A$, we can find $b_a \in C$ such that $a \sqin b_a$ and $\lambda^\ast_{PQ}(x \sqin b_a) < \frac{1}{4|A|}$. By construction, we have that $\lambda^\ast_{PQ}(x \sqin \bigsqcup_{a \in A} b_a) < \frac{1}{4}$. Therefore, we can find $c$ and $d$ in $C$ with $\lambda^\ast_{PQ}(x \sqin c) < \frac{1}{3}$ and $\lambda^\ast_{PQ}(x \sqin d)  > \frac{2}{3}$ such that $c$ and $d$ are disjoint from every $b_a$.

  By quantifier elimination (Proposition A.15), $c$ and $d$ have the same type over $A$, hence  $\cU\models\psi(c) \leftrightarrow \psi(d)$. But $\lambda^\ast_{PQ}(x \sqin c) < \frac{1}{3}$ and $\lambda^\ast_{PQ}(x \sqin d) > \frac{2}{3}$, which is a contradiction. Therefore $\lambda^\ast_{PQ}$ has no definable extensions.
\end{remark}

\begin{remark}
We showed above that $q$ does not admit \fam\ approximations for $\psi(y,x)\coloneqq (x\sqin y)\wedge(\ell(y)=\frac{1}{2})$ within error $\frac{1}{2}$ (i.e., $\Avo{n}{1/2}{q}{\psi}=\emptyset$ for all $n$). Given a fixed rational $\epsilon\in (0,1)$,  let $q_\epsilon$ be the result of replacing $\frac{1}{2}$ with $\epsilon$ in the definition of $q_1$. Then similar arguments show that $q_\epsilon$ determines a complete \dfs\ type, with no \fam\ approximations for $(x\sqin y)\wedge(\ell(y)=\epsilon)$ within error $1-\epsilon$. In other words, one can construct arbitrarily terrible failures of \fam\ approximations for a  \dfs\ type. This is in contrast to the result for $T_{r,s}$ from \cite{CoGa}, mentioned at the beginning of this section, which produced a \dfs\ local $\phi$-type with no \fam\ approximations for $\phi$ within error $(r-1)!/(r-1)^{r-1}$.  

The previous modification also results in a more severe failure of symmetry in that we obtain $(\mu\otimes q_\epsilon)(x\sqin y)=\epsilon$ and $(q_\epsilon\otimes\mu)(x\sqin y)=1$.
\end{remark}

Finally, we reiterate that $\Ti$ is much more complicated than $T_{r,s}$, both in its construction and with respect to its classification in neostability. In particular, $T_{r,s}$ is supersimple, while $\Ti$ interprets the theory of atomless Boolean algebras and thus has TP$_2$ and SOP. 

\begin{question}
Is there a simple theory $T$ and a global type $p\in S_x(\cU)$  such that $p$ is \dfs\ and not \fam?
\end{question}

In general, it would be interesting to find less complicated examples of complete theories separating \dfs\ and \fam, even at the level of measures.

\section{Examples: \fam\ and not \fim}\label{sec:famfim}

In this section, we discuss examples of Keisler measures that are \fam\ and not \fim, starting with a re-examination of some examples given in \cite{CoGa}. We first note in Section \ref{sec:PER} that one of those examples relied on an erroneous claim from \cite{ACPgs}, and we show that in fact the ambient theory in this example has no nontrivial \dfs\ measures. So this reduces the previously known examples of theories with \fam\ and non-\fim\ measures to essentially one family, namely, the generic $K_s$-free graphs for some fixed $s\geq 3$. In Section \ref{sec:Henson}, we will develop more features of this example, and then give a correct proof of a certain result from \cite{CoGa}. Finally, we will show that in the reduct of $\Ti$ obtained by forgetting the  measure $\ell$, the corresponding reduct of  the \dfs\ and non-\fam\ type in Section \ref{sec:Ti} becomes a new example of a \fam\ and non-\fim\ complete type.

\subsection{Parameterized equivalence relations}\label{sec:PER}

Let $T^*_{\feq 2}$ denote the model completion of the theory of parameterized equivalence relations in which each equivalence class has size $2$. In \cite[Example 1.7]{ACPgs}, it is claimed that this theory admits a generically stable (and thus \fim) type $p\in S_1(\cU)$ such that $p^{(2)}$ is not generically stable, and this was elaborated on in \cite[Section 5.1]{CoGa}.  However, it turns out that the type $p$ suggested in \cite{ACPgs} is not well-defined (see Remark \ref{rem:Tfeq2type} for details). Here we show that, in fact, there are no nontrivial \dfs\ measures in $T^*_{\feq 2}$. We first recall some definitions. Let $T$ be a complete $\cL$-theory with monster model $\cU$.

\begin{definition}\label{def:triv}
A measure $\mu\in\kM_x(\cU)$ is \textbf{trivial} if it can be written as $\sum_{n=0}^\infty r_n\delta_{a_n}$ where $a_n\in \cU^x$ and $r_n\in[0,1]$, with $\sum_{n=0}^\infty r_n=1$. We say that $T$ is \textbf{\textit{dfs}-trivial} if every \dfs\ measure is trivial.
\end{definition}

Note that a type $p\in S_x(\cU)$ is trivial if and only if it is realized in $\cU^x$. It is clear that any trivial measure is \fim. The following result was implicitly claimed in \cite{CoGa}, but the proof used an unjustified assumption involving localization of measures, namely, \cite[Remark 4.2]{CoGa}. In reality, the argument only requires a very weak version of this remark, which is easily proved. So we clarify the details.

\begin{proposition}\label{prop:dfstriv1}
A theory $T$ is \dfs-trivial if and only if for every $x$ of length one and every \dfs\ measure $\mu\in\kM_x(\cU)$, there is some $b\in\cU$ such that $\mu(x=b)>0$. 
\end{proposition}
\begin{proof}
As noted in the proof of \cite[Proposition 4.3]{CoGa}, the forward direction is trivial. For the reverse implication, assume that for every $x$ of length one and every \dfs\ measure $\mu\in\kM_x(\cU)$, there is some $b\in\cU$ such that $\mu(x=b)>0$. To show that $T$ is \dfs-trivial, it suffices by \cite[Proposition 4.5]{CoGa}, to show that for any $x$ of length one, every \dfs\ measure in $\kM_x(\cU)$ is trivial. (The cited result from \cite{CoGa} works in a one-sorted setting for simplicity; however the proof is by induction on the length of a tuple of variables, which need not be all in the same sort for the argument to work.) For this, it suffices by the proof of \cite[Proposition 4.3]{CoGa}, to show that for an arbitrary variable tuple $\xbar$, \dfs\ measures in $\kM_{\xbar}(\cU)$ are closed under the following special case of localization. In particular, suppose $\mu\in\kM_{\xbar}(\cU)$ is \dfs, and let $X=S_{\xbar}(\cU)\backslash S$ where $S$ is a fixed countable set $S\seq S_{\xbar}(\cU)$ of \emph{realized} types, with $\mu(S)<1$. Let $\mu_0$ be the localization $\mu_0$ of $\mu$ at $X$, i.e., $\mu_0(\phi(\xbar)):=\mu(\phi(\xbar)\cap X)/\mu(X)$. Then we claim that $\mu_0$ is \dfs. 

Fix $M\prec\cU$ such that $\mu$ is \emph{dfs} over $M$, and any type in $S$ is realized in $M$. We show that $\mu_0$ is \dfs\ over $M$.  We may assume that $\mu(S)>0$, since otherwise $\mu_0=\mu$. It is clear that $\mu_0$ is finitely satisfiable in $M$. Consider the trivial measure $\mu_1 =  \frac{1}{\mu(S)}\sum_{\abar \in S} \mu(\xbar=\abar)\delta_{\abar}$, which is definable over $M$. Set $r:=\mu(X)$, and notice that $\mu=r\mu_0+(1-r)\mu_1$. 
So for any $\cL$-formula $\phi(\xbar,\ybar)$, we have $F^\phi_{\mu_0,M}=\frac{1}{r}(F^\phi_{\mu,M}-(1-r)F^{\phi}_{\mu_1,M})$, and thus $F^\phi_{\mu_0,M}$ is continuous since it can be written as a linear combination of continuous functions. Therefore $\mu_0$ is definable over $M$.
\end{proof}

\begin{remark}\label{rem:CGcor}
For the sake of clarifying the literature, we note that Proposition \ref{prop:dfstriv1} (and its proof) can be used in place of Remark 4.2 and Proposition 4.3 in \cite{CoGa} to recover the proofs of Theorems 4.8, 4.9, and 5.10$(a)$ in \cite{CoGa}. The only other use of Remark 4.2 in \cite{CoGa} is in the proof of Theorem 5.10$(b)$, which we address in the next subsection (see Theorem \ref{thm:Hfimtriv} and preceding discussion).
\end{remark}

Next, we describe a way to lift a \dfs\ measure to an imaginary sort. Suppose $E(x,y)$ is a definable equivalence relation on $\cU$. We extend $E$ to tuples from $\cU^n$ in the obvious way. We view $\cU/E$ as a structure in a relational language $\cL_0$ such that for any $E$-invariant formula $\phi(x_1,\ldots,x_n)$, we have an $n$-ary relation symbol $R_\phi$ interpreted as $\phi(\cU^n)/E$. Note that any quantifier free $\cL_0$-formula is equivalent to $R_\phi$ for some equivariant $\phi$. 

Now suppose we have a \dfs\ measure $\mu$ in $\kM_1(\cU)$. Given an $E$-invariant formula $\phi(x;y_1,\ldots,y_n)$ and $b_1,\ldots,b_n\in\cU/E$, define $\mu_0(R_\phi(x;b_1,\ldots,b_n))=\mu(\phi(x;b_1^*,\ldots,b_n^*))$, where $b^*_i$ is a representative of $b_i$ in $\cU$. 

\begin{proposition}
$\mu_0$ is a \dfs\  measure on quantifier-free $\cL_0(\cU)$-formulas. 
\end{proposition}
\begin{proof}
First note that $\mu_0(R_\phi(x;b_1,\ldots,b_n))$ does not depend on the choice of representatives by $E$-invariance, and so $\mu_0$ is well-defined. From there one easily shows that $\mu_0$ is a finitely-additive probability measure on quantifier-free $\cL_0(\cU)$-formulas.

Fix $M\prec\cU$ such that $\mu$ is \dfs\ over $M$. We show that $\mu_0$ is \dfs\ over $M/E$. First, fix some $R_\phi(x;y_1,\ldots,y_n)$ and $b_1,\ldots,b_n\in \cU/E$, with $\mu_0(R_\phi(x;b_1,\ldots,b_n))>0$. Then there is $a\in M$ such that $\cU\models\phi(a;b^*_1,\ldots,b^*_n)$. So $[a]_E\in M/E$ and $\cU/E\models R_\phi([a]_E;b_1,\ldots,b_n)$. 

Finally, fix an $E$-invariant formula $\phi(x;y_1,\ldots,y_n)$ and some $\epsilon>0$. Define
\[
X=\{\bbar\in (\cU/E)^n:\mu_0(R_\phi(x;\bbar))\leq\epsilon\}.
\]
Then $X=Y/E$ where $Y=\{\bbar^*\in \cU^n:\mu(\phi(x;\bbar^*))\leq\epsilon\}$. Since $\mu$ is \dfs\ over $M$, there is a small collection $\{\psi_i(\ybar;\xbar):i\in I\}$ of Boolean combinations of $\phi^*(\ybar;x_i)$, and tuples $\abar^*_i$ from $M$, such that $Y=\bigcap_{i\in I}\psi_i(\cU;\abar^*_i)$. Let $\abar_i=[\abar^*_i]_E$. Then we have $X=\bigcap_{i\in I}R_{\psi_i}(\cU/E;\abar_i)$.  Therefore $\mu_0$ is definable over $M$.
\end{proof}

We now return to $T^*_{\feq 2}$. 
This theory is in a two-sorted language $\cL$ with sorts $O$ and $P$, and a ternary relation $E_z(x,y)$ on $O_x\times O_y\times P_z$. Then $T^*_{\feq 2}$ is the model completion of the $\cL$-theory asserting that for any $z$, $E_z(x,y)$ is an equivalence relation in which all classes have size $2$. We have quantifier elimination after expanding by the function $f\colon P\times O\to O$ such that, for any $z\in P$, the induced function $f_z\colon O\to O$ swaps the two elements in any $E_z$-class.

\begin{theorem}\label{thm:Tfeq2triv}
$T^*_{\feq2}$ is \dfs-trivial.
\end{theorem}
\begin{proof}
Suppose not. By Proposition \ref{prop:dfstriv1}, there is some \dfs\ $\mu\in \kM_1(\cU)$ such that $\mu(x=b)=0$ for all $b\in\cU$.
Fix a parameter $e\in P(\cU)$. Let $E(x,y)$ be the equivalence relation on $\cU$ which coincides with $E_e(x,y)$ on $O(\cU)$ and equality on $P(\cU)$. Let $\mu_0$ be the (quantifier-free) \dfs\ measure induced on $\cU/E$ as above. Then $\mu_0(x=b)=0$ for any $b\in\cU/E$. Indeed,  if $b=[b^*]_E$ for some $b^*\in\cU$ then, since $[b^*]_E$ is finite, we have $\mu_0(x=b)=\mu(E(x,b^*))=0$.

Finally, we show that the theory $T_{\textit{rbg}}$ of the random bipartite graph is a (strong) reduct of $\cU/E$ using $\cL_0$-formulas. Given this, we will obtain a contradiction to \dfs-triviality of $T_{\textit{rbg}}$ (see \cite[Theorem 4.9]{CoGa}). We work with $T_{\textit{rbg}}$ in the language of bipartite graphs with unary predicates $U$ and $V$, and a binary relation $R$ on $U\times V$. We interpret $U=O(\cU)/E$ and $V=P(\cU)\backslash\{e\}$ (note that both $O(x)$ and $P(x)\wedge x\neq e$ are $E$-invariant). We then interpret $R(x,y)$ on $U\times V$ as $R_\phi(x,y)\wedge y\neq e$ where $\phi(x,y)$ is the formula $f_e(x)=f_y(x)$. To check that $\phi(x,y)$ is $E$-invariant, note that if $b\in\cP(\cU)$ and $a,a'\in O(\cU)$ are distinct and $E$-equivalent, then
\[
f_e(a)=f_b(a)\miff a'=f_b(a)\miff a=f_b(a')\miff f_e(a')=f_b(a').
\]

Let us now verify that $(U,V;R)\models T_{\textit{rgb}}$. First, fix finite disjoint sets $X,Y\seq U$. We want to find some $b\in V$ such that $R(x,b)$ holds for all $x\in X$ and $\neg R(y,b)$ holds for all $y\in Y$. Without loss of generality, assume $Y=\{[a_i]:i<n\}$ where $n$ is even. Set $Z=\bigcup_{x\in X\cup Y}x\seq O(\cU)$. Then we have a well-defined partition $\cP=X\cup\{\{a_i,a_j\}:\{i,j\}\in\cS\}\cup\{\{f_e(a_{2i}),f_e(a_{2i+1})\}:i<n/2\}$ of $Z$ into two-element sets. So there is $b\in\cP(\cU)\backslash\{e\}$ such that $E_b$ partitions $Z$ according to $\cP$. Then $b$ satisfies the desired properties. 

Now suppose $X,Y\seq\cP(\cU)$ are finite and disjoint. We want to find some $a\in O(\cU)$ such that $R([a],x)$ holds for all $x\in X$ and $\neg R([a],y)$ holds for all $y\in Y$. Since $X\cup\{e\}$ is still disjoint from $Y$, there is $a\in O(\cU)$ such that $f_x(a)=f_y(a)$ for all $x,y\in X\cup\{e\}$, and $f_y(a)\neq f_e(a)$ for all $y\in Y$. Then $a$ satisfies the desired properties. 
\end{proof}

\begin{remark}\label{rem:Tfeq2type}
In \cite[Remark 5.2]{CoGa}, it is claimed that any unary definable subset of the object sort $O$ of $T^*_{\feq 2}$ is finite or cofinite. 
This was used  to justify the claim in \cite[Example 1.7]{ACPgs}  that $T^*_{\feq 2}$ admits a global generically stable type $p$ such that $p^{(2)}$ is not generically stable.  
While $p$ is not explicitly defined in \cite{ACPgs}, it is implied to be the unique non-algebraic global type in $O$, the existence of which is equivalent to the remark from \cite{CoGa} described above. But this remark is false, e.g., consider instances of the formula $\phi(x;y,z)$ given by $f_y(x)=f_z(x)$. 
\end{remark}

\subsection{Henson graphs}\label{sec:Henson}

Fix $s\geq 3$ and let $T_s$ denote the theory of the generic $K_s$-free graph, in the language with a binary relation symbol $E$. Let $\cU\models T_s$ be a monster model. By quantifier elimination, we have a unique non-realized type $p_E\in S_1(\cU)$ containing $\neg E(x,b)$ for all $b\in\cU$. In \cite{CoGa}, it is proved that $p_E$ is \fam, but not \fim. The proof of \fam\ was a combinatorial argument relying on growth rates of certain Ramsey numbers. On the other hand, the failure of \fim\ for $p_E$ is easy to see (modulo the equivalence of \fim\ and generic stability for types), since one can clearly witness the order property for $E(x,y)$ using Morley sequences in $p_E$. 

The first goal of this section is to show that $p_E$ commutes with any invariant measure in $T_s$ (as promised in Example \ref{ex:com2}).  First, we state a well-known Borel-Cantelli-type result on finitely additive probability measures.

\begin{fact}\label{fact:BCL}
For any $\epsilon>0$ and $k\geq 1$, there is some $\delta>0$ and $n\geq 1$ such that the following holds.  Let $B$ be a Boolean algebra, and $\mu$ a finitely additive probability measure on $B$. Suppose $x_1,\ldots,x_n\in B$ and $\mu(x_i)\geq\epsilon$ for all $i\in [n]$. Then there is a $k$-element set $I\seq[n]$ such that $\mu(\bigsqcap_{i\in I}x_i)\geq\delta$.
\end{fact}

A standard consequence of the previous fact is that if $\mu\in \kM_x(\cU)$ is an $M$-invariant measure (in any theory) and $\phi(x)$ is an $\cL_{\cU}$-formula that forks over $M$, then $\mu(\phi(x))=0$. 

\begin{lemma}\label{lem:Enull}
Suppose $\mu\in \kM_1(\cU)$ is invariant over $M\prec\cU$. Then $\mu(E(x,a))=0$ for some/any $a\models p_E|_M$. 
\end{lemma}
\begin{proof}
Consider the formula $\phi(x,y_1,\ldots,y_{s-1})\coloneqq \bigwedge_{i=1}^{s-1}E(x,y_i)$. Then for any pairwise distinct $a_1,\ldots,a_{s-1}\models p_E|_M$, the formula $\phi(x,\abar)$ forks over $M$ by \cite[Corollary 4.8]{Co13}, and thus $\mu(\phi(x,\abar))=0$. Now let $\epsilon=\mu(E(x,a))$ where $a$ is some/any realization of $p_E|_M$. Toward a contradiction, suppose $\epsilon>0$. Let $n\geq1$ and $\delta>0$ be as in Fact \ref{fact:BCL}, with $k=s-1$. Choose pairwise distinct  $a_1,\ldots,a_n\models p_E|_M$. Then there is an $(s-1)$-element set $I\seq[n]$ such that $\mu(\phi(x,\abar_I))\geq\delta>0$, which is a contradiction. 
\end{proof}

\begin{corollary}
Suppose $\mu\in\kM_n(\cU)$ is invariant. Then $p_E\otimes\mu=\mu\otimes p_E$. 
\end{corollary}
\begin{proof}
Let $\nu_1=p_E\otimes\mu$ and $\nu_2=\mu\otimes p_E$. It suffices to show that $\nu_1$ and $\nu_2$ agree on formulas that are conjunctions of atomics and negated atomics. So fix such an $\cL_{\cU}$-formula $\phi(x,\ybar)$, where $|\ybar|=n$. Without loss of generality, $\phi(x,\ybar)$ is of the form 
\[
\bigwedge_{i=1}^n E^{\epsilon_i}(x,y_i)\wedge \bigwedge_{i=1}^n (x=_i y_i)\wedge \psi(x)\wedge\theta(\ybar)
\]
where $\epsilon_i\in\{0,1\}$, $E^1$ is $E$, $E^0$ is $\neg E$, $=_i$ is either $=$ or $\neq$, and $\psi(x)\wedge \theta(\ybar)$ is some $\cL_{\cU}$-formula. Fix $M\prec\cU$ such that $\phi(x,\ybar)$ is over $M$ and $\mu$ is invariant over $M$. 

Given $q\in S_{\ybar}(M)$, we have
\[
F^\phi_{p_E}(q)=\begin{cases}
0 & \text{if some $\epsilon_i=1$, some $=_i$ is $=$, $\theta(\ybar)\not\in q$, or $\psi(x)\not\in p$}\\
1 & \text{otherwise.}
\end{cases}
\]
Therefore
\[
\nu_1(\phi(x,y))=\begin{cases}
0 & \text{if some $\epsilon_i=1$, some $=_i$ is $=$, or $\psi(x)\not\in p$}\\
\mu(\theta(\ybar)) & \text{otherwise.}
\end{cases}
\]

On the other hand, we have $\nu_2(\phi(x,y))=\mu(\phi(a,\ybar))$, where $a\models p_E|_M$. If some $\epsilon_i=1$ then $\mu(\phi(a,\ybar))=0$ by Lemma \ref{lem:Enull}. If some $=_i$ is $=$, then $\mu(\phi(a,\ybar))=0$ since $a\not\in M$ and $\mu$ is $M$-invariant. If $\psi(x)\not\in p$, then clearly $\mu(\phi(a,\ybar))=0$. So we may assume all $\epsilon_i$ are $0$, all $=_i$ are $\neq$, and $\psi(x)\in p$. Since $\mu(\neg E(a,y_i))=\mu(a\neq y_i)=1$ for all $i$, we have $\mu(\phi(a,\ybar))=\mu(\theta(\ybar))$. So $\nu_1(\phi(x,\ybar))=\nu_2(\phi(x,\ybar))$. 
\end{proof}

In \cite[Theorem 5.10]{CoGa}, the first two authors made two more assertions about Keisler measures in $T$. First, it was claimed that a measure $\mu\in\kM_1(\cU)$ is \dfs\ if and only if it is \fam, and in this case $\mu$ is a convex combination of $p_E$ and a trivial measure. Second, it is claimed that every \fim\ measure in $\kM_1(\cU)$ is trivial.  As indicated by Remark \ref{rem:CGcor}, both statements relied on an erroneous remark concerning localization of measures, and the first statement is easily recovered using the corrected proof of Proposition \ref{prop:dfstriv1}. On the other hand, the second result is more complicated, and so we take the opportunity here to provide a correct proof.

\begin{theorem}\label{thm:Hfimtriv}
Any \fim\ measure in $\kM_1(\cU)$ is trivial.
\end{theorem}
\begin{proof}
Suppose $\mu\in\kM_1(\cU)$ is \fim. By \cite[Theorem 5.10$(a)$]{CoGa}, we can write $\mu=rp_E+(1-r)\nu$ for some $r\in[0,1]$ and trivial $\nu\in\kM_1(\cU)$. Toward a contradiction, suppose $r>0$. 
Set $\epsilon=r/2$ and fix $n\geq \frac{r(1-r)}{\epsilon^2(1-\epsilon)}$. Since $\mu$ is \fim, there is a formula $\theta(x_1,\ldots,x_n)$ such that $\mu^{(n)}(\theta(\xbar))\geq 1-\epsilon$ and if $\abar\models\theta$ then $\mu\approx^E_\epsilon \Av(\abar)$. 

Fix $M\prec\cU$ such that $\theta(\xbar)$ is over $M$ and $\nu$ is a weighted sum of Dirac measures at points in $M$. By quantifier elimination, we may write
$$
\theta(\xbar)=\bigvee_{t=1}^k \psi_t(\xbar)
$$
where each $\psi_t(\xbar)$ is a consistent conjunction of atomic and negated atomics. Given $1\leq t\leq k$, call a set $X\seq[n]$ \emph{$t$-good} if $\psi_t(\xbar)$ does not prove a formula of the $x_i=b$ for some $i\in X$ and $b\in M$, or of the form $E(x_i,x_j)$ for some $i,j\in X$. 

Suppose first that, for some $1\leq t\leq k$, we have a $t$-good set $X\seq[n]$ of size at least $\epsilon n$. Then we can find a realization $\abar\models\theta(\xbar)$ such that $a_i\not\in M$ for all $i\in X$, and $\neg E(a_i,a_j)$ for all $i,j\in X$. We may then choose $b\in\cU$ such that $E(a_i,b)$ holds for all $i\in X$ and $\neg E(b,m)$ holds for all $m\in M$. Note that $\mu(E(x,b))=0$. On the other hand, $\Av(\abar)(E(x,b))\geq|X|/n\geq\epsilon$, which contradicts the choice of $\theta(\xbar)$. 

So now we can assume that for all $1\leq t\leq k$, any $t$-good set $X\seq[n]$ has size strictly less than $\epsilon n$. Given $X\seq [n]$, set 
\[
r_X=r^{|X|}(1-r)^{n-|X|}\makebox[.5in]{and} \mu_X=\bigotimes_{i=1}^{n} \left\{\begin{matrix}p_E, & i \in X \\ \nu, & i \notin X \end{matrix}\right\}.
\] 
By Fact \ref{fact:binomial} and choice of $n$, we have 
\begin{align*}
\mu^{(n)}(\theta(\xbar))=\sum_{X\seq[n]}r_X\mu_X(\theta(\xbar)) &\leq \sum_{X\not\in \cP_{r,\epsilon}(n)}r_X+\sum_{X\in\cP_{r,\epsilon}(n)}r_X\mu_X(\theta(\xbar))\\
 &<1-\epsilon+\sum_{X\in \cP_{r,\epsilon}(n)}r_X\mu_X(\theta(\xbar)).
\end{align*}
Now fix $X\in \cP_{r,\epsilon}(n)$. Then $|X|\geq rn/2=\epsilon n$, and so $X$ is not $t$-good for any $1\leq t\leq k$. Fix $1\leq t\leq k$. Then for any $t$, either $\psi_t(\xbar)$ contains a conjunct $x_i=b$ for some $i\in X$ and $b\in M$, or a conjunct $E(x_i,x_j)$ for some $i,j\in X$. In the first case we have $\mu_X(\psi_t(\xbar))\leq \mu_X(x_i=b)=p_E(x_i=b)=0$; and in the second case we have $\mu_X(\psi_t(\xbar))\leq\mu_X(E(x_i,x_j))=(p_E\otimes p_E)(E(x_i,x_j))=0$. Altogether, we have $\mu_X(\psi_t(\xbar))=0$ for all $1\leq t\leq k$, and so $\mu_X(\theta(\xbar))=0$. By the inqualities above, it follows that $\mu^n(\theta(\xbar))< 1-\epsilon$, which contradicts the choice of $\theta(\xbar)$. 
\end{proof}

\subsection{A new example of a \fam\ and non-\fim\ complete type}

In this section, we show that a certain reduct of the \dfs\ and non-\fam\ type built in Section \ref{sec:Ti} is \fam\ and non-\fim. First, we prove a  technical lemma regarding \fam\ types in the presence of quantifier elimination.

\begin{lemma} \label{lem:famQE}
Assume $T$ has quantifier elimination, and fix $p \in S_{x}(\mathcal{U})$. Suppose there exists a sequence of tuples $(\overline{c}_{n})_{n\in\omega}$ such that for every atomic formula $\theta(x,\overline{y})$ and every $\epsilon>0$ there exists $N(\epsilon,\theta)$ so that for all $n>N(\epsilon,\theta)$, $p \approx_{\epsilon}^{\theta}\Av(\overline{c}_n)$. Then $p$ is finitely approximated over any small model containing $(\overline{c}_n)_{n \in \omega}$.
\end{lemma}

\begin{proof} We first note that for any formula $\psi(x,\overline{y})$ and any tuple $\overline{a}$ of points in $\mathcal{U}^{x}$, we have $p \approx_{\epsilon}^{\psi} \Av(\overline{a})$ if and only if $p \approx_{\epsilon}^{\neg \psi} \Av(\overline{a})$.  Let $\gamma(x, \overline{y}) = \bigwedge_{j \in J} \theta_{j}(x,\overline{y})$ where for each $j$, the formula $\theta_{j}(x,\overline{y})$ is either an atomic formula or the negation of an atomic formula. Fix $\epsilon > 0$. For each $\theta_{j}(x,\overline{y})$, choose $N_j = N(\epsilon/|J|,\theta_j)$ as in the statement of the lemma and fix $n> \max\{N_j: j \in J\}$. First, assume that $\gamma(x,\overline{b}) \in p$. For each $j \in J$, $\theta_j(x,\overline{b}) \in p$ and so

\begin{multline*}
 \Av(\overline{c}_{n})(\gamma(x,\overline{b}))=1 - \Av(\overline{c}_{n})\Big(\bigvee_{j\in J}\neg\theta_{j}(x,\overline{b})\Big)\\ \geq 1 - \sum_{j \in J} \Av(\overline{c}_n)(\neg\theta_{j}(x,\overline{b})) 
 \geq 1 - \sum_{j \in J} \frac{\epsilon}{|J|} = 1 - \epsilon.
\end{multline*} 
On the other hand, if $\neg \gamma(x,\overline{b}) \in p$, then there exists some fixed $j \in J$ such that $\theta_{j}(x,\overline{b}) \not \in p$. Since $p \approx_{\epsilon/|J|}^{\theta_j}\Av(\overline{c}_n)$, we have $ \Av(\overline{c}_{n})(\gamma(x,\overline{b})) \leq \Av(\overline{c}_n)(\theta_{j}(x,\overline{b})) \leq \epsilon$. 

Now assume that $\rho(x,\overline{y}) = \bigvee_{i \in I} \gamma_{i}(x,\overline{y})$ where each $\gamma_{i}(x,\overline{y})$ is as before (i.e. a conjunction of atomic and negated atomic formulas). By the previous paragraph, we can choose $m \in \omega$ so that for any $i \in I$, then $p\approx_{\epsilon/|I|}^{\gamma_i} \Av(\overline{c}_m) $. First, assume that $\rho(x,\overline{b}) \in p$. Then there exists some fixed $i \in I$ so that $\gamma_{i}(x,b) \in p$. So $\Av(\overline{c}_m)(\rho(x,\overline{b})) \geq \Av(\overline{c}_m)(\gamma_i(x,\overline{b})) \geq 1-\epsilon$. Finally, assume that $\neg \rho(x,\overline{b}) \in p$. So for each $i \in I$, we have that $\neg \gamma_i(x,\overline{b}) \in p$. Then, we have the following computation:

\begin{equation*} \Av(\overline{c}_m)(\rho(x,\overline{b})) \leq \sum_{i \in I} \Av(\overline{c}_m)(\gamma_i(x,\overline{b})) < |I|\Big(\frac{\epsilon}{|I|}\Big) = \epsilon. \qedhere
\end{equation*}
\end{proof}

Let $\Lc_{PQ}$ be the reduct of the language described in Section~\ref{sec:Ti} to just the $P$ and $Q$ sort. Let $q_{PQ}(y)$ be the reduct of the type from Corollary~\ref{cor:dfsnotfam} to the language $\Lc_{PQ}$. By Proposition~\ref{prop:PQomega}, the reduct $T_{PQ}\coloneqq \Ti| \Lc_{PQ}$ has quantifier elimination. From this it is not hard to show that $q_{PQ}(y)$ is axiomatized the formulas
\begin{enumerate}[\hspace{5pt}$\ast$]
\item $a \sqin y$ for every $a \in P(\Uc)$,
\item $y \neq \top$, and
\item $y \not \sim b$ for every $b \in Q(\Uc)$.
\end{enumerate}
By quantifier elimination for $T_{PQ}$, and essentially the same arguments as in Section \ref{sec:Ti}, $q_{PQ}$ determines a unique $\emptyset$-definable complete type, which is finitely satisfiable in any small model. However, we will now show that by dropping the measure sort, $q_{PQ}$ in fact becomes \fam, but is still not \fim.

\begin{proposition}
  $q_{PQ}$ is \fam\ and not \fim.
\end{proposition}
\begin{proof}
We first show \fam.
  Fix $n<\omega$. For each $i,j < n$, we let $d_{i,j} = (i,\lb \frac{j}{n},\frac{j+1}{n})\rp\in Q(\Mi)$ and define the tuple $\cbar_n =(d_{i,j}\cop)_{i,j<n}$. 
  We show that $(\cbar_n)_{n<\omega}$ satisfies the conditions of
 Lemma \ref{lem:famQE} with respect to the type $q_{PQ}$. By quantifier elimination (Proposition~\ref{prop:PQomega}), we can then conclude that $q_{PQ}$ is \fam.

  First, note that for any $a \in P(\Uc)$, we clearly have $\Av(\bar{c}_{n})(a \sqin y) = \frac{n-1}{n}$. So the conditions of Lemma \ref{lem:famQE} are satisfied for the atomic formula $x\sqin y$. Now consider an atomic formula of the form $t(\xbar,y)\asymp s(\xbar,y)$ where $\asymp$ is either $=$ or $\sim$ and $t$ and $s$ are terms in the $Q$ sort. For any tuple, $\bbar$ of elements of $Q(\Uc)$ it is not too hard to see that if $d$ is an element of $Q(\Uc)\setminus\{\top,\bot\}$ that is not $\sim$-equivalent to any $b_i$, then  $t(\bbar,d)\asymp s(\bbar,d)$ holds if and only if $t(\bbar,y)\asymp s(\bbar,y) \in q_{PQ}(y)$. Note that for any tuple $\bbar$ of elements of $Q(\Uc)$, at most $|\xbar|=|\bbar|$ of the elements of $\cbar_n$ can be $\sim$-equivalent to some $b_i$. This implies that we always have $\Av(\cbar_n)(t(\bbar,y)\asymp s(\bbar,y)) \geq \frac{n - |\xbar|}{n}$ if $q_{PQ}(y)$ satisfies $t(\bbar,y)\asymp s(\bbar,y)$ and $\Av(\cbar_n)(t(\bbar,y)\asymp s(\bbar,y)) \leq \frac{|\xbar|}{n}$ if $q_{PQ}(y)$ does not satisfy $t(\bbar,y)\asymp s(\bbar,y)$.

Finally, we show $q_{PQ}$ is not \fim. Recall that for types, \fim\ is equivalent to generically stable (see \cite[Section 3]{CoGa}). Let $(b_i)_{i<\omega}$ be a Morley sequence in $q_{PQ}(y)$ over some set of parameters. Then $b_i\not\sim b_j$ for all $i\neq j$. It follows (using compactness) that for any $I \subseteq \omega$, there is an $a_I \in P(\Uc)$ such that for any $i<\omega$, $a_I \sqin b_i$ if and only if $i \in I$. Therefore $q_{PQ}(y)$ is not generically stable and so not \fim.
\end{proof}

\section{Concluding remarks}

A recurring theme in the previous work is that, outside of NIP theories, the study of Keisler measures is much more complicated and requires confrontation with a greater amount of pure measure theory. We have also seen that much of the aberrant behavior involving Morley products and Borel definable measures can be found in a very straightforward simple unstable theory, namely, the random ternary relation (see Proposition \ref{prop:bern} for another example of bad behavior in this theory). This suggests that a coherent study of Keisler measures in simple theories may need to focus on very different questions, as compared to NIP theories. On the other hand, since our counterexamples were all built using a generic \emph{ternary} relation, perhaps it is possible to recover some good behavior in the setting of $2$-dependent theories (see \cite{ChPaTa} for the definition of $k$-dependence). 

\begin{question}
Is the product of two Borel definable Keisler measures in a $2$-dependent theory again Borel definable? If so, does associativity hold for Borel definable measures in $2$-dependent theories?
\end{question}

Despite the bad news for Borel definability, our results on notions of `generic stability' for measures corroborate the philosophy of \cite{CoGa} that interesting results exist outside of NIP. For example, we have shown further evidence that \fim\ measures are a sufficiently well-behaved class in general theories. Moreover, results such as the weak law of large numbers continue to be effective tools for studying \fim\ measures outside of NIP. However, these developments are somewhat dampened by the fact that, while we have now found interesting and exotic \dfs\ and \fam\ measures in independent theories, there is a concerning dearth of examples of nontrivial \fim\ measures. As for \dfs\ and \fam, our work in Section \ref{sec:com} shows that some nice behavior can be recovered, and several interesting open questions remain. In particular, our results further highlight the power of Keisler's original result on the existence of smooth extensions in NIP theories, and we have demonstrated that this phenomenon remains powerful without a global NIP assumption. More specifically, we have shown that several results about measures in NIP theories generalize to measures in arbitrary theories, as along as one assumes that the measures in question admit extensions with various properties exhibited by smooth measures.

\begin{appendices}

\appendix

\renewcommand*{\thesection}{}
\renewcommand*{\thesubsection}{\Alph{section}.\arabic{subsection}}
\renewcommand{\thetheorem}{\Alph{section}.\arabic{theorem}}
\section{}

\subsection{Borel measures}\label{app:Borel}
Let $(X,\Sigma)$ be a \textbf{measure space}, i.e., $X$ is a set and $\Sigma$ is a $\sigma$-algebra of subsets of $X$. A function $f\colon X\to \lb 0,\infty\rp$ is \textbf{$\Sigma$-measurable} if $f\inv(U)\in\Sigma$ for all open $U\seq \lb 0,\infty\rp$. A $\Sigma$-measurable function on $X$ is \textbf{$\Sigma$-simple} if its image is finite. 

\begin{fact}\label{fact:Bapprox}
If $f\colon X\to \lb 0,\infty\rp$ is $\Sigma$-measurable then there is a sequence $(f_n)_{n=1}^\infty$ of $\Sigma$-simple functions converging pointwise to $f$. Moreover, $(f_n)_{n=1}^\infty$ converges uniformly to $f$ on any subset of $X$ for which $f$ is bounded.
\end{fact}

\begin{remark}
If $f\colon X\to [0,1]$ is $\Sigma$-measurable, then in the previous fact one can take $f_n=\sum_{i=0}^{n-1}\frac{i}{n}\boldsymbol{1}_{B_i}$, where $B_i=f\inv(\lp\frac{i}{n},\frac{i+1}{n}\rb)$ for $0\leq i\leq n-1$. Indeed, for any $n\geq 1$, we have $\|f-f_n\|_\infty\leq\frac{1}{n}$. 
\end{remark}

Now assume $X$ is a compact Hausdorff space and $\Sigma$ is the $\sigma$-algebra of Borel subsets of $X$. Let $\mu$ be a \textbf{Borel measure} on $X$, i.e., a countably additive function $\mu\colon \Sigma\to [0,1]$ such that $\mu(\emptyset)=0$. We call $\mu$ a \textbf{Borel probability measure} if, moreover, $\mu(X)=1$. Also, $\mu$ is called \textbf{regular} if, for any Borel set $B\seq X$,
\[
\sup\{\mu(C):C\seq B,~C\text{ is closed}\}=\mu(B)=\inf\{\mu(U):B\seq U,~U\text{ is open}\}.
\]
Given a continuous surjective map $\rho\colon X\to Y$, with $Y$ compact Hausdorff,  and a regular Borel measure $\mu$ on $X$, the \textbf{pushforward of $\mu$ along $\rho$} is the Borel  measure $\nu$ on $Y$ defined by $\nu(B)=\mu(\rho\inv (B))$ for any Borel $B\seq Y$. 

\begin{fact}\label{fact:push}
Suppose $\rho\colon X\to Y$ is a continuous surjective function between compact Hausdorff spaces, and $\mu$ is a regular Borel probability measure on $X$. Then the pushforward of $\mu$ along $\rho$ is regular.
\end{fact}

Finally, we note some facts about totally disconnected compact Hausdorff spaces.

\begin{fact}\label{fact:BStone}
Suppose $X$ is a totally disconnected compact Hausdorff space, and let $\mu$ be a regular Borel probability measure on $X$.
\begin{enumerate}[$(a)$]
\item If $U\seq X$ is open, then $\mu(U)=\sup\{\mu(K):K\seq U,~\text{$K$ is clopen}\}$.
\item If $\nu$ is a regular Borel probability measure on $X$, and $\nu(K)=\mu(K)$ for all clopen $K\seq X$,  then $\mu=\nu$.
\end{enumerate}
\end{fact} 
\begin{proof}
Part $(a)$ is straightforward. Part $(b)$ follows from part $(a)$ and regularity.
\end{proof}

\subsection{Measures on independent sets}\label{app:TR}
The primary goal of this section is to construct the measure defined in the proof of Lemma \ref{lem:LebTR}. We take a somewhat broader approach of independent interest.  Let $B$ be a  Boolean algebra, with join, meet,  complement, top element, bottom element, and induced partial order denoted $\sqcup$, $\sqcap$, $\cop$, $\top$, $\bot$, and $\sqseq$, respectively. Given $X\seq B$, let $X\cop =\{x\cop :x\in X\}$ and, assuming $X$ is finite, let $\bigsqcap X=\bigsqcap_{x\in X} x$.

\begin{definition}
A subset $F\seq B$ is \textbf{independent} if $\bigsqcap X\sqcap \bigsqcap Y\cop \neq \bot$ for any finite disjoint $X,Y\seq F$.
\end{definition}

The following lemma is certainly well-known, but we were unable to find  a suitable reference.

\begin{lemma}\label{lem:IE1}
Suppose $F\seq B$ is independent, and let $f\colon F\to [0,1]$ be a function. Then there is a finitely additive probability measure $\mu$ on $B$ such that, for any finite disjoint $X,Y\seq F$,
\[
\textstyle \mu\left(\bigsqcap X\sqcap \bigsqcap Y\cop\right)=\prod_{x\in X} f(x)\cdot\prod_{x\in Y} (1-f(x)).
\]
\end{lemma}
\begin{proof}
We first observe that we can reduce to the case that $F$ is finite and generates $B$. Indeed, given a finite subset $E\seq F$, let $B_E$ be the sub-algebra generated by $E$. Suppose that for all finite $E\seq F$, we have a finitely additive probability measure $\mu_E$ on $B_E$ satisfying the desired conclusion for all finite disjoint $X,Y\seq E$. We can extend each $\mu_E$ arbitrarily to some finitely additive probability measure $\mu^*_E$ on $B$ (e.g., by \cite{LosMar}; see also \cite[Theorem 3.7]{StarBour}). Then $(\mu^*_E)_E$ is a net in the compact space of all finitely additive probability measures on $B$, and thus has a subnet converging to some measure $\mu$ with the desired properties.

So now assume $F$ is finite and generates $B$. Let $n=|F|$. Since $F$ is independent, $B$ has $2^n$ atoms, which are precisely the elements of the form $a_{X} \coloneqq 
\bigsqcap X\sqcap \bigsqcap (F\backslash X)\cop$ for $X\seq F$. A direct calculation show that the unique measure $\mu$ satisfying $\mu(a_{X}) = \prod_{x \in X} f(x) \cdot \prod_{x \in F \backslash X}(1 - f(x))$ has the desired properties.

Thus we can view $B$ as the event space of the experiment of flipping $n$ independent coins (identified with the elements of $F$). If we assign $x\in F$ the probability $f(x)$ of landing heads, then the resulting probability function is the desired finitely additive measure on $B$. 
\end{proof}

\begin{corollary}\label{cor:IE2}
Let $T$ be a complete theory with monster model $\cU$, and suppose $\cF\seq\Def_x(\cU)$ is independent. 
Then, for any $f\colon \cF\to [0,1]$ there is some $\mu\in\kM_x(\cU)$ such that, for any finite disjoint $\cX,\cY\seq\cF$, 
\[
\textstyle \mu\left(\bigcap_{A\in \cX}A\cap\bigcap_{B\in \cY}\neg B \right)=\prod_{A\in \cX} f(A)\cdot\prod_{B\in \cY} (1-f(B)).
\]
\end{corollary}
\begin{proof}
This follows directly from Lemma \ref{lem:IE1}. 
\end{proof}

We now return to the theory $T_R$ of the random ternary relation $R$, defined in Section \ref{sec:TR}. Let $\cU\models T_R$ be a monster model. Let $\cF$ be the collection of (positive) instances of $R$ in one free variable. Then $\cF$ is independent by the extension axioms for $T_R$. So we can apply Corollary \ref{cor:IE2} with the constant $\frac{1}{2}$ function to obtain a measure $\lambda\in\kM_1(\cU)$ such that if $\theta_1(x),\ldots,\theta_n(x)$ are pairwise distinct (positive) instances of $R$ in one free variable, and $\psi_i(x)$ is either $\theta_i(x)$ or $\neg\theta_i(x)$, then  
\[
\lambda(\psi_1(x)\wedge\ldots\wedge\psi_n(x))=\frac{1}{2^n}.
\]
This finishes the construction of the measure defined in the proof of Lemma \ref{lem:LebTR}. 

We can use a similar construction to justify a claim made after Example \ref{ex:mu-meas}. Fix any countably infinite set $A\subset\cU$. Define $\nu\in\kM_1(\cU)$ in the same way as $\lambda$ above, except start by insisting that any  instance of $R$ involving only parameters from $A$ has $\nu$-measure $0$, and all other instances of $R$ have $\nu$-measure $\frac{1}{2}$. Now view $\nu$ as a measure in $\kM_y(\cU)$. Fix a Bernstein set $Z\seq S_{yz}(A)$, and define $p\in S_x(\cU)$ such that the positive instances of $R$ in $p$ are precisely those of the form $R(x,b,c)$, where $\tp(b,c/A)\in Z$. Note that $p$ and $\nu$ are $A$-invariant. 

\begin{proposition}\label{prop:bern}
The type $p$ is $\nu$-measurable over $A$, but not $\nu$-measurable over any proper extension $B\supset A$.
\end{proposition}
\begin{proof}
Note that $\nu|_A$ coincides with the unique type in $S_y(A)$ that contains the negation of any instance of $R$ involving $y$ and parameters from $A$. So $p$ is $\nu$-measurable over $A$ since any $f\colon S_y(A)\to [0,1]$ is $\nu|_A$-measurable. 

Now fix a proper extension $B\supset A$, and fix some $c\in B\backslash A$. Let $D=dp(R(x,y,c))\coloneqq \{q\in S_y(B):R(x,b,c)\in p\text{ for some }b\models q\}$. We claim that $D$ is not $\nu|_B$-measurable, and thus $p$ is not $\nu$-measurable over $B$. First, since $c\not\in A$, and $A$ is infinite, it follows that $\nu|_{Ac}$ is strongly continuous (as in the proof of Lemma \ref{lem:LebTR}). Let $\rho=\rho^y_{B,Ac}$ and let $f\colon S_y(Ac)\to S_{yz}(A)$ be the natural inclusion map. Then $\rho(D)=f\inv(Z)$, and $D=\rho\inv(\rho(D))$ (by $A$-invariance of $p$). So, by Lemma \ref{lem:bern}, it suffices to show that $f\inv(Z)$ is a Bernstein set in $S_y(Ac)$. To see this, note that  $X\coloneqq f(S_y(Ac))$  is a closed set in $S_{yz}(A)$, whence $Z \cap X$ is a Bernstein set in $X$, and so $f\inv(Z)$ is a Bernstein set in $S_y(Ac)$ as well.
\end{proof}

\subsection{Quantifier elimination for $\Ti$}\label{app:QE}

In this section, we prove that theory $\Ti$ defined in Section \ref{sec:Ti} has quantifier elimination (this was stated in Theorem \ref{thm:TiQE}).

Recall that $\cH$ is the Boolean algebra on $\lb 0,1\rp$ generated by sets of the form $\lb a,b\rp$ with $0\leq a < b \leq 1$. Note that every element of $\cH$ can be expressed as a (possibly empty) finite union of half-open intervals of this form, and so in particular $\cH$ contains no singletons and is an atomless Boolean algebra. Recall also that $\lambda$ denotes the Lebesgue measure on $\lb 0,1\rp$.

We start with the following easy fact.

\begin{fact}\label{fact:duh}
  For any $X \in \cH$, any finite $A \subset X$, and any real number $r \in (0,\lambda(X))$, there is $Y\in \cH$ such that $A \subset Y\subset X$ and $\lambda(Y)=r$.
\end{fact}

\begin{lemma}\label{lem:fo-gen}
  For any finite set $B\subset Q(\Mi)$, the substructure of $Q(\Mi)$ generated by $B$ is exhausted by elements of the form $\bigsqcup C$ where $C$ is some set of elements of the form $\bigsqcap D$ for some $D \subseteq B\cup \{b\cop : b \in B\}$ (where $\sqcup \varnothing = \bot$ and $\sqcap \varnothing = \top$).

  In particular, for any set of variables $\bar{x}$ of sort $Q$, there is a fixed finite list of $\cL_Q$-terms $\{t_i(\bar{x})\}_{i<n}$ such that for any $\bar{b} \in Q(\Mi)$, $\{t_i(\bar{b})\}_{i<n}$ exhausts the $\cL_Q$-substructure of $Q(\Mi)$ generated by $B$. 
\end{lemma}
\begin{proof}
  The first statement is easy to check, so for the second statement we may take $\{t_i(\bar{x})\}_{i<n}$ to be an enumeration of all disjunctive normal form formal Boolean combinations of the variables $\bar{x}$.
\end{proof}

Let $\Lc_{PQ}$ be the sub-language of $\Lc$ obtained by removing the sort $R$ and all associated symbols, and let $T_{PQ}$ denote the reduct of $\Ti$ to $\Lc_{PQ}$.

\begin{corollary}
For any $M\models \Ti$,  $Q(M)$ is locally finite with respect to $\cL_{PQ}$ (i.e., every finite subset of $Q(M)$ generates a finite $\cL_{PQ}$-substructure).
\end{corollary}

Let $M$ be a model of $\Ti$. Then $\sqcap$ and $\sqcup$ are lattice operations on $Q(M)$. We will denote the induced partial order by $\sqsubseteq$. Furthermore, $\sim$ is an equivalence relation on $Q_0(M)\coloneqq Q(M)\backslash\{\bot,\top\}$. Given $b\in Q_0(M)$, we set $[b]_\sim=\{c\in Q_0(M):b\sim c\}$ and $[b]^*_\sim\coloneqq [b]_\sim\cup\{\bot,\top\}$. We refer to $[b]_\sim$ as the \emph{$\sim$-class} of $b$. Note that $([b]^*_\sim,\sqcap,\sqcup,\,\cop,\bot,\top)$ is an atomless Boolean algebra for any $b\in Q_0(M)$. We also emphasize that all of this notation depends implicitly on the ambient model $M$.

\begin{definition}\label{def:minimal}
Given a finite substructure $B\subset Q(M)$, an element $b \in B$ is \textbf{minimal} if it is not $\bot$ and is minimal with regards to the partial order $\sqsubseteq$. 
\end{definition}

\begin{lemma}\label{lem:min-in}
  For any $M \models \Ti$, any finite $\cL_Q$-substructure $B \subset Q(M)$, and any $a \in P(M)$, there is, for each $\sim$-class $C$ with representatives in $B$, a unique minimal $b \in B \cap C$ such that $a \sqin b$.
\end{lemma}
\begin{proof}
  This is clearly a first-order property that holds in $\Mi$. 
\end{proof}

We will need the following lemma and proposition for quantifier elimination of $\Ti$.

\begin{lemma}\label{lem:sep-points}
  Let $M$ be a model of $T_{PQ}$. Let $Y \subset P(M)$ be some finite set. For any $\sim$-class $C$ of $M$, and any sequence $c_0,\dots,c_{n-1} \in C\cup\{\top\}$ with $c_i \sqcap c_j = \bot$ for each $i<j<n$ and with $\bigsqcup_{i<n}c_i = \top$, there exists a family $\{d_a\}_{a \in Y}$ of elements of $C$ such that
  \begin{enumerate}[\hspace{5pt}$\ast$]
  \item for any distinct $a,a' \in Y$, $d_a\sqcap d_{a'}=\bot$,
  \item for each $a \in Y$, $d_a \sqsubset c_i$ for some $i<n$,
  \item for each $i<n$, $\bigsqcup \{d_a : d_a \sqsubseteq c_i\} \sqsubset c_i$, and
  \item for each $a \in Y$, $a \sqin d_a$.
  \end{enumerate}
\end{lemma}
\begin{proof}
  First note that for each fixed $n$, the stated property is equivalent to some sentence in $\Lc_{PQ}$, so it is sufficient to show that the property holds in $M$.

  So fix a finite set $Y \subset P(M)$, a $\sim$-class $C$ in $M$, and a finite sequence $c_0,\dots,c_{n-1} \in C$ with $c_i\sqcap c_j = \bot$ for each $i<j<n$ and with $\bigsqcup_{i<n}c_i = \top$. Let $k\in \N$ be such that $C$ is $\{k\}\times \cH_0$.

  For each $i<n$, let $Y_i=\{a\in Y:a \sqin c_i\}$. Fix some $i < n$. Call an element of $[0,1]$ a \emph{special point of $c_i$} if it either is $a(k)$ for some in  $a\in Y_i$ or is in one of the righthand endpoints of one of the constituent intervals in the subset of $\lb 0,1\rp$ determined by $c_i$ (where if $c_i=\top$ then this is $\lb 0,1\rp$). Note that by our choice of $P(M)$, we have $a(k) \neq a'(k)$ for any distinct $a,a' \in Y_i$. 
  So each element of $Y_i$ corresponds to a unique special point of $c_i$.

  For each $a \in Y_i$, let $d_a=(k,\lb a(k),\frac{a(k)+r}{2}\rp)$, where $r$ is the smallest special point of $c_i$ strictly greater than $a(k)$. Since no element of $Y_i$ yields a righthand endpoint interval in $c_i$, we clearly have that $d_a \sqsubset c_i$. We also clearly have that $a \sqin d_a$, that for any distinct $a,a' \in Y_i$, $d_a\sqcap d_{a'} = \bot$, and that $\bigsqcup \{d_a : d_a \sqsubseteq c_i\} = \bigsqcup_{a \in Y_i} d_a \sqsubset c_i$.

  Since $d_a \sqsubset c_i$ for each $a \sqin c_i$, we have that for any distinct $a,a' \in Y$, $d_a\sqcap d_{a'} = \bot$, and so the family $\{d_a\}_{a \in Y}$ fulfills the requirements of the lemma.
\end{proof}

\begin{proposition}\label{prop:PQomega}
  The theory $T_{PQ}$ is $\omega$-categorical and has quantifier elimination.
\end{proposition}
\begin{proof}
  Let $M_0$ and $M_1$ be two countable models of $T_{PQ}$. Let $(Y_0,C_0)$ and $(Y_1,C_1)$ be isomorphic finite sub-structures of $M_0$ and $M_1$, respectively, and let $f\colon (Y_0,C_0)\to (Y_1,C_1)$ be a fixed isomorphism. All we need to show is that
  \begin{enumerate}[\hspace{5pt}$\ast$]
  \item for any $a_0 \in P(M_0)$, there exists an $a_1 \in P(M_1)$ such that $f$ extends to an isomorphism from $(Y_0a_0,C_0)$ to $(Y_1a_1,C_1)$ sending $a_0$ to $a_1$, and
  \item for any $b_0 \in Q(M_0)$, there exists a $b_1 \in Q(M_1)$ such that $f$ extends to an isomorphism from $(Y_0,\left<C_0b_0\right>)$ to $(Y_1,\left<C_1b_1\right>)$ sending $b_0$ to $b_1$.
  \end{enumerate}
 Note that the same follows with $0$ and $1$ switched by symmetry. Also note that since we are only assuming that $f$ is an isomorphism, rather than an elementary map, this is sufficient to show quantifier elimination. 

  Suppose we have $(Y_0,C_0)$, $(Y_1,C_1)$, $f$, and $a_0\in P(M_0)$. Assume that $a_0 \notin Y_0$. By Lemma~\ref{lem:min-in}, for each $\sim$-class $D$ in $C_0$, there is a unique minimal $b_{D} \in D \subset C_0$ such that $a_0 \sqin b_{D}$. The elements $f(b_D)$ (for each $\sim$-class $D$ that intersects $C_0$) are pairwise $\sim$-inequivalent. $\Ti$ says that for any finite set $X$ of pairwise $\sim$-inequivalent elements of $\cQ_0$, there are infinitely many elements of $P$ which are $\sqin$-in each element of $X$. Therefore we may find $a_1\in P(M_1)$ that is $\sqin$-in each $f(b_D)$ and that is not in $Y_1$, and we get that the obvious extension of $f$ from $(Y_0a_0,C_0)$ to $(Y_1a_1,C_1)$ is an isomorphism.

  Suppose we have $(Y_0,C_0)$, $(Y_1,C_1)$, $f$, and $b_0\in Q(M_0)$. Assume that $b_0 \notin C_0$. 
  There are two cases: Either $b_0$ is $\sim$-equivalent to some element of $C_0$ or it is not.   If $b_0$ is not $\sim$-equivalent to an element of $C_0$, then we may choose $b_1\in \cQ_0$ so that $b_1$ is not $\sim$-equivalent to any element of $C_1$ and, for $a\in Y_0$, we have $a\sqin b_0$ if and only if $f(a)\sqin b_1$. Then $f$ extends to an isomorphism from $(Y_0,\left< C_0b_0 \right>)$ to $(Y_1, \left< C_1b_1 \right>)$ sending $b_0$ to $b_1$. Now suppose $b_0$ is $\sim$-equivalent to some element of $C_0$. Let $E$ be the set of minimal elements of $[b_0]_\sim \cap C_0$. 
  For each $e \in E$ and $a \in Y_0$ with $a \sqin e$, one of the following cases holds:
    \begin{enumerate}[\hspace{5pt}$\ast$]
    \item $b_0 \sqcap e = \bot$ and $a \not\sqin b_0 \sqcap e$,
    \item $\bot \sqsubset b_0 \sqcap e \sqsubset e$ and $a \not\sqin b_0 \sqcap e$,
    \item $\bot \sqsubset b_0 \sqcap e \sqsubset e$ and $a \sqin b_0 \sqcap e$, or
    \item $b_0 \sqcap e = e$ and $a \sqin b_0 \sqcap e$.
    \end{enumerate}
   Let $E=\{c_0,\dots,c_{n-1}\}$. The list $\{f(c_0),\dots,f(c_{n-1})\}$ satisfies the requirements of Lemma~\ref{lem:sep-points}, so we can apply with $Y_1$ to get a family $\{d_x\}_{x \in Y_1}$ satisfying that
      \begin{enumerate}[\hspace{5pt}$\ast$]
      \item for any distinct $x,x' \in Y_1$, $d_x\sqcap d_{x'}=\bot$,
      \item for each $x \in Y_1$, $d_x \sqsubset f(c_i)$ for some $i<n$, 
      \item for each $i<n$, $\bigsqcup \{d_x : d_x \sqsubseteq f(c_i)\} \sqsubset f(c_i)$, and
      \item for each $x \in Y_1$, $x \sqin d_x$.
      \end{enumerate}
      Given the third condition, we can also clearly find, for each $i<n$, an element $h_i$ satisfying that $\bot \sqsubset h_i \sqsubset f(c_i)\sqcap \left( \bigsqcup \{d_x : d_x \sqsubseteq f(c_i)\} \right)\cop$. In particular, each $h_i$ satisfies $x \not \sqin h_i$ for all $x \in Y_1$.

      Now let
      \[
        b_1 = \bigsqcup_{i<n}\bigsqcup_{\begin{subarray}{1} a \in Y_0 \\ a\sqin c_i\end{subarray}}
        \begin{cases}
          \bot, & b_0 \sqcap c_i = \bot \\
          h_i, & \bot \sqsubset b_0 \sqcap c_i \sqsubset c_i\text{ and }a \not \sqin b_0 \sqcap c_i \\
          d_{f(a)}, & \bot \sqsubset b_0 \sqcap c_i \sqsubset c_i\text{ and }a \sqin b_0 \sqcap c_i \\
          f(c_i), & b_0\sqcap c_i = c_i 
        \end{cases}
      \]
      By construction, we now have that for any $e \in E$ and $a \in Y_0$ with $a \sqin e$,
    \begin{enumerate}[\hspace{5pt}$\ast$]
    \item $b_1 \sqcap f(e) = \bot$ and $f(a)\not\sqin b_1$ if and only if $b_0 \sqcap e = \bot$ and $a\not\sqin b_0$,
    \item $\bot \sqsubset b_1 \sqcap f(e) \sqsubset f(e)$ and $f(a) \not\sqin b_1$ if and only if $\bot \sqsubset b_0 \sqcap e \sqsubset e$ and $a \not\sqin b_0$,
    \item $\bot \sqsubset b_1 \sqcap f(e) \sqsubset f(e)$ and $f(a) \sqin b_1$ if and only if $\bot \sqsubset b_0 \sqcap e \sqsubset e$ and $a \sqin b_0$, and
    \item $b_1 \sqcap f(e) = f(e)$ and $f(a) \sqin b_1$ if and only if $b_0 \sqcap e = e$ and $a \sqin b_0$.
    \end{enumerate}
    These conditions are enough to imply that $f$ extends (uniquely) to an isomorphism from $(Y_0,\left< C_0b_0 \right>)$ to $(Y_1, \left< C_1b_1 \right>)$ sending $b_0$ to $b_1$.
  \end{proof}

  \begin{fact}\label{fact:doAb-QE}
    The reduct of $\Ti$ to the sort $R$ is the theory of ordered divisible Abelian groups with a nonzero constant. This theory has quantifier elimination.
  \end{fact}
  
  We can now prove Theorem \ref{thm:TiQE} ($\Ti$ has quantifier elimination).

\begin{proof}[\textnormal{\textbf{Proof of Theorem \ref{thm:TiQE}}}]
  For any finite tuples of variables $\bar{x}$ of sort $P$ and $\bar{y}$ of sort $Q$ and any type $p(\bar{x},\bar{y}) \in S_{\bar{x}\bar{y}}(T_{PQ})$, let $\tau_p(\bar{x},\bar{y})$ be some fixed quantifier-free formula isolating  $p$ (note that $\tau_p$ exists by Proposition \ref{prop:PQomega}). 
  
Quantifier elimination can be established by showing that for any tuples of variables $\bar{x}$ of sort $P$, $\bar{y}$ of sort $Q$, and $\bar{z}$ of sort $R$ and for any formula of the form
  \begin{equation*}
\psi(\bar{x},\bar{y},\bar{z}) \coloneqq  \tau_p(\bar{x},\bar{y})\wedge \varphi(\ell(t_0(\bar{y})),\ell(t_1(\bar{y})),\dots,\ell(t_{k-1}(\bar{y})),\bar{z}),
\end{equation*}
each of $\exists x_0 \psi$, $\exists y_0\psi$, and $\exists z_0\psi$ is equivalent to a quantifier free formula, where
\begin{enumerate}[\hspace{5pt}$\ast$]
\item  $p \in S_{\bar{x}\bar{y}}(T_{PQ})$  is some type (which, by Proposition \ref{prop:PQomega}, we may view as a complete quantifier-free type in $\cL_{PQ}$),
\item   $\varphi(\bar{w},\bar{z})$ is a quantifier-free atomic $\Lc_R$-formula, and
\item  $t_j(\bar{y})$ is an $\Lc_Q$-term for each $j < k$.
\end{enumerate}
(To see that this is sufficient, note that every quantifier-free formula is logically equivalent to a disjunction of formulas of the same form as $\psi(\bar{x},\bar{y},\bar{z})$.)

Let $\bar{\ell}$ be shorthand for the tuple $(\ell(t_0(\bar{y})),\ell(t_1(\bar{y})),\dots,\ell(t_{k-1}(\bar{y})))$. (So we will write $\varphi$ as $\varphi(\bar{\ell},\bar{z})$.) Let $\bar{x}_\ast$ be $\bar{x}$ without $x_0$, and let $\bar{y}_\ast$ and $\bar{z}_\ast$ be defined similarly. \medskip 

\noindent\textbf{Eliminating $P$ quantifiers:} Consider $\exists x_0 \psi(\bar{x},\bar{y},\bar{z})$. Since $\varphi$ does not actually contain $x_0$, $\exists x_0 \psi(\bar{x},\bar{y},\bar{z})$ is logically equivalent to $\varphi(\bar{\ell},\bar{z}) \wedge \exists x_0\tau_p(\bar{x},\bar{y})$. By quantifier elimination for $T_{PQ}$, this is equivalent to $\varphi(\bar{\ell},\bar{z})\wedge \tau_{p|{\bar{x}_*\bar{y}}}(\bar{x}_*,\bar{y})$, which is quantifier free.\medskip

\noindent\noindent\textbf{Reducing $Q$ quantifiers to $R$ quantifiers:} The type $p(\bar{x},\bar{y})$ fully determines the $\Lc_{PQ}$-isomorphism type of the substructure of $PQ$ generated by $\xbar\bar{y}$. Let $E$ be a set of terms $s(\bar{y}_\ast)$ corresponding to the minimal elements of the $\cL_Q$-substructure of $Q$ generated by $\bar{y}_\ast$. (These terms exist for any given type $p$ by Lemma \ref{lem:fo-gen}.) Without loss of generality, assume that $y_0$ is not $\top$ or $\bot$ (modulo $p$). Let $E_{\sim}$ be the set of terms in $E$ that are $\sim$-equivalent to $y_0$ (modulo $p$) and let $E_{\not\sim}$ be the set of those that are not $\sim$-equivalent to $y_0$ (modulo $p$). 
\medskip

\noindent\emph{Claim:} In $\psi$, we may assume that each $t_i(\bar{y})$ is either $s(\bar{y}_\ast)$ for some $s \in E_{\not\sim}$ or $s(\bar{y}_\ast)\sqcap y_0$ or $s(\bar{y}_\ast)\sqcap y_0\cop$ for some $s \in E_{\sim}$.

\noindent\emph{Proof:} For every term of the form $\ell(t_j(\bar{y}))$ either
\begin{enumerate}[\hspace{5pt}$\ast$]
\item there is some $E_0 \subseteq E_{\not\sim}$ such that $\ell(t_j(\bar{y})) = \sum_{s \in E_0}\ell(s(\bar{y}_\ast))$ or
\item there are some $E_1,E_2 \subseteq E_{\sim}$ such that 
\[
\ell(t_j(\bar{y})) = \sum_{s\in E_1}\ell(s(\bar{y}_\ast)\sqcap y_0) + \sum_{s\in E_2}\ell(s(\bar{y}_\ast)\sqcap y_0^c)
\]
\end{enumerate}
(modulo $p$). By substituting these expressions into $\varphi$, we get the claim. \claim
\medskip

\def\ellnotsim{\bar{\ell}^{\not\sim}}
\def\ellsimcap{\bar{\ell}^\sim_{\sqcap}}
\def\ellsimcapc{\bar{\ell}^{\sim}_{\textsf{c}}}
\def\ellsim{\bar{\ell}^{\sim}}

In light of the claim, we will split $\bar{\ell}$ into three subtuples $\ellnotsim$, $\ellsimcap$, and $\ellsimcapc$, where
\begin{enumerate}[\hspace{5pt}$\ast$]
\item $\ellnotsim$ is a list of all terms of the form $\ell(s(\bar{y}_\ast))$ for $s \in E_{\not\sim}$,
\item $\ellsimcap$ is a list of all terms of the form $\ell(s(\bar{y}_\ast)\sqcap y_0)$ for $s \in E_\sim$, and
\item $\ellsimcapc$ is a list of all terms of the form $\ell(s(\bar{y}_\ast)\sqcap y_0\cop)$ for $s \in E_\sim$ (in the same order).
\end{enumerate}
So now we will think of $\varphi$ as $\varphi(\ellnotsim,\ellsimcap,\ellsimcapc,\bar{z})$. It will also be useful to have the notation $\ellsim$ for a list of all terms of the form $\ell(s(\bar{y}_\ast))$ for $s \in E_\sim$ (also in the same order). Note that $\ellnotsim$ and $\ellsim$ do not contain the variable $y_0$. 

The core idea for reducing the quantifier $\exists y_0$ to some quantifiers in the $R$ sort is that once one fixes $\bar{a}$ in $P$ and $\bar{b}$ in $Q$ satisfying $p|_{\bar{x}\bar{y}_\ast}$ as well as some $\bar{c}$ in $R$, the existence of some $d$ such that $\psi(\bar{a},d\bar{b},\bar{c})$ holds depends \emph{only} on the existence of some values $\{m_s\}_{s \in E_{\sim}}$ for $\ell(s(\bar{b})\sqcap d)$ which are consistent with the existing measures of elements of $E$ as well as whatever requirements are imposed by the formula $\varphi$. In order to be consistent with the existing measures of elements of $E$ it is necessary and sufficient that for each $s \in E_{\sim}$
\begin{enumerate}[\hspace{5pt}$\ast$]
\item if $p$ requires that $s(\bar{y}_\ast)\sqcap y_0 = \bot$, then $m_s = 0$,
\item if $p$ requires that $\bot \sqsubset s(\bar{y}_\ast)\sqcap y_0\sqsubset s(\bar{y}_\ast)$, then $0 < m_s < \ell(s(\bar{y}_\ast))$, and
\item if $p$ requires that $s(\bar{y}_\ast)\sqcap y_0 = s(\bar{y}_\ast)$, then $m_s = \ell(s(\bar{y}_\ast))$.
\end{enumerate}
Let $\bar{m}$ and $\bar{u}$ be two new tuples of $R$-variables in the same order as $\ellsim$ (and so also in the same order as $\ellsimcap$ and $\ellsimcapc$). Rather than writing literal numerical indices for $\bar{m}$ and $\bar{u}$, we will write expressions such as $m_s$ to mean the variable in $\bar{m}$ in the same position as $\ell(s(\bar{y}_\ast))$ in $\ellsim$. We need a formula $\eta(\bar{m},\bar{u})$ expressing these compatibility requirements. So, to accomplish this, let
\[
\eta(\bar{m},\bar{u})\coloneqq \bigwedge_{s \in E_{\sim}} \begin{cases} m_s = 0, &   p(\bar{x},\bar{y}) \models s(\bar{y}_\ast) \sqcap y_0 = \bot  \\ 0 < m_s < u_s, & p(\bar{x},\bar{y}) \models \bot \sqsubset s(\bar{y}_\ast) \sqcap y_0 \sqsubset s(\bar{y}_\ast) \\ m_s = u_s, &  p(\bar{x},\bar{y}) \models s(\bar{y}_\ast) \sqcap y_0 = s(\bar{y}_\ast)\end{cases}.
\]
It's easy to see that the compatibility condition for $\bar{m}$ is equivalent to $\eta(\bar{m},\ellsim)$.

Now we can reduce the $\exists y_0$ quantifier to $\exists \bar{m}$. Consider the formula 
\[
\chi(\bar{x},\bar{y}_\ast,\bar{z}) \coloneqq \tau_{p|{\bar{x} \bar{y}_\ast}}(\bar{x},\bar{y}_\ast)\wedge \exists \bar{m}\left[ \eta(\bar{m},\ellsim) \wedge \varphi(\ellnotsim,\bar{m},\ellsim - \bar{m},\bar{z}) \right],
\]
where $\ellsim - \bar{m}$ is the tuple whose elements are  $\ell(s(\bar{y}_\ast)) - m_s$ for $s \in E_{\sim}$.
\medskip

\noindent\emph{Claim:} $\chi(\bar{x},\bar{y}_\ast,\bar{z})$ is logically equivalent to $\exists y_0 \psi(\bar{x},\bar{y},\bar{z})$.

\noindent\emph{Proof:} Fix $\bar{a}$ in $P$, $\bar{b}$ in $Q$, and $\bar{c}$ in $R$.

$\Leftarrow$: Suppose that there exists some $d$ in $Q$ such that $\psi(\bar{a},d\bar{b},\bar{c})$ holds. Clearly we have that $\bar{a}\bar{b}\models p|_{\bar{x}\bar{y}_\ast}$, so $\tau_{p|{\bar{x}\bar{y}_\ast}}(\bar{a},\bar{b})$ holds. By setting $m_s$ equal to $\ell(s(\bar{b})\sqcap d)$ for each $s \in E_{\sim}$, we get that the second part of $\chi(\bar{a},\bar{b},\bar{c})$ holds. (Noting that $\ell(s(\bar{b})\sqcap d\cop) = \ell(s(\bar{b}))-\ell(s(\bar{b})\sqcap d)$.) Therefore $\chi(\bar{a},\bar{b},\bar{c})$ holds.

$\Rightarrow$: Suppose that $\chi(\bar{a},\bar{b},\bar{c})$ holds. Let $\bar{e}$ be the tuple of elements of $R$ witnessing the $\exists \bar{m}$ quantifier. For each $s \in E_{\sim}$, choose $f_s$ such that 
\begin{enumerate}[\hspace{5pt}$\ast$]
\item if $p$ requires that $s(\bar{y}_\ast)\sqcap y_0 = \bot$, then $f_s = \bot$,
\item if $p$ requires that $\bot \sqsubset s(\bar{y}_\ast)\sqcap y_0\sqsubset s(\bar{y}_\ast)$, then $f_s$ is some element satisfying $\bot \sqsubset f_s \sqsubset s(\bar{b})$, $\ell(f_s) = e_s$, and $a_i \sqin f_s$ if and only if $p\models x_i \sqin y_0$ and
\item if $p$ requires that $s(\bar{y}_\ast)\sqcap y_0 = s(\bar{y}_\ast)$, then $f_s = s(\bar{b})$.
\end{enumerate}
This is alway possible for each $s \in E_{\sim}$ since $\eta(\bar{e},\ellsim(\bar{b}))$ holds and since the theory $\Ti$ entails that, for any $y\in Q$, any disjoint finite sets $P_{\sqin}$ and $P_{\not\sqin}$ of elements of $P$, and any $m \in (0,\ell(z))$, there exists a $z \in Q$ with $\bot \sqsubset z \sqsubset y$ and $\ell(z) = m$ such that $g \sqin z$ for every $g \in P_{\sqin}$ and $g \not \sqin z$ for every $g \in P_{\not\sqin}$. (Note that this is a family of first-order statements that hold in $\Mi$.)
Finally, let
\[
d = \bigsqcup_{s \in E_{\sim}} f_s.
\]
By quantifier elimination for $T_{PQ}$, we have that $\bar{a}d\bar{b}\models p$ (where $\bar{a}$ is assigned to $\bar{x}$ and $d\bar{b}$ is assigned to $\bar{y}$), so $\tau_p(\bar{a},d\bar{b})$ holds. Since $\varphi(\ellnotsim(\bar{b}),\bar{e},\ellsim(\bar{b}) - \bar{e},\bar{c})$ holds and since $\bar{e} = \ellsimcap(\bar{b},d)$, we have that $\psi(\bar{a},d\bar{b},\bar{c})$ holds, whence $\exists y_0 \psi(\bar{a},\bar{b},\bar{c})$ holds. \claim
\medskip

Therefore, once we can show that we can eliminate quantifiers of sort $R$, we will have shown that we can eliminate quantifiers of sort $Q$.\medskip

\noindent\textbf{Eliminating $R$ quantifiers:} The formula $\exists z_0 \psi(\bar{x},\bar{y},\bar{z})$ is logically equivalent to $\tau_p(\bar{x},\bar{y})\wedge \exists z_0 \varphi(\bar{\ell},\bar{z})$. By quantifier elimination for $T_R$, $\exists z_0 \varphi(\bar{v},\bar{z})$ (which is an $\Lc_R$-formula) is logically equivalent to some $\Lc_R$-formula $\theta(\bar{v},\bar{z}_\ast)$. Therefore we have that $\exists z_0 \psi(\bar{x},\bar{y},\bar{z})$ is logically equivalent to $\tau_p(\bar{x},\bar{y})\wedge \theta(\bar{\ell},\bar{z}_\ast)$, which is a quantifier-free formula.\medskip

Altogether, since we can reduce $Q$-quantifiers to $R$-quantifiers and eliminate $P$- and $R$-quantifiers, we can eliminate quantifiers in general, and $\Ti$ admits quantifier elimination.
\end{proof}

\begin{remark}
As an aside, quantifier elimination for $\Ti$ implies that the $R$ sort is stably embedded and that any types $p(x)$ in the $P$ sort and $r(y)$ in the $R$ sort are weakly orthogonal (i.e., $p(x)\cup r(y)$ axiomatizes a complete type).
\end{remark}

\subsection{Heirs of definable measures}\label{app:heir}

The purpose of this section is to  discuss definability for Keisler measures over small models, and  prove that such definable measures have definable global extensions.   This material is known in the folklore, especially from the perspective of continuous logic (see Remark \ref{rem:CLheir}). See also  \cite[Remark 2.7]{HPP} and \cite[Remark 3.20]{StarBour}. However, since a complete account does not appear in the literature, we take the opportunity in this appendix to provide complete definitions and some details of various proofs.

Let $T$ be a complete theory with monster model $\cU$. Throughout this section, we fix a model $M\preceq\cU$ and an arbitrary parameter set $A\seq M$.

\begin{definition}\label{def:babyfiber}
Given $\mu\in\kM_x(\cU)$ and an $\cL$-formula $\phi(x,y)$, define the map $f^\phi_\mu\colon M^y\to [0,1]$ such that $f^\phi_\mu(b)=\mu(\phi(x,b))$. 
\end{definition}

We view $M^x$ as a dense subset of $S_x(M)$ by identifying $a\in M^x$ with the isolated type $\tp(a/M)$. A  function $f\colon S_x(M)\to [0,1]$ is called  \emph{$A$-invariant} if $f(p)=f(q)$ whenever $p|_A=q|_A$.  

\begin{fact}\label{fact:defdefs2}
Given a measure $\mu\in\kM_x(M)$ and an $\cL$-formula $\phi(x,y)$, the following are equivalent.
\begin{enumerate}[$(i)$]
\item $f^\phi_\mu$ extends to an $A$-invariant continuous function from $S_y(M)$ to $[0,1]$.
\item For any $\epsilon>0$, there are $\cL_A$-formulas $\psi_1(y),\ldots,\psi_n(y)$ such that:
\begin{enumerate}[$\ast$]
\item $\{\psi_i(y):1\leq i\leq n\}$ is a partition of $M^y$, and
\item for any $1\leq i\leq n$, if $b_1,b_2\in \psi_i(M)$ then $|f^\phi_\mu(b_1)-f^\phi_\mu(b_2)|<\epsilon$.
\end{enumerate}
\item For any $\epsilon<\delta$ in $[0,1]$, there is an $\cL_A$-formula $\psi(y)$ such that 
\[
\{b\in M^y:f^\phi_\mu(b)\leq\epsilon\}\seq \psi(M)\seq\{b\in M^y:f^\phi_\mu(b)<\delta\}.
\]
\end{enumerate}
\end{fact}
\begin{proof}
This is a standard exercise in topology, which is similar to Fact \ref{fact:defdefs} and could be phrased entirely for functions on arbitrary Stone spaces. The direction requiring the most work is $(ii)\Rightarrow (i)$. So we note that this task can be simplified using \emph{Ta\u{\i}manov's Theorem}, which is a classical result that characterizes when a function on a dense subset of a space $X$ can be extended to a continuous function on $X$. See \cite{BlairECF} for details. 
\end{proof}

\begin{definition}\label{def:defsmall}
A measure $\mu\in\kM_x(M)$ is \textbf{$A$-definable} if, for any $\cL$-formula $\phi(x,y)$, the map $f^\phi_\mu$ satisfies the equivalent properties in Fact \ref{fact:defdefs2}. 
\end{definition}

The previous definition appears also in \cite[Definition 3.19]{StarBour} (using characterization $(ii)$ of Fact \ref{fact:defdefs2}).  
Note that this definition does not conflict with the formulation of definability for global measures. In particular, if  we take $M$ to be the monster $\cU$, then Fact \ref{fact:defdefs2} aligns with Fact \ref{fact:defdefs}. 

\begin{remark}
Let $f$ be a map from $M^x$ to a compact Hausdorff space $X$. Then $f$ is called \emph{$A$-definable} if for any closed $C\seq X$ and open $U\seq X$, with $C\seq U$, there is some $A$-definable set $D\seq M^x$ such that $f\inv(C)\seq D\seq f\inv(U)$. 
 In particular, condition $(ii)$ of Fact \ref{fact:defdefs2} is equivalent to $A$-definability of $f^{\phi}_\mu$.
\end{remark}

\begin{remark}\label{rem:heirQE}
Suppose $T$ has quantifier elimination in the language $\cL$. Then $\mu\in\kM_x(M)$ is $A$-definable if and only if the equivalent properties of Fact \ref{fact:defdefs2} hold for any $\cL$-formula $\phi(x,y)$, which is a conjunction of atomic and negated atomic $\cL$-formulas. Indeed, if every formula $\phi(x,y)$ of the described form satisfies condition $(i)$ of Fact \ref{fact:defdefs2}, then so does every quantifier-free $\cL$-formula by inclusion-exclusion. 
\end{remark}

The main result of this section says that definable global ``heirs" of definable measures exist and are unique. The uniqueness aspect is a consequence of the following observation, which also makes explicit the analogy to heirs of types.
 
 \begin{remark}\label{rem:heir}
Fix $\mu\in\kM_x(M)$, and suppose $\hat{\mu}\in\kM_x(\cU)$ is an $A$-definable extension of $\mu$. Then for any $\cL$-formula $\phi(x,y)$, $F^{\phi}_{\hat{\mu},M}$ is an $A$-invariant continuous extension of $f^\phi_\mu$.  Therefore $\mu$ is $A$-definable. Moreover, for any $\cL$-formula $\phi(x,y)$ and any open set $U\seq[0,1]$, if $b\in\cU^y$ and $\hat{\mu}(\phi(x,b))\in U$, then there is some $\psi(y)\in \tp(b/A)$ such that $\hat{\mu}(\phi(x,c))\in U$ for any $c\in\psi(\cU)$ (so, in particular, $\mu(\phi(x,c))\in U$ for some $c\in M^y$). It  follows that $\hat{\mu}$ is the unique $A$-definable global extension of $\mu$.
\end{remark}

\begin{theorem}\label{thm:heir}
Suppose $\mu\in\kM_x(M)$ is $A$-definable. Then $\mu$ has a unique $A$-definable  extension $\hat{\mu}\in \kM_x(\cU)$. 
\end{theorem}
\begin{proof}
By Remark \ref{rem:heir}, it suffices to just show that $\hat{\mu}$ exists. Given an $\cL$-formula $\phi(x,y)$ and $b\in M^y$, define 
\[
    \hat{\mu}(\varphi(x,b)) = \hat{f}_{\mu}^{\varphi}(\tp(b/M)), 
\]
where $\hat{f}_{\mu}^\phi$ is the continuous $A$-invariant extension of $f^\phi_\mu$ to $S_y(M)$.  Assuming $\hat{\mu}$ is a well-defined Keisler measure, it follows that $\hat{\mu}$ is an $A$-invariant global extension of $\mu$, and $F_{\hat{\mu},M}^{\phi}=\hat{f}_\mu^\phi$ for any $\cL$-formula $\phi(x,y)$. In particular, $\hat{\mu}$ is $A$-definable.

To show that $\hat{\mu}$ is well-defined, we need to verify that any inconsistent $\cL_{\cU}$-formula has measure $0$, and that finite additivity holds. So first fix an $\cL$-formula $\phi(x,y)$ and $b\in M^y$ such that $\phi(x,b)$ is inconsistent. Then $\tp(b/M)$ is in the clopen set $C\coloneqq [\forall x\neg\phi(x,y)]\seq S_y(M)$. Since $f^\phi_\mu$ is identically $0$ on $C\cap M^y$, which is dense in $C$, we have  $\hat{f}^\phi_\mu(\tp(b/M))=0$, as desired. Next, to verify finite additivity, fix $\cL$-formulas $\phi(x,y)$ and $\psi(x,z)$, and let $\theta(x;y,z)$ and $\chi(x;y,z)$ denote $\phi(x,y)\vee\psi(x,z)$ and $\phi(x,y)\wedge\psi(x,z)$, respectively. We need to show that if $bc\in \cU^{yz}$ then
\[
\hat{f}^{\theta}_\mu(\tp(bc/M))=\hat{f}^{\phi}_{\mu}(\tp(b/M))+\hat{f}^{\psi}_\mu(\tp(c/M))-\hat{f}^{\chi}_\mu(\tp(bc/M)).
\]
Note that the previous equation holds for any $bc\in M^{yz}$ since $\mu$ is a Keisler measure and the $\hat{f}_\mu$-maps extend the $f_\mu$-maps. So fix some $bc\in\cU^{yz}$, and let $(b_ic_i)_{i \in I}$ be a net of points in $M^{yz}$ such that $\lim_{i \in I} b_ic_i = \tp(bc/M)$ (recall that we identify points from $M$ with realized types over $M$). Then we have the following computation:
\begin{align*}
\hat{f}_{\mu}^{\theta}(\tp(bc/M)) &=  \lim_{i\in I}\hat{f}_{\mu}^{\theta}(b_{i}c_{i}) = \lim_{i\in I}\left[\hat{f}^{\phi}_{\mu}(b_i)+\hat{f}^{\psi}_\mu(c_i)-\hat{f}^{\chi}_\mu(b_ic_i)\right]\\
&= \lim_{i\in I}\hat{f}^{\phi}_{\mu}(b_i)+\lim_{i\in I}\hat{f}^{\psi}_\mu(c_i)-\lim_{i\in I}\hat{f}^{\chi}_\mu(b_ic_i)\\
&=\hat{f}^{\phi}_{\mu}(\tp(b/M))+\hat{f}^{\psi}_\mu(\tp(c/M))-\hat{f}^{\chi}_\mu(\tp(bc/M)).\qedhere
\end{align*} 
\end{proof}

\begin{remark}\label{rem:CLheir}
  The proof of Theorem \ref{thm:heir} can be understood in an abstract way with continuous logic. For the case of types, one can show that a definable type $p$ over a model $M$ has a canonical definable global extension by arguing that the theory of $M$ `knows' that the defining schema of $p$ gives a complete type. This is essentially the same as the argument we have presented here. A continuous real valued function on a type space is the same thing as a definable predicate in the sense of continuous logic (or a formula if one has a broad enough notion of formula). Given a definable measure $\mu$ over a model $M$, the theory of $M$ `knows' that the defining schema of $\mu$ gives a measure, and so it follows that the same schema gives a global measure.
\end{remark}

Finally, we take a brief moment to note the existence of `coheirs' for measures over small models. This is not used in any part of the paper, but it is thematically relevant to the aims of this part of the appendix.

\begin{proposition} 
For any $\mu \in \mathfrak{M}_{x}(M)$ there is some $\hat{\mu} \in \mathfrak{M}_{x}(\mathcal{U})$ such that $\hat{\mu}|_{M} = \mu$ and $\hat{\mu}$ is finitely satisfiable in $M$. 
\end{proposition}
\begin{proof} 
Define $\kM_x(\cU,M)\coloneqq \{\mu\in\kM_x(\cU):\text{$\mu$ is finitely satisfiable in $M$}\}$. Let $\rho$ be the restriction of $\rho^x_M\colon\kM_x(\cU)\to \kM_x(M)$ to $\kM_x(\cU,M)$. We want to show that $\rho$ is surjective. Note that $\kM_x(\cU,M)$ is closed in $\kM_x(\cU)$, and so $\rho$ is a continuous map between compact Hausdorff spaces. Thus $\im(\rho)$ is closed in $\kM_x(M)$. Moreover, $\im(\rho)$ is a dense subset of $\kM_x(M)$ (consider the image of $\{\Av(\overline{a}): \overline{a} \in (M^{x})^{< \omega}\}$). Therefore $\rho$ is surjective. 
\end{proof}

\end{appendices}

\end{document}